\documentclass[reqno]{amsart}
\usepackage{amsmath,amsthm,amscd,amssymb,enumitem,latexsym,upref,hyperref,xfrac}

\newcommand*{\mailto}[1]{\href{mailto:#1}{\nolinkurl{#1}}}

\newtheorem{theorem}{Theorem}[section]

\newtheorem{corollary}[theorem]{Corollary}
\newtheorem{proposition}[theorem]{Proposition}

\theoremstyle{definition}

\newtheorem{remark}[theorem]{Remark}

\DeclareMathOperator*{\esssup}{ess\,sup}
\def\ol{\overline}
\def\ul{\underline}
\newcommand{\nn}{\nonumber}

\def\qq{q}
\def\cT{\mathcal{T}}
\def\cW{\mathcal{W}}
\def\psii{u}

\def\N{\mathbb{N}}
\newcommand{\hG}{\hat{\Gamma}}
\newcommand{\Om}{\Omega}
\def\R{\mathbb{R}}
\def\L2{L^2(\Omega)}

\def\LT{L^2(0,T;L^2(\Omega))}
\def\LiT{L^{\infty}(0,T;L^2(\Omega))}
\def\LqT{L^{q+1}(0,T;L^{q+1}(\Omega))}
\def\Chl{C^{\Omega}_{H^1,L^4}}
\def\Cwl{C^{\Omega}_{W^{1,q+1},L^{\infty}}}
\def\qCwl{C^{\Omega}_{W^{1,q+1},L^{\infty}}}

\numberwithin{equation}{section}

\allowdisplaybreaks

\begin{document}

\title[Local existence results]{Local existence results for the Westervelt equation with nonlinear damping and Neumann as well as absorbing boundary conditions}

\author[V. Nikoli\' c]{\vspace{1mm} Vanja Nikoli\' c}
\address{Insitut f\"ur Mathematik\\ Universit\"at Klagenfurt\\
Universit\"atsstra{\ss}e 65-57\\ 9020 Klagenfurt am W\"orthersee\\ Austria}
\email{\mailto{vanja.nikolic@aau.at}}
\begin{abstract} \bigskip \vspace{2mm}
We investigate the Westervelt equation with several versions of nonlinear damping and lower order damping terms and Neumann as well as  absorbing boundary conditions. We prove local in time existence of weak solutions under the assumption that the initial and boundary data are sufficiently small. Additionally, we prove local well-posedness in the case of spatially varying $L^{\infty}$ coefficients, a model relevant in high intensity focused ultrasound (HIFU) applications. 
\end{abstract}
\thanks{Research supported by the Austrian Science Fund (FWF): P24970}
\maketitle
\medskip
\section{Introduction}
High intensity focused ultrasound (HIFU) is crucial in many medical and industrial applications including lithotripsy, thermotherapy, ultrasound cleaning or welding and sonochemistry. Widely used mathematical model for nonlinear wave propagation is the Westervelt equation, which can either be written in terms of the acoustic pressure $p$
\begin{align} \label{press_form}
(1-2kp)p_{tt} -c^2 \Delta p - b\Delta p_t = 2k(p_t)^2,
\end{align}
or in terms of the acoustic velocity potential $\psi$
\begin{equation}\label{poten_form}
(1-2\tilde{k}\psi_t)\psi_{tt} -c^2 \Delta \psi - b\Delta\psi_t 
= 0 ,
\end{equation}
with $\varrho \psi_t = p$. Here, $c$ denotes the speed and 
$b$ the diffusivity of sound, $k = \beta_a / \lambda$, $\beta_a = 1 + B/(2A)$, $B/A$ represents the parameter of nonlinearity, $\varrho$ is the mass density, $\lambda=\varrho c^2$ is the bulk modulus and $\tilde{k}=\varrho k$. For a detailed derivation of \eqref{press_form} and \eqref{poten_form} we refer the reader to \cite{HamiltonBlackstock}, \cite{manfred}, \cite{Westervelt}.\\
\indent Well-posedness and exponential decay of small and $H^2-$spatially regular solutions is established for the Westervelt equation with homogeneous~\cite{KL08} and inhomogeneous~\cite{KLV10} Dirichlet 
and Neumann~\cite{Neumann} boundary conditions as well as with boundary instead of interior damping~\cite{K10}.   \\
\indent A significant task in the analysis of the Westervelt equation is avoiding degeneracy of the coefficient $1-2kp$ for the second time derivative $p_{tt}$ in \eqref{press_form} and, similarly, of the term $1-2\tilde{k}\psi_t$ in the formulation \eqref{poten_form}.  At the same time, in applications the existence of spatially less regular solutions is important, e.g. in the coupling of acoustic with acoustic or elastic regions with different material parameters. In \cite{brunn}, Brunnhuber, Kaltenbacher and Radu treated this issue by introducing nonlinear damping terms to the Westervelt equation and considering the following equations
\begin{align}
&(1-2ku)u_{tt}-c^2\Delta u-b\,\text{div}\Bigl(((1-\delta) +\delta|\nabla u_t|^{\qq-1})\nabla u_t\Bigr) 
=2k(u_t)^2, \label{Westervelt1} \\
& (1-2ku)u_{tt}-c^2 \text{div} (\nabla u +\varepsilon |\nabla u|^{q-1}\nabla u)
-b \Delta u_t =2k(u_t)^2, \label{Westervelt2} \\
&u_{tt}-\frac{c^2}{1-2\tilde{k}u_t}\Delta u
-b\,\text{div}\Bigl(((1-\delta) +\delta|\nabla u_t|^{\qq-1})\nabla u_t\Bigr)=0, \label{Westervelt3} 
\end{align}
with homogeneous Dirichlet boundary data. First two equations are derived from the Westervelt equation in the acoustic pressure formulation \eqref{press_form}, while the third equation comes from the acoustic potential formulation \eqref{poten_form} (with the notation changed to $p \rightarrow u$, $\psi \rightarrow u$).  Added nonlinear damping terms make obtaining $L^{\infty}(0,T;L^{\infty}(\Om))$ estimate on $u$ ($u_t$) possible, without the need to estimate $\Delta u$ ($\Delta u_t$) and thus refraining from too high regularity. \\
\indent The central aim of the present paper is to investigate this relaxation of regularity by nonlinear damping, but equipped with practically relevant absorbing and Neumann boundary data. This is motivated by many applications of high intensity focused ultrasound where the need for more realistic boundary conditions is evident. E.g. in lithotripsy one faces the problem of a physically unbounded domain, as typical in acoustics, which should be truncated for numerical computations.  Absorbing boundary conditions are then used to avoid reflections on the artificial boundary $\hG$ of the computational domain. \\
\indent Ultrasound excitation, e.g. by piezoelectric transducers, can be modeled by Neumann boundary conditions on the rest of the boundary $\Gamma=\partial \Om \setminus \hG$. \\
\indent In our case, the design of the nonlinear absorbing and inhomogeneous Neumann boundary conditions is influenced by the presence of the nonlinear strong damping in the equations.
We will study initial boundary value problems of the following type:
\begin{align}
\begin{split}\label{W1_beta}
\begin{cases}
(1-2ku)u_{tt}-c^2\Delta u-b\,\text{div}\Bigl(((1-\delta) +\delta|\nabla u_t|^{q-1})\nabla u_t\Bigr) +\beta u_t \\
 =2k(u_t)^2 
 \, \text{ in } \Omega \times (0,T],
\vspace{2mm} \\
c^2 \frac{\partial u}{\partial n}+b((1-\delta)+\delta|\nabla u_t|^{q-1})\frac{\partial u_t}{\partial n}=g  \ \ \text{on} \ \Gamma \times (0,T],
\vspace{2mm} \\
\alpha u_t+c^2 \frac{\partial u}{\partial n}+b((1-\delta)+\delta|\nabla u_t|^{q-1})\frac{\partial u_t}{\partial n}=0  \ \ \text{on} \ \hat{\Gamma}\times (0,T],
\vspace{2mm} \\
(u,u_t)=(u_0, u_1) \ \ \text{on} \ \overline{\Om}\times \{t=0\}
,\\
\end{cases}
\end{split}\\ \nn \\
\begin{split} \label{W1_gamma}
\begin{cases}
(1-2ku)u_{tt}-c^2\Delta u-b\,\text{div}\Bigl(((1-\delta) +\delta|\nabla u_t|^{\qq-1})\nabla u_t\Bigr) +\gamma |u_t|^{q-1} u_t  \\
=2k(u_t)^2 
 \, \text{ in } \Omega \times (0,T],
\vspace{2mm} \\
c^2 \frac{\partial u}{\partial n}+b((1-\delta)+\delta|\nabla u_t|^{q-1})\frac{\partial u_t}{\partial n}=g  \ \ \text{on} \ \Gamma \times (0,T],
\vspace{2mm} \\
\alpha u_t+c^2 \frac{\partial u}{\partial n}+b((1-\delta)+\delta|\nabla u_t|^{q-1})\frac{\partial u_t}{\partial n}=0  \ \ \text{on} \ \hat{\Gamma}\times (0,T],
\vspace{2mm} \\
(u,u_t)=(u_0, u_1) \ \ \text{on} \ \overline{\Om}\times \{t=0\}
,\\
\end{cases}
\end{split}\\ \nn \\
\begin{split} \label{W2_beta}
\begin{cases}
(1-2ku)u_{tt}-c^2 \text{div} (\nabla u +\varepsilon |\nabla u|^{q-1}\nabla u)
-b \Delta u_t+\beta u_t  \\
=2k(u_t)^2 \, \text{ in } \Omega \times (0,T],
\vspace{2mm} \\
c^2 \frac{\partial u}{\partial n}+c^2 \varepsilon|\nabla u|^{q-1}\frac{\partial u}{\partial n}+b\frac{\partial u_t}{\partial n}=g  \ \ \text{on} \ \Gamma \times (0,T],
\vspace{2mm} \\
\alpha u_t+c^2 \frac{\partial u}{\partial n}+c^2 \varepsilon|\nabla u|^{q-1}\frac{\partial u}{\partial n}+b\frac{\partial u_t}{\partial n}=0  \ \ \text{on} \ \hat{\Gamma} \times (0,T],
\vspace{2mm} \\
(u,u_t)=(u_0, u_1) \ \ \text{on} \ \overline{\Om}\times \{t=0\}
,\\
\end{cases}
\end{split}\\ \nn  \\
\begin{split}\label{W3_gamma}
\begin{cases}
u_{tt}-\frac{c^2}{1-2\tilde{k}u_t}\Delta u
-b\,\text{div}\Bigl(((1-\delta) +\delta|\nabla u_t|^{\qq-1})\nabla u_t\Bigr) +\gamma |u_t|^{q-1}u_t  \\
 =0 \, \text{ in } \Omega \times (0,T],
\vspace{2mm} \\
\frac{c^2}{1-2\tilde{k}u_t} \frac{\partial u}{\partial n}+b((1-\delta)+\delta|\nabla u_t|^{q-1})\frac{\partial u_t}{\partial n}=g   \ \text{on} \ \Gamma \times (0,T],
\vspace{2mm} \\
\alpha u_t+\frac{c^2}{1-2\tilde{k}u_t} \frac{\partial u}{\partial n}+b((1-\delta)+\delta|\nabla u_t|^{q-1})\frac{\partial u_t}{\partial n}=0  \ \text{on} \ \hat{\Gamma}\times (0,T],
\vspace{2mm} \\
(u,u_t)=(u_0, u_1) \ \ \text{on} \ \overline{\Om}\times \{t=0\}.
\end{cases}
\end{split}
\end{align}
\indent Note that in the case of $b=0$, $\alpha=c$ and $\tilde{k}=0$ the absorbing conditions prescribed in \eqref{W1_beta}-\eqref{W3_gamma} would reduce to the standard linear absorbing boundary conditions of the form $u_t+c\frac{\partial u}{\partial n}=0$. \\
\indent In the equations, we assume that the parameters $\beta$ and $\gamma$ are nonnegative; the case $\beta=\gamma=0$ reduces them to \eqref{Westervelt1}-\eqref{Westervelt2}. Another task of the present paper is to investigate possible introduction of these lower order \textit{linear} and \textit{nonlinear} damping terms to the equations \eqref{Westervelt1}-\eqref{Westervelt2}, this becomes beneficial when deriving energy estimates. \\
\indent Additionally, in the context of HIFU devices based on the acoustic lens immersed in a fluid medium, a problem of Westervelt's equation coupled with other equations or with jumping coefficients arises. We will treat acoustic-acoustic coupling which can be modeled by Westervelt's equation in the pressure formulation with spatially varying coefficients (see \cite{BambergerGlowinskiTran} for the linear case and \cite{brunn} for the nonlinear case with homogeneous Dirichlet boundary conditions):
\begin{equation}\label{W1_coupled}
\begin{cases}
\frac{1}{\lambda(x)}(1-2k(x)u) u_{tt}-\text{div}(\frac{1}{\varrho(x)} \nabla u)-\text{div}\Bigl(b(x)( 
((1-\delta(x)) +\delta(x)|\nabla u_t|^{q-1})\nabla u_t\Bigr) 
\vspace{2mm} \\
=\frac{2k(x)}{\lambda(x)}(u_t)^2 \, \text{ in } \Omega \times (0,T],
\vspace{2mm} \\
\frac{1}{\varrho(x)}\frac{\partial u}{\partial n}+b(x)((1-\delta(x))+\delta(x)|\nabla u_t|^{q-1})\frac{\partial u_t}{\partial n}=g  \ \ \text{on} \ \Gamma \times (0,T],
\vspace{2mm} \\
\alpha (x) u_t+\frac{1}{\varrho(x)} \frac{\partial u}{\partial n}+b(x)((1-\delta(x))+\delta(x)|\nabla u_t|^{q-1})\frac{\partial u_t}{\partial n}=0  \ \ \text{on} \ \hat{\Gamma} \times (0,T] ,
\vspace{2mm} \\
(u,u_t)=(u_0, u_1) \ \ \text{on} \ \overline{\Om}\times \{t=0\}. 
\end{cases}
\end{equation}
\subsection{Notations and Preliminaries} \label{notations} 
\noindent 
We assume $\Om \subset \mathbb{R}^d$, $d \in \{1,2,3\}$ to be an open, connected, bounded set  with Lipschitz boundary; $\partial \Om$ is assumed to be a disjoint union of $\Gamma$ and $\hG$. We denote by $n$ the outward unit normal vector. \\
\indent We will study the problems with strong damping $b>0$ and with $c^2>0$, $\delta \in (0,1)$, $\varepsilon>0$ and $k, \tilde{k} \in \mathbb{R}$. Our results will hold for $\alpha$ assumed to be nonegative; the case $\alpha=0$ reduces \eqref{W1_beta}-\eqref{W3_gamma} to problems with only Neuman boundary conditions.\\
\indent Note that, in general, we will assume that $q \geq 1$, but this condition will have to be strenghtened at several instances to assure well-posedness of \eqref{W1_beta}, \eqref{W1_gamma} and existence results for \eqref{W2_beta} and \eqref{W3_gamma}. We will often make use of the continuous embeddings
\begin{align*}
& H^1(\Om) \hookrightarrow L^4(\Om), \ \text{with the norm} \ C^{\Om}_{H^1,L^4},\ \text{and} \\
& W^{1,q+1}(\Om) \hookrightarrow L^{\infty}(\Om), \ \text{with the norm} \ C^{\Om}_{W^{1,q+1},L^{\infty}},
\end{align*}
with the latter being valid for $q+1>d$. In Section \ref{section_W1} and \ref{section_W3} we will need to employ the embedding $L^{q+1}(\Omega) \hookrightarrow L^4(\Omega)$, which holds true for $q \geq 3$. \\
\indent We denote with $C_1^{tr}$ the norm of the trace mapping $$Tr:W^{1,q+1}(\Om) \rightarrow W^{1-\frac{1}{q+1},q+1}(\Gamma),$$ and with $C_2^{tr}$ the norm of the trace mapping $tr: H^1(\Om) \rightarrow H^{-1/2}(\Gamma)$ (with $C_1^{tr}=C_2^{tr}$ for $q=1$). \\
\indent Throughout the paper we assume $t \in [0,T]$, where $T$ is a finite time horizon.
\subsection{Outline of the paper}
The rest of the paper is organized as follows. Subsection \ref{inequalities} contains the derivation of $L^{\infty}$-bounds on $u$ and $u_t$ as well as several useful inequalities that will be employed in the paper. \\
\indent In Section \ref{section_W1}, we start by looking at a linearized version of \eqref{W1_beta} and \eqref{W1_gamma} with $\beta=\gamma=0$, with nonlinearity appearing only through damping, and show local well-posedness. Then we discuss linearized versions of \eqref{W1_beta} and \eqref{W1_gamma} with $\beta, \gamma>0$. By employing the result for the linearized version we proceed to prove local well-posedness for  \eqref{W1_beta} and \eqref{W1_gamma}.\\
\indent Section \ref{section_W1coupling} deals with the short time well-posedness of the acoustic-acoustic coupling modeled by \eqref{W1_coupled}. \\
\indent In Section \ref{section_W2} and \ref{section_W3} we consider \eqref{W2_beta} and \eqref{W3_gamma}, respectively. We again begin by investigating the linearized versions of the problems at hand for $\beta=0$ and $\gamma=0$ respectively, and continue with introducing lower order damping terms and the proof of local existence of solutions.
\subsection{Inequalities} \label{inequalities} \noindent In the case of problems with inhomogeneous Neumann boundary data it is often necessary to employ Poincar\' e's inequality valid for functions in $W^{1,q+1}(\Om)$. We recall such inequality (cf. Theorem 12.23, \cite{Leoni}), namely that there exists a constant $C_P>0$ depending on $q$ and $\Om$ such that
\begin{align} \label{Poincare}
 |\varphi-\frac{1}{|\Om|}\int_{\Om} \varphi \, dx|_{L^{q+1}(\Om)} \leq C_P|\nabla \varphi|_{L^{q+1}(\Om)},
\end{align} 
for all $\varphi\in W^{1,q+1}(\Om)$. \\
\indent The nonlinear damping term appearing in the equations \eqref{W1_beta}-\eqref{W3_gamma} will enable us to avoid degeneracy of the coefficients $1-2ku$ and $1-2ku_t$ by deriving $L^{\infty}$ estimates on $u$ and $u_t$. From \eqref{Poincare} we can obtain
\begin{align} \label{Poincare_est1}
|u(t)|_{W^{1,q+1}(\Om)} \leq
 (1+C_P) |\nabla u(t)|_{L^{q+1}(\Omega)}+C_1^{\Om}\Big|\int_{\Om} u(t) \, dx \Bigr|,
\end{align}
and by replacing $u$ with $u_t$ also
\begin{align} \label{Poincare_est2}
|u_t(t)|_{W^{1,q+1}(\Om)} \leq
 (1+C_P) |\nabla u_t(t)|_{L^{q+1}(\Omega)}+C_2^{\Om}|u_t(t)|_{L^2(\Om)} ,
\end{align}
where $C_1^{\Om}=|\Om|^{-\frac{q}{q+1}}$ and $C_2^{\Om}=|\Om|^{-\frac{q-1}{2(q+1)}}$. \\
From \eqref{Poincare_est2}, by making use of the embedding $W^{1,q+1}(\Omega) \hookrightarrow L^{\infty}(\Om)$, $q>d-1$, we obtain an $L^{\infty}$ estimate on $u_t$
\begin{equation} \label{W3Poincare}
\begin{split}
\|u_t\|_{L^{\infty}(0,T;L^{\infty}(\Om))}
\leq& \, C_{W^{1,\qq+1},L^{\infty}}^\Omega
\Bigl [(1+C_P)\|\nabla u_t\|_{L^{\infty}(0,T;L^{q+1}(\Om))} \\
&\quad \quad \quad \quad \quad \quad +C_2^{\Om} \|u_{t}\|_{L^{\infty}(0,T;L^2(\Om))}  \Bigr],
\end{split}
\end{equation}
which will be used to avoid degeneracy of the factor $1-2ku_t$ in the problem \eqref{W3_gamma}.\\
Employing  \eqref{Poincare_est1} and the estimate
\begin{align} \label{time_ineq}
\|u_t\|^2_{L^2(0,T;L^2(\Om))} \leq T \|u_t\|^2_{L^{\infty}(0,T;L^2(\Om)},
\end{align} we can get an $L^{\infty}$ estimate on $u$
\begin{equation} \label{W2Poincare}
\begin{split}
\|u\|_{L^{\infty}(0,T;L^{\infty}(\Om))}
\leq& \, C_{W^{1,\qq+1},L^{\infty}}^\Omega
\Bigl[(1+C_P)|\nabla u(t)|_{L^{q+1}(\Omega)} \\
&\quad \quad \quad \quad \quad \ +C_1^{\Om}\Big|\int_{\Om} (u_0+\int_0^t u_t(s) \, ds) \, dx\Big| \Bigr]  \\
\leq& \, C_{W^{1,\qq+1},L^{\infty}}^\Omega
\Bigl[(1+C_P)\|\nabla u\|_{L^{\infty}(0,T;L^{q+1}(\Om))}  \\
&\quad \quad \quad \quad \quad \ +C_1^{\Om}|u_0|_{L^{1}(\Om)} +C_2^{\Om}\sqrt{T}\|u_t\|_{L^2(0,T;L^2(\Om))}  \Bigr]  \\
\leq& \, C_{W^{1,\qq+1},L^{\infty}}^\Omega
\Bigl[(1+C_P)\|\nabla u\|_{L^{\infty}(0,T;L^{q+1}(\Om))}  \\
&\quad \quad \quad \quad \quad \ +C_1^{\Om}|u_0|_{L^{1}(\Om)} +C_2^{\Om}T\|u_t\|_{L^{\infty}(0,T;L^2(\Om))}  \Bigr],
\end{split}
\end{equation}
which we will apply when investigating \eqref{W2_beta}.\\
From \eqref{Poincare_est1} we can as well obtain
\begin{equation*}
\begin{split}
|u(t)|_{L^{\infty}(\Om)} \leq& \, C_{W^{1,\qq+1},L^{\infty}}^\Omega
\Bigl[(1+C_P)|\nabla u_0 + \int_0^t \nabla u_t(s) \, ds|_{L^{q+1}(\Om)}  \\  
&\quad \quad \quad \quad \quad \ +C_1^{\Om} |u_0|_{L^{1}(\Om)}+C_2^{\Om}\int_0^t |u_t(t)|_{L^2(\Om)} \, ds \Bigr]  \\
\leq& \, C_{W^{1,\qq+1},L^{\infty}}^\Omega
\Bigl[(1+C_P)(|\nabla u_0|_{L^{q+1}(\Om)}+(t^q\int_0^t |\nabla u_t|^{q+1}_{L^{q+1}(\Om)}\, ds)^{1/{q+1}})  \\  
&\quad \quad \quad \quad \quad \  + C_1^{\Om} |u_0|_{L^{1}(\Om)}+C_2^{\Om}\int_0^t|u_t(t)|_{L^2(\Om)}\, ds \Bigr],
\end{split}
\end{equation*}
which leads to the estimate
\begin{equation} \label{W1Poincare}
\begin{split}
\|u\|_{L^{\infty}(0,T;L^{\infty}(\Om))} \leq& \, C_{W^{1,\qq+1},L^{\infty}}^\Omega
\Bigl[(1+C_P)(|\nabla u_0|_{L^{q+1}(\Om)}  \\
&\quad \quad \quad \quad \quad \ +T^{\frac{q}{q+1}}\|\nabla u_t\|_{L^{q+1}(0,T;L^{q+1}(\Om))}) \\ &\quad \quad \quad \quad \quad \ +C_1^{\Om} |u_0|_{L^{1}(\Om)} +C_2^{\Om}T\|u_t\|_{L^{\infty}(0,T;L^2(\Om))} \Bigr], 
\end{split}
\end{equation}
that will be employed when dealing with the possible degeneracy of the coefficient $1-2ku$ in \eqref{W1_beta} and \eqref{W1_gamma}.
\noindent We will also frequently make use of Young's inequality in the form
\begin{align} \label{Young}
\quad  \quad \quad ab \leq \varepsilon a^s+C(\varepsilon, s) b^{\frac{s}{s-1}}  \quad (a, b >0, \  \varepsilon >0, \ 1<s< \infty),
\end{align}
with $C(\varepsilon, s)=(s-1)s^{\frac{s}{s-1}}\varepsilon ^{-\frac{1}{1-s}}$. \\
\noindent When dealing with the $q$-Laplace damping term in the equations, the inequality (cf. \cite{lindqvist}) 
\begin{align} \label{pLaplace_ineq}
\langle |b|^{q-1}b-|a|^{q-1}a,b-a \rangle \geq 0, \ \ a, b \in \mathbb{R}^d, 
\end{align}
valid for all $q$, will be of use as well. 
 \thispagestyle{plain}
\section{Westervelt's equation in the formulation \eqref{W1_beta} and \eqref{W1_gamma}} \label{section_W1}
\noindent We will begin by looking at the problems \eqref{W1_beta} and \eqref{W1_gamma} with $\beta=\gamma=0$:
\begin{align}\label{W1}
\begin{cases}
(1-2ku)u_{tt}-c^2\Delta u-b\,\text{div}\Bigl(((1-\delta) +\delta|\nabla u_t|^{\qq-1})\nabla u_t\Bigr) \vspace{2mm} \\=2k(u_t)^2 
 \, \text{ in } \Omega \times (0,T],
\vspace{2mm} \\
c^2 \frac{\partial u}{\partial n}+b((1-\delta)+\delta|\nabla u_t|^{q-1})\frac{\partial u_t}{\partial n}=g  \ \ \text{on} \ \Gamma \times (0,T],
\vspace{2mm} \\
\alpha u_t+c^2 \frac{\partial u}{\partial n}+b((1-\delta)+\delta|\nabla u_t|^{q-1})\frac{\partial u_t}{\partial n}=0  \ \ \text{on} \ \hat{\Gamma}\times (0,T],
\vspace{2mm} \\
(u,u_t)=(u_0, u_1) \ \ \text{on} \ \overline{\Om}\times \{t=0\},
\end{cases}
\end{align}
\noindent Following the approach in \cite{brunn}, we will first consider the equation where nonlinearity appears only in the damping term
\begin{equation}\label{W1lin}
\begin{cases}
a u_{tt}-c^2\Delta u-b\,\text{div}\Bigl( 
((1-\delta) +\delta|\nabla u_t|^{q-1})\nabla u_t\Bigr)+fu_t \vspace{2mm} \\ =0 \, \text{ in } \Omega \times (0,T]
\vspace{2mm} \\
c^2 \frac{\partial u}{\partial n}+b((1-\delta)+\delta|\nabla u_t|^{q-1})\frac{\partial u_t}{\partial n}=g  \ \ \text{on} \ \Gamma \times (0,T],
\vspace{2mm} \\
\alpha u_t+c^2 \frac{\partial u}{\partial n}+b((1-\delta)+\delta|\nabla u_t|^{q-1})\frac{\partial u_t}{\partial n}=0  \ \ \text{on} \ \hat{\Gamma} \times (0,T],
\vspace{2mm} \\
(u,u_t)=(u_0, u_1) \ \ \text{on} \ \overline{\Om}\times \{t=0\},
\\
\end{cases}
\end{equation}
and prove local well-posedness.
\begin{proposition} \label{prop:W1lin}
Let $T>0$, $c^2$, $b>0$, $\alpha \geq 0$, $\delta \in (0,1)$, $q \geq 1$ and assume that
\begin{enumerate}
\item[(i)]
\begin{itemize} 
\item 
$a\in L^{\infty}(0,T;L^{\infty}(\Omega)) 
$, $a_t\in L^{\infty}(0,T;L^2(\Omega))$,
$0<\underline{a}\leq a(t,x)\leq \overline{a}$,
\item
$f\in L^{\infty}(0,T;L^2(\Omega))$, 
\item
 $g \in L^{\frac{q+1}{q}}(0,T;W^{-\frac{q}{q+1},\frac{q+1}{q}}(\Gamma))$,
\item
$ u_0\in H^1(\Omega)$, $u_1\in L^2(\Omega)$,
\end{itemize}
\end{enumerate}
with 
\begin{align}
\|f-\frac12 a_t\|_{L^{\infty}(0,T;L^2(\Omega))}\leq\hat{b}
< \min \Bigl \{\frac{b(1-\delta)}{2(C_{H^1,L^4}^\Omega)^2}, \frac{\underline{a}}{4T(C_{H^1,L^4}^\Omega)^2} \Bigr \}\,. \label{W1_f_a}
\end{align}
\noindent Then \eqref{W1lin} has a weak solution 
\begin{equation} \label{W1_Xtilde}
\begin{split}
u\in \tilde{X}:=\{v:& \, v \in C(0,T;H^1(\Om)) \cap C^1(0,T;L^2(\Om))  \\
& \wedge  v_t \in  L^{q+1}(0,T;W^{1,q+1}(\Om))\},
\end{split}
\end{equation}
which is unique and satisfies the energy estimate
\begin{align} \label{est1}
&\Bigl[\frac{\underline{a}}{4}-\hat{b}(\Chl)^2T-\epsilon_0 \Bigr]\|u_t\|^2_{\LiT}
+ \frac{c^2}{4} \|\nabla u\|_{\LiT}^2 \nn  \\
&+ \Bigl[\frac{b(1-\delta)}{2}-\hat{b}(\Chl)^2 \Bigr]\|\nabla u_t\|^2_{\LT}+\frac{\alpha}{2}\| u_{t}\|^2_{L^2(0,T;L^2(\hG))} \nn \\
&+\Bigl[\frac{b\delta}{2}-\epsilon_1 \Bigr]\|\nabla u_t\|^{q+1}_{\LqT}  \\
\leq& \, \frac{\overline{a}}{2}|u_1|^2_{L^2(\Om)}+\frac{c^2}{2}|\nabla u_0|^2_{L^2(\Om)}+\frac{1}{4\epsilon_0}(C_1^{tr} C^{\Om}_2)^2\|g\|^2_{L^{1}(0,T;W^{-\frac{q}{q+1},\frac{q+1}{q}}(\Gamma))} \nn \\ &+C(\epsilon_1,q+1)(C_1^{tr}(1+C_P))^{\frac{q+1}{q}}\|g\|^{\frac{q+1}{q}}_{L^{\frac{q+1}{q}}(0,T;W^{-\frac{q}{q+1},\frac{q+1}{q}}(\Gamma))}, \nn
\end{align}
for some constants 
\begin{equation} \label{W1_epsilon}
0< \epsilon_0 < \frac{\underline{a}}{4}-\hat{b}(C^{\Om}_{H^1,L^4})^2T, \ 0< \epsilon_1<\frac{b \delta}{2}.
\end{equation}
\noindent If, in addition to (i), 
\begin{enumerate}
\item[(ii)]
\begin{itemize}
\item 
$f \in L^{\infty}(0,T;H^1(\Om)), \ \|f\|_{L^{\infty}(0,T;H^1(\Om))} \leq \tilde{b}$,
\item
 $g \in L^{\infty}(0,T;W^{-\frac{q}{q+1},\frac{q+1}{q}}(\Gamma))$,  $g_t \in L^{\frac{q+1}{q}}(0,T;W^{-\frac{q}{q+1},\frac{q+1}{q}}(\Gamma))$,
\item
$ u_1\in W^{1,q+1}(\Omega)$,
\end{itemize} 
\end{enumerate} then 
\begin{align} \label{W1_X}
u\in X:=~& C^1(0,T;W^{1,q+1}(\Om)) \cap H^2(0,T;L^2(\Om)),
\end{align}
and satisfies the energy estimate
\begin{equation}  \label{est2}
\begin{split}
&\mu \frac{\underline{a}-\tau}{2} \|u_{tt}\|^2_{L^2(0,T;L^2(\Omega))}+\mu [\frac{b(1-\delta)}{4}-\sigma]\|\nabla u_{t}\|^2_{L^{\infty}(0,T;L^2(\Om))}  \\
& +\Bigl[\frac{\underline{a}}{4}-(C^{\Om}_{H^1,L^4})^2\hat{b}T- \epsilon_0(\mu+1) -\mu \frac{1}{2\tau}(C^{\Om}_{H^1,L^4})^4\tilde{b}^2T \Bigr]\|u_{t}\|^2_{L^{\infty}(0,T;L^2(\Om))}  \\
& + \Bigl[\frac{b (1-\delta)}{2}-(C^{\Omega}_{H^1,L^4})^2\hat{b}-\mu (\frac{1}{2\tau}(C^{\Omega}_{H^1,L^4})^4\tilde{b}^2+ c^2) \Bigr]\|\nabla u_{t}\|^2_{L^2(0,T;L^2(\Omega))}   \\
& +\frac{c^2}{4}(1-\mu \frac{c^2}{\sigma}) \|\nabla u\|^2_{L^{\infty}(0,T;L^2(\Om))} +\mu [\frac{b\delta}{2(q+1)}-\eta]\|\nabla u_{t}\|^{q+1}_{L^{\infty}(0,T;L^{q+1}(\Om))}
 \\
&+[\frac{b\delta}{2}-\epsilon_1(\mu+1)] \|\nabla u_{t}\|^{q+1}_{L^{q+1}(0,T;L^{q+1}(\Om))} +\frac{\alpha}{2}\|u_{t}\|^2_{L^2(0,T;L^2(\hG))} \\
&   +\mu \frac{\alpha}{4}\|u_{t}\|^2_{L^{\infty}(0,T;L^2(\hG))} \\
\leq& \,  \,\overline{C}\Bigl( C_{\Gamma}(g)+|u_{1}|^2_{H^1(\Om)}+|\nabla u_{0}|^2_{L^2(\Omega)}+|u_{1}|^{q+1}_{W^{1,q+1}(\Omega)}+ |u_{1}|^2_{L^2(\hG)} \Bigr),
\end{split}
\end{equation}
for some sufficiently small constants $\mu, \sigma, \tau, \eta>0$, some large enough $\overline{C}>0$, and
\begin{equation} \label{Cg}
\begin{split}
C_{\Gamma}(g)=& \displaystyle \sum_{s=0}^1  \|\frac{d^s}{d t^s} g\|^2_{L^{1}(0,T;W^{-\frac{q}{q+1},\frac{q+1}{q}}(\Gamma))}+\displaystyle \sum_{s=0}^1 \|\frac{d^s}{d t^s} g\|^{\frac{q+1}{q}}_{L^{\frac{q+1}{q}}(0,T;W^{-\frac{q}{q+1},\frac{q+1}{q}}(\Gamma))}  \\
& +\|g\|^{2}_{L^{\infty}(0,T;W^{-\frac{q}{q+1},\frac{q+1}{q}}(\Gamma))}+\|g\|^{\frac{q+1}{q}}_{L^{\infty}(0,T;W^{-\frac{q}{q+1},\frac{q+1}{q}}(\Gamma))}.
\end{split}
\end{equation}
\end{proposition}
\begin{proof}
The weak form of \eqref{W1lin} is given as
\begin{align}\label{W1linweak}
&\int_\Omega \Bigl\{a u_{tt} w + c^2 \nabla u \cdot \nabla w 
+ b\Bigl((1-\delta) +\delta |\nabla u_t|^{q-1}
\Bigr)\nabla u_t \cdot \nabla w \Bigr\} \, dx 
 +\alpha \int_{\hat{\Gamma}} u_t w \, dx \nn \\
=& \ -\int_0^t \int_\Omega fu_t w\, dx+\int_{\Gamma}g w \, dx, \ \forall w\in W^{1,q+1}(\Omega), 
\end{align}
with initial conditions $(u_{0},u_{1})$.\\
\indent We will use the standard Galerkin method (see for instance Section 7.2, \cite{evans} for the case of second-order linear hyperbolic equations and Section 2, \cite{brunn} for the problem \eqref{W1lin} with homogeneous Dirichlet boundary data), where we will first construct approximations of the solution, and then by obtaining energy estimates guarantee weak convergence of these approximations. \\
\noindent \textbf{1. Smooth approximation of $a$, $f$, and $g$.}  Let us first introduce sequences $(a_k)_{k\in\N}$, $(f_k)_{k\in\N}$ and $(g_k)_{k\in\N}$ which represent smooth in time approximations of $a$, $f$, and $g$:
\begin{itemize}
\item 
$(a_k)_{k\in\N}\subseteq C^\infty([0,T]\times\overline{\Omega})\cap W^{1,\infty}(0,T;L^2(\Om))$,\\
$a_k\to a$ in $L^{\infty}(0,T;L^{\infty}(\Omega))$, $a_{k,t} \to a_t$ in $L^{\infty}(0,T;L^2(\Om))$, \\
$0<\underline{a}\leq a_k(t,x)\leq \overline{a}$,
\item 
$(f_k)_{k\in\N}\subseteq C^\infty((0,T)\times\Omega)$, 
$f_k\to f$ in $L^{\infty}(0,T;L^2(\Omega))$,
\item 
$(g_k)_{k\in\N} \subseteq C^\infty(0,T;W^{-\frac{q}{q+1},\frac{q+1}{q}}(\Gamma))$, 
$g_k\to g$ in 
$L^{\frac{q+1}{q}}(0,T;W^{-\frac{q}{q+1},\frac{q+1}{q}}(\Gamma))$,
\item
$\|f_k-\frac12a_{k,t}\|_{L^{\infty}(0,T;L^2(\Omega))}\leq\hat{b}$,
\end{itemize}
and, for fixed $k\in\N$, prove that there exists a solution $u^{(k)}$ of
\begin{equation}\label{W1linweak_k}
\begin{split}
\int_\Omega \Bigl\{& \, a_ku^{(k)}_{tt} w + c^2 \nabla u^{(k)} \cdot \nabla w 
+ b\Bigl((1-\delta) +\delta |\nabla u^{(k)}_t|^{q-1}
\Bigr)\nabla u^{(k)}_t \cdot \nabla w \Bigr\}  \, dx  \\
& +\alpha \int_{\hG}u^{(k)}_t w \, dx 
= -\int_\Omega f_ku^{(k)}_t w\, dx+\int_{\Gamma} g w \, dx, \ \quad \forall w\in W^{1,q+1}(\Omega),
\end{split}
\end{equation}
with initial conditions $(u_{0},u_{1})$. \\
\noindent \textbf{(a) Galerkin approximations.} We start by proving existence and uniqueness of a solution for a finite-dimensional approximation of \eqref{W1linweak_k}. We choose smooth functions $w_m=w_m(x)$, $m\in\N$ such that
\begin{align*}
&\{w_m\}_{m\in\N}\text{ is an orthonormal basis of }L^2_{\tilde{a}_k}(\Omega),\\ 
&\{w_m\}_{m\in\N}\text{ is a basis of } W^{1,q+1}(\Omega), \\
&\{w_{m|_{\hG}}\}_{m\in\N}\text{ is an orthonormal basis of } L^2(\hG), 
\end{align*}
where $L^2_{\tilde{a}_k}$ is the weighted $L^2$-space based on the inner product $\langle f,g\rangle_{L^2_{\tilde{a}_k}(\Omega)} := \int_\Omega \tilde{a}_k fg \, dx$,
with $\tilde{a}_k=\frac{1}{T}\int_0^Ta_k(t)\, dt$. \\
Next, we construct a sequence of finite dimensional subspaces $V_n$ of $L^2_{\tilde{a}_k} (\Omega) \cap W^{1,q+1}(\Omega)$, 
\begin{align*}
V_n=\text{span}\{w_1,w_2,\ldots,w_n\}.
\end{align*}  
Clearly, $V_n \subseteq V_{n+1}$, 
$V_n\subseteq L^2_{\tilde{a}_k} (\Omega) \cap W^{1,q+1}(\Omega)$ and $\overline{\bigcup_{n\in\N} V_n}=W^{1,q+1}(\Omega)$.\\
Let $(u_{0,n})_{n\in\N}$, $(u_{1,n})_{n\in\N}$ be sequences such that 
\begin{itemize}
\item 
$u_{0,n}\in  V_n$, 
$ u_{0,n}\to  u_0$ in $H^1(\Omega)$,
\item $u_{1,n}\in  V_n$, 
$u_{1,n}\to u_1$ in $L^2(\Omega)$.
\end{itemize}
We can now consider a sequence of discretized versions of \eqref{W1linweak_k},
\begin{align}\label{W1linweakdis}
\int_\Omega \Bigl\{& a_k u_{n,tt}^{(k)} w_n + c^2 \nabla u_n^{(k)} \cdot \nabla w_n 
+ b\Bigl((1-\delta) +\delta|\nabla u_{n,t}^{(k)}|^{q-1}\Bigr)\nabla u_{n,t}^{(k)} \cdot \nabla w_n \Bigr\} \, dx  \nn \\
& +\alpha\int_{\hG}u_{n,t}^{(k)}  w_n \, dx 
=-\int_\Omega f_k u_{n,t}^{(k)} w_n \, dx+\int_{\Gamma} g_k w_n \, dx, \quad \forall w_n\in V_n,
\end{align}
with $u_n^{(k)}(t)\in V_n$ and initial conditions $(u_{0,n},u_{1,n})$. For each $n\in\N$, we face an initial value problem for a second order system of ordinary differential equations with coefficients and right hand side that are $C^{\infty}$ functions of t. According to standard existence theory for ordinary differential equations (cf. \cite{Teschl}), there exists a unique solution $u_n^{(k)}\in C^\infty(0,\tilde{T},V_n)$ of \eqref{W1linweakdis} for some $\tilde{T}\leq T$ sufficiently small. By employing the uniform energy estimates obtained below, we can conclude that $\tilde{T}=T$. \\
\noindent \textbf{(b) Lower energy estimate.} Testing \eqref{W1linweakdis} with $w_n=u_{n,t}^{(k)}(t) \in V_n$ and integrating with respect to time results in
\begin{equation} \label{W1lin_1est}
\begin{split}
&\frac{1}{2}\left[\int_\Omega a_k\left(u_{n,t}^{(k)}\right)^2\, dx
+ c^2 |\nabla u_n^{(k)}|_{L^2(\Omega)}^2 \right]_0^t+\alpha \int_0^t \int_{\hG} | u_{n,t}^{(k)}|^2 \, dx \, ds  \\
&\quad \ + b \int_0^t \int_{\Omega} \Bigl((1-\delta) +\delta|\nabla u_{n,t}^{(k)}|^{q-1}\Bigr)|\nabla u_{n,t}^{(k)}|^2 \, dx \, ds  \\
=& \ 
- \int_0^t  \int_\Omega (f_k- \frac{1}{2}a_{k,t}) \left(u_{n,t}^{(k)}\right)^2 \, dx \, ds
+ \int_0^t \int_{\Gamma} g_k u_{n,t}^{(k)} \, dx \, ds  \\
\leq& \, \|f_k- \frac{1}{2}a_{k,t}\|_{L^{\infty}(0,T;L^2(\Om))}\int_0^t |u^{(k)}_{n,t}|^2_{L^4(\Om)} \, ds + \int_0^t \int_{\Gamma} g_k u_{n,t}^{(k)} \, dx \, ds .
\end{split}
\end{equation}
For estimating the boundary integral appearing on the right side, we will make use of \eqref{Poincare_est2} to obtain
 \begin{align} \label{W1_inequality1}
 \int_0^t \int_{\Gamma} g_k u_{n,t}^{(k)} \, dx \, ds \leq& \, \int_0^t |u_{n,t}^{(k)}(s)|_{W^{1-\frac{1}{q+1},q+1}(\Gamma)} |g_k(s)|_{W^{-\frac{q}{q+1},\frac{q+1}{q}}(\Gamma)} \, ds \nn \\
\leq &\, C_1^{tr} \int_0^t |u_{n,t}^{(k)}(s)|_{W^{1,q+1}(\Om)} |g_k(s)|_{W^{-\frac{q}{q+1},\frac{q+1}{q}}(\Gamma)} \, ds  \nn \\
\leq& \, C_1^{tr} \int_0^t \Bigl[ (1+C_P)|\nabla u_{n,t}^{(k)}(s)|_{L^{q+1}(\Om)} \\
&\quad \quad \quad \quad +C^{\Om}_2 |u_{n,t}^{(k)}(s)|_{L^2(\Om)}\Bigr]|g_k(s)|_{W^{-\frac{q}{q+1},\frac{q+1}{q}}(\Gamma)} \, ds  \nn \\
\leq& \, \epsilon_1 \|\nabla u_{n,t}^{(k)}\|^{q+1}_{L^{q+1}(0,T;L^{q+1}(\Om))} +\epsilon_0 \|u_{n,t}^{(k)}\|^2_{L^{\infty}(0,T;L^2(\Om))}  \nn \\
& +C(\epsilon_1,q+1)(C_1^{tr}(1+C_P))^{\frac{q+1}{q}}\|g_k\|^{\frac{q+1}{q}}_{L^{\frac{q+1}{q}}(0,T;W^{-\frac{q}{q+1},\frac{q+1}{q}}(\Gamma))}  \nn \\
&+\frac{1}{4\epsilon_0}(C_1^{tr} C^{\Om}_2)^2\|g_k\|^2_{L^{1}(0,T;W^{-\frac{q}{q+1},\frac{q+1}{q}}(\Gamma))}, \nn
\end{align}
with $\epsilon_0, \epsilon_1>0$. By taking the essential supremum with respect to $t$ in \eqref{W1lin_1est} and employing the embedding $H^1(\Om) \hookrightarrow L^4(\Om)$, as well as the inequality \eqref{time_ineq}, we obtain the estimate
\begin{align} \label{lower_estimate}
&\Bigl[\frac{\underline{a}}{4}-\hat{b}(\Chl)^2T-\epsilon_0 \Bigr]\|u_{n,t}^{(k)}\|^2_{\LiT}
+ \frac{c^2}{4} \|\nabla u_{n}^{(k)}\|_{\LiT}^2 \nn \\
&+ \Bigl[\frac{b(1-\delta)}{2}-\hat{b}(\Chl)^2 \Bigr]\|\nabla u_{n,t}^{(k)}\|^2_{\LT} +\frac{\alpha}{2}\| u_{n,t}^{(k)}\|^2_{L^2(0,T;L^2(\hG))} \nn  \\
&+\Bigl[\frac{b\delta}{2}-\epsilon_1\Bigr]\|\nabla u_{n,t}^{(k)}\|^{q+1}_{\LqT}  \\
\leq& \ C(\epsilon_1,q+1)(C_1^{tr}(1+C_P))^{\frac{q+1}{q}}\|g_k\|^{\frac{q+1}{q}}_{L^{\frac{q+1}{q}}(0,T;W^{-\frac{q}{q+1},\frac{q+1}{q}}(\Gamma))}+\frac{\overline{a}}{2}|u_{1,n}^{(k)}|^2_{L^2(\Om)} \nn \\
& +\frac{1}{4\epsilon_0}(C_1^{tr} C^{\Om}_2)^2\|g_k\|^2_{L^{1}(0,T;W^{-\frac{q}{q+1},\frac{q+1}{q}}(\Gamma))}+\frac{c^2}{2}|\nabla u_{0,n}^{(k)}|^2_{L^2(\Om)}.\nn
\end{align}
We choose $\epsilon_0,  \epsilon_1$ small enough
\begin{align} \label{epsilon_choice}
& 0< \epsilon_0 < \frac{\underline{a}}{4}-\hat{b}(C^{\Om}_{H^1,L^4})^2T, \ 0< \epsilon_1<\frac{b \delta}{2},
\end{align}
so that coefficients appearing in the estimate remain positive. As by assumption 
$g_k \in L^{\frac{q+1}{q}}(0,T;W^{-\frac{q}{q+1},\frac{q+1}{q}}(\Gamma))$, 
we conclude that the sequence of Galerkin approximations $\big(u_n^{(k)}\big)_{n\in\N}$ is bounded in the Banach space
\begin{align*}
\tilde{X}:=~&\{v: v \in C(0,T;H^1(\Om)) \cap C^1(0,T;L^2(\Om)) \wedge  v_t \in  L^{q+1}(0,T;W^{1,q+1}(\Om))\}.
\end{align*}
It follows from \eqref{lower_estimate} that 
\begin{align} 
& \big( u_{n,t}^{(k)}\big)_{n \in\N}\text{ is uniformly bounded in } L^2(0,T; L^2(\Omega)), \label{bound1} \\
& \big(\nabla u_{n,t}^{(k)}\big)_{n \in\N}\text{ is uniformly bounded in } L^{q+1}(0,T; L^{q+1}(\Omega)),  \label{bound2}  \\
& |\nabla u_{n,t}^{(k)}|^{q-1} \nabla u_{n,t}^{(k)} \ \text{ is uniformly bounded in } L^{\frac{q+1}{q}}(0,T;L^{\frac{q+1}{q}}(\Omega)), \ \text{and} \label{bound3}  \\
& \big( u_{n,t|\hG}^{(k)}\big)_{\in\N}\text{ is uniformly bounded in } L^2(0,T; L^2(\hG)), \label{bound4}
\end{align}
which are all reflexive Banach spaces. \\
\noindent \textbf{(c) Convergence of Galerkin approximations.} Due to \eqref{bound1}-\eqref{bound4} there exists a weakly convergent subsequence of $(u_n^{(k)})_{n\in\N}$, which we still denote $(u_n^{(k)})_{n\in\N}$, and a $u^{(k)}$ such that 
\begin{align}
& u_{n,t}^{(k)} \rightharpoonup  u^{(k)}_t \mbox{ in }L^2(0,T; L^2(\Omega)),  \label{weakconv1} \\
&\nabla u_{n,t}^{(k)} \rightharpoonup \nabla u^{(k)}_t \mbox{ in }L^{q+1}(0,T; L^{q+1}(\Omega)), \label{weakconv2} \\
& |\nabla u_{n,t}^{(k)}|^{q-1} \nabla u_{n,t}^{(k)}
\rightharpoonup |\nabla u_t^{(k)}|^{q-1} \nabla u_t^{(k)} 
\mbox{ in }L^{\frac{q+1}{q}} (0,T;L^{\frac{q+1}{q}}(\Omega)), \label{weakconv3} \\
& u^{(k)}_{n,t| \hG}\rightharpoonup u^{(k)}_{t| \hG} \mbox{ in } L^2(0,T; L^2(\hG)). \label{weakconv4}
\end{align}
\noindent Our task next is to prove that the weak limit $u^{(k)}$ solves \eqref{W1linweak_k}. Fix $k,m\in\N$ and let $\phi_m\in C^\infty(0,T,V_m)\subset L^{q+1}(0,T;W^{1,q+1}(\Omega))$ with $\phi_m(T)=0$. For any $n\geq m$, by $V_m\subseteq V_n$ we have 
\begin{align} \label{convntoinfty_W1}
&\int_0^T \int_\Omega \Bigl\{ a_ku_{tt}^{(k)} \phi_m 
+ c^2 \nabla u^{(k)}\cdot  \nabla \phi_m
+ b\Bigl((1-\delta) +\delta|\nabla u_t^{(k)}|^{q-1}\Bigr)\nabla u_t^{(k)}\cdot \nabla \phi_m  \nn \\ 
&\quad + f_k u_t^{(k)} \phi_m \Bigr\} \, dx\, ds+\alpha \int_0^T \int_{\hG}u_t^{(k)} \phi_m \, dx \, ds- \int_0^T \int_{\Gamma} g_k \phi_m \, dx \, ds \nn \\
=& \,
-\int_0^T\int_\Omega [u_{t}^{(k)}-u_{n,t}^{(k)}] \Bigl(a_k  \phi_m\Bigr)_t \, dx \, ds
-\int_\Omega [u_1-u_{1,n}] a_k(0)  \phi_m(0) \, dx \, ds  \nn \\
&+ c^2 \int_0^T\int_\Omega [\nabla u^{(k)}-\nabla u_n^{(k)}]\cdot \nabla \phi_m \, dx \, ds  \\ 
&
+ \int_0^T\int_\Omega [u_t^{(k)} -u_{n,t}^{(k)}]f_k \phi_m \, dx\, ds
+ b (1-\delta) \int_0^T\int_\Omega  [\nabla u_t^{(k)}-\nabla u_{n,t}^{(k)}]\cdot \nabla \phi_m \, dx\, ds  \nn \\ 
&+ b\delta \int_0^T\int_\Omega [|\nabla u_t^{(k)}|^{q-1}\nabla u_t^{(k)}-|\nabla u_{n,t}^{(k)}|^{q-1}\nabla u_{n,t}^{(k)}]\cdot \nabla \phi_m \, dx\, ds \nn \\
& +\alpha \int_0^T \int_{\hG}[u_t^{(k)} -u_{n,t}^{(k)}] \phi_m \, dx \, ds \  \to 0 \mbox{ as }n\to \infty, \nn
\end{align}
due to \eqref{weakconv1}-\eqref{weakconv4}.
Since $\bigcup_{m\in\N} V_m$ is dense in $W^{1,q+1}(\Omega)$, $u^{(k)}$ indeed solves \eqref{W1linweak_k}. By testing the problem \eqref{W1linweak_k} with $u^{(k)}_t$ and proceding as in 1.(b) we can conclude that this weak limit satisfies the estimate \eqref{lower_estimate} with $u_n^{(k)}$ replaced by $u^{(k)}$.  \\ 
\noindent \textbf{2. $ \pmb{k} \mathbf{\to} \pmb{\infty} $.} Owing to the previous conclusion, we can find a weakly convergent subsequence of $(u^{(k)})$, which
we again denote $(u^{(k)})$, and $u\in \tilde{X}$ such that
\begin{align}
& u_{t}^{(k)} \rightharpoonup  u_t \mbox{ in }L^2(0,T; L^2(\Omega)),  \label{weakconv5} \\
&\nabla u_{t}^{(k)} \rightharpoonup \nabla u_t \mbox{ in }L^{q+1}(0,T; L^{q+1}(\Omega)), \label{weakconv6} \\
& |\nabla u_{t}^{(k)}|^{q-1} \nabla u_{n,t}^{(k)}
\rightharpoonup |\nabla u_t|^{q-1} \nabla u_t 
\mbox{ in }L^{\frac{q+1}{q}} (0,T;L^{\frac{q+1}{q}}(\Omega)), \label{weakconv7} \\
& u^{(k)}_{t| \hG}\rightharpoonup u_{t| \hG} \mbox{ in } L^2(0,T; L^2(\hG)). \label{weakconv8}
\end{align}
\noindent It remains to show that $u$ satisfies \eqref{W1linweak}. For all $w\in C^\infty(0,T;W^{1,q+1}(\Omega))$ with $w(T)=0$ we have
\begin{equation*}
\begin{split}
&\int_0^t \int_\Omega \Bigl\{ a u_{tt} w + c^2 \nabla u \cdot \nabla w
+ b\Bigl((1-\delta) +\delta|\nabla u_t|^{q-1}\Bigr)\nabla u_t \cdot \nabla w +f u_t w \Bigr\} \, dx\, ds \\
& +\alpha \int_0^t \int_{\hG}u_t w \, dx \, ds-\int_0^t \int_{\Gamma} gw \, dx \, ds \\
=& \, -\int_0^t \int_\Omega [u_{t}-u_{t}^{(k)}] \Bigl(aw\Bigr)_t \, dx \, ds 
- \int_0^t \int_\Omega u_{t}^{(k)} \Bigl([a-a_k]w\Bigr)_t \, dx \, ds\\
&- \int_\Omega u_{1} \Bigl([a(0)-a_k(0)]w(0)\Bigr) \, dx 
+ \int_0^t \int_\Omega c^2 [\nabla u - \nabla u^{(k)})] \cdot \nabla w \, dx \, ds \\
&+ \int_0^t \int_\Omega  b(1-\delta) [\nabla u_t - \nabla u_t^{(k)}]\cdot  \nabla w \, dx \, ds \\
&+ \int_0^t \int_\Omega b\delta [|\nabla u_t|^{q-1}\nabla u_t - |\nabla u_t^{(k)}|^{q-1}\nabla u_t^{(k)}]\cdot \nabla w \, dx \, ds \\
&+ \int_0^t \int_\Omega [u_t - u_t^{(k)}] f w \, dx \, ds + \int_0^t \int_\Omega [f-f_k] u_t^{(k)}  w \, dx\, ds \\
&+\alpha \int_0^t \int_{\hG} [u_t - u_t^{(k)}]w \, dx \, ds- \int_0^t \int_\Gamma [g - g_k] w \, dx \, ds  \to 0 \text{ as } k \to \infty, \end{split}
\end{equation*}
since we demanded that $a_k \to a$ in $L^{\infty}(0,T;L^{\infty}(\Omega))$, $a_{k,t} \to a_t$ in $L^{\infty}(0,T;L^2(\Om))$, $f_k \to f$ in $L^{\infty}(0,T;L^2(\Omega))$ and $g_k \to g$ in 
$L^{\frac{q+1}{q}}(0,T;W^{-\frac{q}{q+1},\frac{q+1}{q}}(\Gamma))$. This relation proves that $u$ solves \eqref{W1linweak}. The weak limit then satisfies the estimate \eqref{est1}. \\
\noindent \textbf{3. Uniqueness.} To confirm uniqueness, note that the difference $\hat{u}=u^1-u^2$ between any two weak solutions $u^1,u^2$ of \eqref{W1lin} is a weak solution of the problem
\begin{equation}\label{W1_uniq}
\begin{cases}
a\hat{u}_{tt}-c^2\Delta \hat{u}-b(1-\delta)\Delta \hat{u}_t -b\delta \,\text{div} \Bigl(|\nabla u^1_t|^{q-1}\nabla u^1_t-|\nabla u^2_t|^{q-1}\nabla u^2_t \Bigr)\\
+f\hat{u}_t\ = \ 0,
\vspace{2mm} \\
c^2 \frac{\partial \hat{u}}{\partial n}+b(1-\delta)\frac{\partial \hat{u}_t}{\partial n}+b\delta(|\nabla u^1_t|^{q-1}\frac{\partial u^1_t}{\partial n}-|\nabla u^2_t|^{q-1}\frac{\partial u^2_t}{\partial n})=0  \ \ \text{on} \ \Gamma,
\vspace{2mm} \\
\alpha \hat{u}_t+c^2 \frac{\partial \hat{u}}{\partial n}+b(1-\delta)\frac{\partial \hat{u}_t}{\partial n}+b\delta(|\nabla u^1_t|^{q-1}\frac{\partial u^1_t}{\partial n}-|\nabla u^2_t|^{q-1}\frac{\partial u^2_t}{\partial n})=0 \ \text{on} \ \hat{\Gamma},\vspace{2mm} \\
(\hat{u},\hat{u}_t)|_{t=0}=(0,0).
\end{cases}
\end{equation}
Multiplication of \eqref{W1_uniq} by $\hat{u}_t$ yields
\begin{align*}
&\frac{1}{2}\left[\int_\Omega a(\hat{u}_t)^2\, dx
+ c^2 |\nabla \hat{u}|_{L^2(\Omega)}^2 \right]_0^t +b(1-\delta)\int_0^t \int_{\Om} |\nabla \hat{u}_t|^2 \, dx \, ds
+\alpha \int_0^t |\hat{u}_t|^2_{L^2(\hG)} \, ds \\
& \quad + \int_0^t  \int_\Omega (f-\frac12a_t) (\hat{u}_t)^2 \, dx \, ds \ \leq \ 0,
\end{align*}
since due to the inequality \eqref{pLaplace_ineq} we have 
\begin{equation}\label{wsig_nonneg}
 b\delta \int_0^t \int_{\Om} (|\nabla u^1_t|^{q-1}\nabla u^1_t-|\nabla u^2_t|^{q-1}\nabla u^2_t) \cdot \nabla \hat{u}_t \, dx \, ds \geq 0.
\end{equation}
From here we conclude that $\hat{u}_t=0$ and $\nabla \hat{u}=0$ almost everywhere, which results in the solution being unique up to an additive constant. The initial condition $\hat{u}|_{t=0}=0$ provides us with uniqueness.\\
\noindent \textbf{4. Higher energy estimate.} To obtain higher order estimate \eqref{est2}, we will test \eqref{W1linweakdis} with $w_n=u_{n,tt}^{(k)}(t)\in V_n$ and then combine the result with the lower order estimate \eqref{est1} we derived previously. Multiplication by $u_{n,tt}^{(k)}(t)$ and integration with respect to time produces
\begin{align} \label{W1_higher_}
&\int_0^t\int_\Omega a_k(u^{(k)}_{n,tt})^2\, dx\, ds
+\left[\frac{b(1-\delta)}{2}|\nabla u^{(k)}_{n,t}|_{L^2(\Omega)}^2 
+\frac{b\delta}{q+1}|\nabla u^{(k)}_{n,t}|_{L^{q+1}(\Omega)}^{q+1}\right]_0^t \nn \\
&+\frac{\alpha}{2} \Bigl[\int_{\hG} (u^{(k)}_{n,t})^2  \, dx\Bigr]_0^t
\nn \\
=&\,
c^2\int_0^t |\nabla u^{(k)}_{n,t}|_{L^2(\Omega)}^2\, ds
-c^2\left[\int_\Omega \nabla u^{(k)}_n \cdot \nabla  u^{(k)}_{n,t} \, dx\right]_0^t+\int_0^t \int_{\Gamma} g_k u^{(k)}_{n,tt} \, dx \, ds \\ 
& -\int_0^t\int_\Omega f_k u^{(k)}_{n,t}u^{(k)}_{n,tt}\, dx\, ds. \nn
\end{align}
\noindent To estimate the boundary integral on the right side, we employ \eqref{Poincare_est2} to obtain
\begin{equation} \label{W1_ineq2}
\begin{split}
&\int_0^t \int_{\Gamma}  g_k u^{(k)}_{n,tt} \, dx \, ds \\
 \leq & \ C_1^{tr}|u^{(k)}_{n,t}(t)|_{W^{1,q+1}(\Om)}|g_k(t)|_{W^{-\frac{q}{q+1},\frac{q+1}{q}}(\Gamma)}-\int_{\Gamma} g_k(0)u^{(k)}_{n,t}(0) \, dx  \\
&-\int_0^t \int_{\Gamma} g_{k,t} u^{(k)}_{n,t} \, dx \, ds  \\
 \leq& \ \eta \|\nabla u^{(k)}_{n,t}\|^{q+1}_{L^{\infty}(0,T;L^{q+1}(\Om))}+ \epsilon_0 \|u^{(k)}_{n,t}\|^{2}_{L^{\infty}(0,T;L^2(\Om))} \\
 & + C(\eta, q+1)(C_1^{tr}(1+C_P))^{\frac{q+1}{q}}\|g_k\|^{\frac{q+1}{q}}_{L^{\infty}(0,T;W^{-\frac{q}{q+1},\frac{q+1}{q}}(\Gamma))}  \\
&+\frac{1}{2\epsilon_0}(C_1^{tr}C_2^{\Om})^2\|g_k\|^2_{L^{\infty}(0,T;W^{-\frac{q}{q+1},\frac{q+1}{q}}(\Gamma))} +|u_{1,n}|^{q+1}_{W^{1,q+1}(\Omega)} \\ &+C(1,q+1)(C_1^{tr}|g_k(0)|_{W^{-\frac{q}{q+1},\frac{q+1}{q}}(\Gamma)})^{\frac{q+1}{q}}\\
&  +C(\epsilon_1,q+1)(C_1^{tr}(1+C_P))^{\frac{q+1}{q}}\|g_{k,t}\|^{\frac{q+1}{q}}_{L^{\frac{q+1}{q}}(0,T;W^{-\frac{q}{q+1},\frac{q+1}{q}}(\Gamma))}  \\
& +\epsilon_1 \|\nabla u_{n,t}^{(k)}\|^{q+1}_{L^{q+1}(0,T;L^{q+1}(\Om))}+\frac{1}{2\epsilon_0}(C_1^{tr} C^{\Om}_2)^2\|g_{k,t}\|^2_{L^{1}(0,T;W^{-\frac{q}{q+1},\frac{q+1}{q}}(\Gamma))}, 
\end{split}
\end{equation}
which together with 
\begin{align} \label{W1_ineq3}
&\int_0^t \int_{\Omega} f_k u^{(k)}_{n,t} u^{(k)}_{n,tt}\, dx \, ds \nn \\
\leq& \ \frac{1}{2 \tau}(C^{\Omega}_{H^1,L^4})^4 \|f_k\|^2_{L^{\infty}(0,T;H^1(\Omega))}\Bigl[T\|u^{(k)}_{n,t}\|^2_{L^{\infty}(0,T;L^2(\Omega))}\nn \\
 &+\|\nabla u^{(k)}_{n,t}\|^2_{L^2(0,T;L^2(\Omega))}\Bigr] 
 +\frac{\tau}{2} \|u^{(k)}_{n,tt}\|^2_{L^2(0,T;L^2(\Omega))}, 
\end{align}
and taking $\displaystyle \esssup_{[0,T]}$ in \eqref{W1_higher_} leads to the estimate
\begin{equation}  \label{W1lower}
\begin{split}
& \frac{\underline{a}-\tau}{2}  \|u^{(k)}_{n,tt}\|^2_{L^2(0,T;L^2(\Omega))}+\Bigl(\frac{b(1-\delta)}{4}-\sigma \Bigr)\|\nabla u^{(k)}_{n,t}\|^2_{L^{\infty}(0,T;L^2(\Om))}  \\
&+\Bigl(\frac{b\delta}{2(q+1)}-\eta \Bigr)\|\nabla u^{(k)}_{n,t}\|^{q+1}_{L^{\infty}(0,T;L^{q+1}(\Om))}+\frac{\alpha}{4}\|u^{(k)}_{n,t}\|^2_{L^{\infty}(0,T;L^2(\hG))}  \\
 \leq& \ c^2\|\nabla u^{(k)}_{n,t}\|^2_{L^2(0,T;L^2(\Om))}+\sigma |\nabla u_{1,n}|^2_{L^2(\Omega)}+\frac{c^4}{4 \sigma}(\|\nabla u^{(k)}_n\|^2_{L^{\infty}(0,T;L^2(\Omega))}  \\
 &+|\nabla u_{0,n}|^2_{L^2(\Omega)}) +\frac{1}{2 \tau}(C^{\Omega}_{H^1,L^4})^4 \tilde{b}^2[T\|u^{(k)}_{n,t}\|^2_{L^{\infty}(0,T;L^2(\Omega))} \\
&+\|\nabla u^{(k)}_{n,t}\|^2_{L^2(0,T;L^2(\Omega))}] 
+ \eta \|\nabla u^{(k)}_{n,t}\|^{q+1}_{L^{\infty}(0,T;L^{q+1}(\Om))}\\
& +\epsilon_1 \|\nabla u_{n,t}^{(k)}\|^{q+1}_{L^{q+1}(0,T;L^{q+1}(\Om))}
+ \epsilon_0 \| u^{(k)}_{n,t}\|^{2}_{L^{\infty}(0,T;L^2(\Om))} \\
&+\frac{1}{2\epsilon_0}(C_1^{tr}C_2^{\Om})^2\Bigl(\|g_k\|^2_{L^{\infty}(0,T;W^{-\frac{q}{q+1},\frac{q+1}{q}}(\Gamma))} 
+\|g_{k,t}\|^2_{L^{1}(0,T;W^{-\frac{q}{q+1},\frac{q+1}{q}}(\Gamma))}\Bigr) \\ &+|u_{1,n}|^{q+1}_{W^{1,q+1}(\Omega)}+C(1,q+1)(C_1^{tr}|g_k(0)|_{W^{-\frac{q}{q+1},\frac{q+1}{q}}(\Gamma)})^{\frac{q+1}{q}} \\
& + \frac{b(1-\delta)}{2}|\nabla u_{1,n}|^2_{L^2(\Om)} +\frac{\alpha}{2} |u_{1,n}|^2_{L^2(\hG)}+\frac{b\delta}{q+1}|\nabla u_{1,n}|^{q+1}_{L^{q+1}(\Om)} \\
&+(C_1^{tr}(1+C_P))^{\frac{q+1}{q}}\Bigl(C(\epsilon_1,q+1) \|g_{k,t}\|^{\frac{q+1}{q}}_{L^{\frac{q+1}{q}}(0,T;W^{-\frac{q}{q+1},\frac{q+1}{q}}(\Gamma))} \\
&+C(\eta, q+1)\|g_k\|^{\frac{q+1}{q}}_{L^{\infty}(0,T;W^{-\frac{q}{q+1},\frac{q+1}{q}}(\Gamma))}\Bigr).
\end{split}
\end{equation}
\noindent Since there are terms on the right side in \eqref{W1lower} which cannot be dominated by the terms on the left hand side, we need to also employ the lower estimate \eqref{lower_estimate}. Adding \eqref{lower_estimate} and $\mu$ times \eqref{W1lower}  yields \eqref{est2} with $u$ replaced by $u^{(k)}_n$, provided that we choose
\begin{align} \label{W1constants}
& 0< \tau < \underline{a}, \ 0 < \eta < \frac{b\delta}{2(q+1)}, \ 0< \sigma < \frac{b(1-\delta)}{4}, \nn \\
& 0< \mu < \min \Bigl\{\frac{\frac{b(1-\delta)}{2}-(C^{\Omega}_{H^1,L^4})^2\hat{b}}{\frac{1}{2\tau}(C^{\Omega}_{H^1,L^4})^4\tilde{b}^2 +c^2},\frac{\frac{\underline{a}}{4}-(C^{\Omega}_{H^1,L^4})^2\hat{b}T- \epsilon_0}{\epsilon_0 +\frac{1}{2\tau}(C^{\Omega}_{H^1,L^4})^4\tilde{b}^2T},\frac{\sigma}{c^2},\frac{\frac{b\delta}{2}-\epsilon_1}{\epsilon_1} \Bigr\},
\end{align}
so that the coefficients in \eqref{est2} are positive.\\
 As by assumption
$g_k \in L^{\infty}(0,T;W^{-\frac{q}{q+1},\frac{q+1}{q}}(\Gamma))$ and $g_{k,t} \in L^{\frac{q+1}{q}}(0,T;W^{-\frac{q}{q+1},\frac{q+1}{q}}(\Gamma))$, $\big(u_n^{(k)}\big)_{n\in\N}$ is a bounded sequence in 
\begin{equation*}
\begin{aligned}
X:=~& C^1(0,T;W^{1,q+1}(\Om)) \cap H^2(0,T;L^2(\Om)).
\end{aligned}
\end{equation*}
We further obtain 
\begin{align}
& \big( u_{n,t}^{(k)}\big)_{\in\N}\text{ is uniformly bounded in } L^2(0,T; L^2(\Omega)), \label{bound5}\\
& \big(\nabla u_{n,t}^{(k)}\big)_{\in\N}\text{ is uniformly bounded in } L^{q+1}(0,T; L^{q+1}(\Omega)),   \label{bound6} \\
& |\nabla u_{n,t}^{(k)}|^{q-1} \nabla u_{n,t}^{(k)} \ \text{ is uniformly bounded in } L^{\frac{q+1}{q}}(0,T;L^{\frac{q+1}{q}}(\Omega)) \ \text{and} \label{bound7}  \\
& \big( u_{n,t|\hG}^{(k)}\big)_{\in\N}\text{ is uniformly bounded in } L^2(0,T; L^2(\hG)), \label{bound8}
\end{align}
which are reflexive Banach spaces. \\
From here, after proceeding as in the step 1.(c) and 2, we can conclude that \eqref{W1linweak} has a unique solution $u \in X$ which satisfies the estimate \eqref{est2}.
\end{proof}
Let us now consider the boundary value problem \eqref{W1lin} with an added  lower order linear damping term:
\begin{equation}\label{W1lin_beta}
\begin{cases}
a u_{tt}-c^2\Delta u-b\,\text{div}\Bigl( 
((1-\delta) +\delta|\nabla u_t|^{q-1})\nabla u_t\Bigr)+\beta u_t \vspace{2mm} \\+fu_t=0
\, \text{ in } \Omega \times (0,T),
\vspace{2mm} \\
c^2 \frac{\partial u}{\partial n}+b((1-\delta)+\delta|\nabla u_t|^{q-1})\frac{\partial u_t}{\partial n}=g  \ \ \text{on} \ \Gamma \times (0,T),
\vspace{2mm} \\
\alpha u_t+c^2 \frac{\partial u}{\partial n}+b((1-\delta)+\delta|\nabla u_t|^{q-1})\frac{\partial u_t}{\partial n}=0  \ \ \text{on} \ \hat{\Gamma} \times (0,T),
\vspace{2mm} \\
(u,u_t)=(u_0, u_1) \ \ \text{on} \ \overline{\Om}\times \{t=0\},
\end{cases}
\end{equation}
where $\beta>0$. This is a linearized version of the problem \eqref{W1_beta} with nonlinearity appearing only through the damping term. The additionaly introduced $\beta-$lower order term will allow us to remove restrictions on final time $T$ in the estimates \eqref{est1} and \eqref{est2}.
Indeed, by testing the equation with $u_t$ and integrating with respect to space and time,  we obtain
\begin{equation*}
\begin{split}
&\frac{1}{2}\left[\int_\Omega a\left(u_{t}\right)^2\, dx
+ c^2 |\nabla u|_{L^2(\Omega)}^2 \right]_0^t+ b \int_0^t \int_{\Omega} \Bigl((1-\delta) +\delta|\nabla u_{t}|^{q-1}\Bigr)|\nabla u_{t}|^2 \, dx \, ds \\
&+\beta \int_0^t \int_{\Omega} |u_{t}|^{2} \, dx \, ds +\alpha \int_0^t \int_{\hG} | u_{t}|^2 \, dx \, ds\\
\leq& \
\hat{b}(C^{\Om}_{H^{1},L^4})^2 \int_0^t | u_{t}|_{H^{1}(\Omega)}^2 \, ds
+\epsilon_1 \|\nabla u_{t}\|^{q+1}_{L^{q+1}(0,T;L^{q+1}(\Om))}\nn \\
&+C(\epsilon_1,q+1)(C_1^{tr}(1+C_P))^{\frac{q+1}{q}}\|g\|^{\frac{q+1}{q}}_{L^{\frac{q+1}{q}}(0,T;W^{-\frac{q}{q+1},\frac{q+1}{q}}(\Gamma))} \nn \\
&+\epsilon_0 \|u_{t}\|^2_{L^{\infty}(0,T;L^2(\Om)}+\frac{1}{4\epsilon_0}(C_1^{tr} C^{\Om}_2)^2\|g\|^2_{L^{1}(0,T;W^{-\frac{q}{q+1},\frac{q+1}{q}}(\Gamma))},
\end{split}
\end{equation*}
which leads to the lower order energy estimate
\begin{equation} \label{W1_beta_est1}
\begin{split}
&(\frac{\underline{a}}{4}-\epsilon_0)\|u_{t}\|^2_{\LiT}
 + \Bigl(\frac{b(1-\delta)}{2}-\hat{b}(C^{\Om}_{H^{1},L^4})^2\Bigr)\|\nabla u_{t}\|^2_{\LT}  \\
&  +\Bigl(\frac{b\delta}{2}-\epsilon_1\Bigr)\|\nabla u_{t}\|^{q+1}_{\LqT} 
+(\frac{\beta}{2}-\hat{b}(C^{\Om}_{H^{1},L^4})^2))\|u_t\|^{2}_{\LT}  \\
& + \frac{c^2}{4} \|\nabla u\|_{\LiT}^2+\frac{\alpha}{2}\| u_{t}\|^2_{L^2(0,T;L^2(\hG))} \\
\leq& \  \frac{\overline{a}}{2}|u_{1}|^2_{L^2(\Om)}+\frac{c^2}{2}|\nabla u_{0}|^2_{L^2(\Om)}+\frac{1}{4\epsilon_0}(C_1^{tr} C^{\Om}_2)^2\|g\|^2_{L^{1}(0,T;W^{-\frac{q}{q+1},\frac{q+1}{q}}(\Gamma))} \\
&+C(\epsilon_1,q+1)(C_1^{tr}(1+C_P))^{\frac{q+1}{q}}\|g\|^{\frac{q+1}{q}}_{L^{\frac{q+1}{q}}(0,T;W^{-\frac{q}{q+1},\frac{q+1}{q}}(\Gamma))}, 
\end{split}
\end{equation}
provided that $\|f-\frac{1}{2}a_t\|_{L^{\infty}(0,T;L^2(\Om))} \leq \hat {b} < \text{min} \{\frac{\beta}{2(C^{\Om}_{H^{1},L^4})^2},\frac{b(1-\delta)}{2(C^{\Om}_{H^{1},L^4})^2}\}$ and that  $0< \epsilon_0 < \frac{\underline{a}}{4}$, $0<\epsilon_1 < \frac{b \delta}{2}$.\\
\noindent Testing with $u_{tt}$ and adding $\mu$ times the obtained estimate to \eqref{W1_beta_est1} results in the higher order energy estimate valid for arbitrary time:
\begin{equation}  \label{W1lin_est2_beta}
\begin{split}
& \mu \frac{\underline{a}-\tau}{2} \|u_{tt}\|^2_{L^2(0,T;L^2(\Omega))}+\mu \Bigl(\frac{b(1-\delta)}{4}-\sigma \Bigr)\|\nabla u_{t}\|^2_{L^{\infty}(0,T;L^2(\Om))}  \\
&+\mu(\frac{\underline{a}+\mu\beta}{4}-\epsilon_0(\mu+1))\|u_t\|^2_{L^{\infty}(0,T;L^2(\Om))}+\check{b} \|\nabla u_{t}\|^2_{L^2(0,T;L^2(\Omega))} \\
&+\mu \frac{\alpha}{4}\|u_{t}\|^2_{L^{\infty}(0,T;L^2(\hG))}+(\frac{b\delta}{2}-\mu(\epsilon_1+1)) \|\nabla u_{t}\|^{q+1}_{L^{q+1}(0,T;L^{q+1}(\Om))} 
  \\
  & + \Bigl(\frac{\beta}{2}-(C^{\Om}_{H^{1},L^4})^2\hat{b}-\frac{\mu}{2\tau}(C^{\Om}_{H^{1},L^4})^4\tilde{b}^2\Bigr)\|u_t\|^2_{\LT}\\
& +\frac{c^2}{4}\Bigl(1-\mu \frac{c^2}{\sigma}\Bigr) \|\nabla u\|^2_{L^{\infty}(0,T;L^2(\Om))}+\frac{\alpha}{2}\|u_{t}\|^2_{L^2(0,T;L^2(\hG))} \\
&+\mu (\frac{b\delta}{2(q+1)}-\eta)\|\nabla u_{t}\|^{q+1}_{L^{\infty}(0,T;L^{q+1}(\Om))} \\
\leq&   \ \overline{C}\Bigl(C_{\Gamma}(g)+|u_{1}|^2_{H^1(\Om)}+|\nabla u_{0}|^2_{L^2(\Omega)} +|u_{1}|^{q+1}_{W^{1,q+1}(\Omega)}+ |u_{1}|^2_{L^2(\hG)} \Bigr), 
\end{split}
\end{equation}
with $\check{b}= \frac{b (1-\delta)}{2}-(C^{\Omega}_{H^1,L^4})^2\hat{b}  -\mu(\frac{1}{2\tau}(C^{\Omega}_{H^1,L^4})^4\tilde{b}^2+c^2)$, for some appropriately chosen $\overline{C}>0$. Therefore we obtain:
\begin{proposition} \label{prop:W1lin_beta}
Let $\beta>0$ and the assumptions (i) in Proposition \ref{prop:W1lin} hold, with 
\begin{align*}
\|f-\frac{1}{2}a_t\|_{L^{\infty}(0,T;L^2(\Om))} \leq \hat {b} < \text{min} \{\frac{\beta}{2(C^{\Om}_{H^{1},L^4})^2},\frac{b(1-\delta)}{2(C^{\Om}_{H^{1},L^4})^2}\} \,. \label{W1_f_a_beta}
\end{align*}
\noindent Then \eqref{W1lin_beta} has a unique weak solution in $\tilde{X}$, with $\tilde{X}$ defined as in \eqref{W1_Xtilde}, which satisfies \eqref{W1_beta_est1} for some sufficiently small constants $\epsilon_0, \epsilon_1>0$. \\
\noindent If, in addition to (i), the assumptions (ii) in Proposition \ref{prop:W1lin} are satisfied, then $u\in X$, with $X$ defined as in \eqref{W1_X},
and $u$ satisfies the energy estimate \eqref{W1lin_est2_beta} for some sufficiently small constants $\epsilon_0, \epsilon_1, \mu, \sigma, \tau>0$ and some large enough $\overline{C}$, independent of $T$.
\end{proposition}
We continue with considering an equation with an added lower order nonlinear damping term:
\begin{equation}\label{W1lin_gamma}
\begin{cases}
a u_{tt}-c^2\Delta u-b\,\text{div}\Bigl( 
((1-\delta) +\delta|\nabla u_t|^{q-1})\nabla u_t\Bigr)+\gamma |u_t|^{q-1}u_t+fu_t \\ =0
\, \text{ in } \Omega \times (0,T],
\vspace{2mm} \\
c^2 \frac{\partial u}{\partial n}+b((1-\delta)+\delta|\nabla u_t|^{q-1})\frac{\partial u_t}{\partial n}=g  \ \ \text{on} \ \Gamma \times (0,T],
\vspace{2mm} \\
\alpha u_t+c^2 \frac{\partial u}{\partial n}+b((1-\delta)+\delta|\nabla u_t|^{q-1})\frac{\partial u_t}{\partial n}=0  \ \ \text{on} \ \hat{\Gamma} \times (0,T],
\vspace{2mm} \\
(u,u_t)=(u_0, u_1) \ \ \text{on} \ \overline{\Om}\times \{t=0\},
\end{cases}
\end{equation}
with $\gamma>0$, which is motivated by the problem \eqref{W1_gamma}. Once we multiply \eqref{W1lin_gamma} by $u_t$ and integrate by parts, we produce 
\begin{align*}
&\frac{1}{2}\left[\int_\Omega a\left(u_{t}\right)^2\, dx
+ c^2 |\nabla u|_{L^2(\Omega)}^2 \right]_0^t
+b \int_0^t \int_{\Omega} \Bigl((1-\delta) +\delta|\nabla u_{t}|^{q-1}\Bigr)|\nabla u_{t}|^2 \, dx \, ds \\
& \quad \  +\alpha \int_0^t \int_{\hG} | u_{t}|^2 \, dx \, ds+\gamma \int_0^t \int_{\Omega} |u_{t}|^{q+1} \, dx \, ds\\
=& \ \int_0^t \int_{\Om} (f-\frac{1}{2} a_t) (u_{t})^2 \, dx\, ds
+\int_0^t \int_{\Gamma} gu_t \, \ dx \, ds.
\end{align*}
We will make use of the following inequality
\begin{equation} \label{W1_boundary_gamma1}
\begin{split}
\int_0^t \int_{\Gamma} gu_t \,  dx \, ds \leq& \ \frac{\epsilon_0}{2} \|u_t\|^{q+1}_{L^{q+1}(0,T;W^{1,q+1}(\Om))} \\
& +C(\tfrac{\epsilon_0}{2}, q+1)(C_1^{tr} \|g\|_{L^{\frac{q+1}{q}}(0,T,W^{-\frac{q}{q+1},\frac{q+1}{q}}(\Gamma))})^{\frac{q+1}{q}},
\end{split}
\end{equation}
and, for $q>1$,
\begin{equation*}
\begin{split}
&\int_0^t \int_{\Om} (f-\frac12 a_t) (u_{t})^2 \, dx\, ds \leq \int_0^t \Bigl(\int_{\Om} |u_t|^{q+1} \, dx \Bigr)^{\frac{2}{q+1}}\Bigl(\int_{\Om} |f-\frac12 a_t|^{\frac{q+1}{q-1}}\Bigr)^{\frac{q-1}{q+1}} \, ds \\
=& \ \int_0^t |u_t|^2_{L^{q+1}(\Om)} |f-\frac12 a_t|_{L^{\frac{q+1}{q-1}}(\Om)} \, ds \\
\leq & \ \frac{\epsilon_0}{2}\|u_t\|^{q+1}_{L^{q+1}(0,T;L^{q+1}(\Om))}+C(\tfrac{\epsilon_0}{2}, \tfrac{q+1}{2})\|f-\frac12 a_t\|^{\frac{q+1}{q-1}}_{L^{\frac{q+1}{q-1}}(0,T;L^{\frac{q+1}{q-1}}(\Om))},
\end{split}
\end{equation*}
to obtain lower order energy estimate
\begin{align}\label{W1_gamma_est1}
&\frac{\underline{a}}{4}\|u_t\|^2_{L^{\infty}(0,T;L^2(\Om))}+\frac{c^2}{4}\|\nabla u\|^2_{L^{\infty}(0,T;L^2(\Om))}+\frac{b(1-\delta)}{2}\|\nabla u_t\|^2_{L^2(0,T;L^2(\Om))}  \nn \\
&+\frac{b\delta-\epsilon_0}{2}\|\nabla u_t\|^{q+1}_{\LqT}+(\frac{\gamma}{2}-\epsilon_0)\|u_t\|^{q+1}_{\LqT} \nn \\&+\frac{\alpha}{2}\|u_t\|^2_{L^2(\hG)} \nn  \\
\leq& \ C(\tfrac{\epsilon_0}{2}, q+1)(C_1^{tr} \|g\|_{L^{\frac{q+1}{q}}(0,T;W^{-\frac{q}{q+1},\frac{q+1}{q}}(\Gamma))})^{\frac{q+1}{q}}+\frac{\overline{a}}{2}|u_1|^2_{L^2(\Om)} \\
&+C(\tfrac{\epsilon_0}{2}, \tfrac{q+1}{2})\|f-\frac12 a_t\|_{L^{\frac{q+1}{q-1}}(0,T;L^{\frac{q+1}{q-1}}(\Om))}^{\frac{q+1}{q-1}}
+\frac{c^2}{2}|\nabla u_0|^2_{L^2(\Om)}, \nn
\end{align}
assuming that $f, a_t \in L^{\frac{q+1}{q-1}}(0,T;L^{\frac{q+1}{q-1}}(\Om))$ and $0<\epsilon_0< \frac{\gamma}{2}$.\\
For obtaining higher order estimate, we multiply \eqref{W1lin_gamma} with $u_{tt}$, integrate with respect to space and time and make use of the estimate
\begin{align} \label{W1_boundary_gamma2}
\int_0^t \int_{\Gamma}  gu_{tt} \, dx \, ds  \leq & \ \eta \|u_t\|^{q+1}_{L^{\infty}(0,T;W^{1,q+1}(\Om))}+\frac{\epsilon_0}{2} \|u_t\|^{q+1}_{L^{q+1}(0,T;W^{1,q+1}(\Om))} \nn \\
& +C(\eta, q+1)(C_1^{tr} \|g\|_{L^{\infty}(0,T,W^{-\frac{q}{q+1},\frac{q+1}{q}}(\Gamma))})^{\frac{q+1}{q}} \\
&+|u_1|^{q+1}_{W^{1,q+1}(\Omega)}+C(1,q+1) (C_1^{tr}|g(0)|_{W^{-\frac{q}{q+1},\frac{q+1}{q}}(\Gamma)})^{\frac{q+1}{q}}\nn \\
&+C(\tfrac{\epsilon_0}{2}, q+1)(C_1^{tr} \|g_t\|_{L^{\frac{q+1}{q}}(0,T,W^{-\frac{q}{q+1},\frac{q+1}{q}}(\Gamma))})^{\frac{q+1}{q}}\, . \nn
\end{align}
In order to avoid dependence on time, we approach estimate \eqref{W1_ineq3} differently this time: by employing the embedding $L^{q+1}(\Om) \hookrightarrow L^{4}(\Om)$, valid for $q \geq 3$, we obtain
\begin{align*}
\int_0^t \int_{\Omega} f u_{t} u_{tt}\, dx \, ds \leq& \ \frac{1}{2 \tau}(C^{\Omega}_{H^1,L^4})^2(C^{\Omega}_{L^{q+1},L^4})^2 \int_0^t |f(s)|^2_{H^1(\Omega)} |u_{t}(s)|^2_{L^{q+1}(\Omega)} \, ds  \\
& +\frac{\tau}{2} \|u_{tt}\|^2_{L^2(0,T;L^2(\Omega))}  \\
\leq& \ \frac{\epsilon_0}{2} \|u_{t}\|^{q+1}_{L^{q+1}(0,T;L^{q+1}(\Omega))}+\frac{\tau}{2} \|u_{tt}\|^2_{L^2(0,T;L^2(\Omega))}  \\
& +C(\tfrac{\epsilon_0}{2},\tfrac{q+1}{2})(\frac{1}{2 \tau}(C^{\Omega}_{H^1,L^4} C^{\Omega}_{L^{q+1},L^4})^2 \|f\|^2_{L^{\frac{2(q+1)}{q-1}}(0,T;H^1(\Omega))})^{\frac{q+1}{q-1}},
\end{align*}
which, together with \eqref{W1_gamma_est1}, leads to the higher order estimate
\begin{equation}  \label{W1lin_est2_gamma}
\begin{split}
& \mu \frac{\underline{a}-\tau}{2} \|u_{tt}\|^2_{L^2(0,T;L^2(\Omega))}+\mu \Bigl(\frac{b(1-\delta)}{4}-\sigma\Bigr)\|\nabla u_{t}\|^2_{L^{\infty}(0,T;L^2(\Om))}  \\
&+\frac{\underline{a}}{4}\|u_{t}\|^2_{L^{\infty}(0,T;L^2(\Om))} +\frac{c^2}{4}\Bigl(1-\mu \frac{c^2}{\sigma}\Bigr) \|\nabla u\|^2_{L^{\infty}(0,T;L^2(\Om))}   \\
& +\mu \frac{\alpha}{4}\|u_{t}\|^2_{L^{\infty}(0,T;L^2(\hG))}+\Bigl(\frac{b\delta}{2}-\frac{\epsilon_0}{2}(\mu+1)\Bigr) \|\nabla u_{t}\|^{q+1}_{L^{q+1}(0,T;L^{q+1}(\Om))}   \\
&+ \Bigl(\frac{b (1-\delta)}{2}-\mu c^2\Bigr)\|\nabla u_{t}\|^2_{L^2(0,T;L^2(\Omega))}+\frac{\alpha}{2}\|u_{t}\|^2_{L^2(0,T;L^2(\hG))}\\
& +\Bigl(\frac{\gamma}{2}-\epsilon_0(\mu+1)\Bigr)\|u_t\|^{q+1}_{\LqT} \\
&+\mu \Bigl(\frac{b\delta}{2(q+1)}-\eta \Bigr)\|\nabla u_{t}\|^{q+1}_{L^{\infty}(0,T;L^{q+1}(\Om))}\\
&+\mu\Bigl(\frac{\gamma}{2(q+1)}-\eta \Bigr)\|u_t\|^{q+1}_{L^{\infty}(0,T;L^{q+1}(\Om))}  \\
\leq& \  \, \overline{C} \Bigl(\displaystyle \sum_{s=0}^1 \|\frac{d^s}{d t^s} g\|^{\frac{q+1}{q}}_{L^{\frac{q+1}{q}}(0,T,W^{-\frac{q}{q+1},\frac{q+1}{q}}(\Gamma))}+\|g\|^{\frac{q+1}{q}}_{L^{\infty}(0,T,W^{-\frac{q}{q+1},\frac{q+1}{q}}(\Gamma))} 
  \\
& +\|f-\frac{1}{2} a_t\|_{L^{\frac{q+1}{q}}(0,T;L^{\frac{q+1}{q}}(\Om))}^{\frac{q+1}{q}}+\|f\|_{L^{\frac{2(q+1)}{q-1}}(0,T;H^1(\Om))}^{\frac{2(q+1)}{q-1}}+ |u_{1}|^2_{H^1(\Om)} \\
&+|\nabla u_{0}|^2_{L^2(\Om)} 
+|u_{1}|^{q+1}_{W^{1,q+1}(\Om)}+ |u_{1}|^2_{L^2(\hG)} \Bigr), 
\end{split}
\end{equation}
assuming $f \in L^{\frac{2(q+1)}{q-1}}(0,T;H^1(\Om)$ and choosing $\tau, \sigma, \eta, \mu>0$ to be sufficiently small.
\begin{proposition} \label{prop:W1lin_gamma}
Let $T>0$, $c^2$, $b$, $\gamma>0$, $\alpha \geq 0$, $\delta \in (0,1)$, $q>1$ and 
\begin{enumerate}
\item[(i)]
\begin{itemize} 
\item 
$a\in L^{\infty}(0,T;L^{\infty}(\Omega)) 
$, $a_t\in L^{\frac{q+1}{q-1}}(0,T;L^{\frac{q+1}{q-1}}(\Om))$,
$0<\underline{a}\leq a(t,x)\leq \overline{a}$,
\item
$f \in L^{\frac{q+1}{q-1}}(0,T;L^{\frac{q+1}{q-1}}(\Omega))$, 
\item
 $g \in L^{\frac{q+1}{q}}(0,T;W^{-\frac{q}{q+1},\frac{q+1}{q}}(\Gamma))$,
\item
$ u_0\in H^1(\Omega)$, $u_1\in L^2(\Omega)$.
\end{itemize}
\end{enumerate}
\noindent Then \eqref{W1lin_gamma} has a unique weak solution in $\tilde{X}$, with $\tilde{X}$ defined as in \eqref{W1_Xtilde}, which satisfies \eqref{W1_gamma_est1} for some $0<\epsilon_0< \frac{\gamma}{2}$. \\
\noindent If, in addition to (i), the following assumptions are satisfied
\begin{enumerate}
\item[(ii)]
\begin{itemize}
\item $q \geq 3$,
\item 
$f \in L^{\frac{2(q+1)}{q-1}}(0,T;H^1(\Om))$,
\item
 $g \in L^{\infty}(0,T;W^{-\frac{q}{q+1},\frac{q+1}{q}}(\Gamma))$,  $g_t \in L^{\frac{q+1}{q}}(0,T;W^{-\frac{q}{q+1},\frac{q+1}{q}}(\Gamma))$,
\item
$ u_1\in W^{1,q+1}(\Omega)$,
\end{itemize} 
\end{enumerate} then \eqref{W1lin_gamma} has a unique weak solution in $X$, with $X$ defined as in \eqref{W1_Xtilde}, which satisfies the energy estimate \eqref{W1lin_est2_gamma} for some sufficiently small constants $\mu, \sigma, \tau, \eta>0$ and some large enough $\overline{C}>0$, independent of $T$.
\end{proposition}
\begin{remark}
Due to the terms $\|f-\frac12 a_t\|_{L^{\frac{q+1}{q-1}}(0,T;L^{\frac{q+1}{q-1}}(\Om))}^{\frac{q+1}{q-1}}$ and $\|f\|_{L^{\frac{2(q+1)}{q-1}}(0,T;H^1(\Om))}^{\frac{2(q+1)}{q}}$ appearing on the right hand side in the estimate \eqref{W1lin_est2_gamma}, we will not be able to prove local well-posedness of the problem \eqref{W3_gamma} by employing this estimate. Instead, provided the assumptions (ii) in Proposition \ref{prop:W1lin} hold, we could proceed with the same estimates as in the proof of that proposition, and for evaluating boundary integrals apply \eqref{W1_boundary_gamma1} and \eqref{W1_boundary_gamma2}, to obtain the following energy estimate:
\begin{equation}  \label{W1lin_est2_gamma2}
\begin{split}
& \mu \frac{\underline{a}-\tau}{2} \|u_{tt}\|^2_{L^2(0,T;L^2(\Omega))}+\mu \Bigl(\frac{b(1-\delta)}{4}-\sigma\Bigr)\|\nabla u_{t}\|^2_{L^{\infty}(0,T;L^2(\Om))}  \\
&+\Bigl(\frac{\underline{a}}{4}-(C^{\Om}_{H^1,L^4})^2\hat{b}T- \mu \frac{1}{2\tau}(C^{\Om}_{H^1,L^4})^4\tilde{b}^2T\Bigr)\|u_{t}\|^2_{L^{\infty}(0,T;L^2(\Om))}  \\
&+\frac{c^2}{4}\Bigl(1-\mu \frac{c^2}{\sigma}\Bigr) \|\nabla u\|^2_{L^{\infty}(0,T;L^2(\Om))} +\mu \frac{\alpha}{4}\|u_{t}\|^2_{L^{\infty}(0,T;L^2(\hG))}  \\
& + \check{b}\|\nabla u_{t}\|^2_{L^2(0,T;L^2(\Omega))} +\Bigl(\frac{b\delta}{2}-\epsilon_0(\mu+1)\Bigr) \|\nabla u_{t}\|^{q+1}_{L^{q+1}(0,T;L^{q+1}(\Om))}  \\
& +(\frac{\gamma}{2}-\epsilon_0(\mu+1))\|u_t\|^{q+1}_{\LqT}+\frac{\alpha}{2}\|u_{t}\|^2_{L^2(0,T;L^2(\hG))} \\
&+\mu \Bigl(\frac{b\delta}{2(q+1)}-\eta \Bigr)\|\nabla u_{t}\|^{q+1}_{L^{\infty}(0,T;L^{q+1}(\Om))}\\
& +\mu\Bigl(\frac{\gamma}{2(q+1)}-\eta \Bigr)\|u_t\|^{q+1}_{L^{\infty}(0,T;L^{q+1}(\Om))}  \\
\leq&  \ \overline{C} \Bigl(\displaystyle \sum_{s=0}^1 \|\frac{d^s}{d t^s} g\|^{\frac{q+1}{q}}_{L^{\frac{q+1}{q}}(0,T,W^{-\frac{q}{q+1},\frac{q+1}{q}}(\Gamma))}+\|g\|^{\frac{q+1}{q}}_{L^{\infty}(0,T,W^{-\frac{q}{q+1},\frac{q+1}{q}}(\Gamma))} 
  \\
& + |u_{1}|^2_{L^2(\hG)}+ |u_{1}|^2_{H^1(\Om)}+|\nabla u_{0}|^2_{L^2(\Om)} 
+|u_{1}|^{q+1}_{W^{1,q+1}(\Om)}\Bigr), 
\end{split}
\end{equation}
with $\check{b}=\frac{b (1-\delta)}{2}-(C^{\Om}_{H^1,L^4})^2\hat{b}- \mu( \frac{1}{2\tau}(C^{\Om}_{H^1,L^4})^4\tilde{b}^2 + c^2)$, for some appropriately chosen constants $\tau, \sigma, \epsilon_0, \eta, \mu>0$ and large enough $\overline{C}$, independent of $T$.
\end{remark}
\noindent Relying on Proposition \ref{prop:W1lin}, we can now prove the local well-posedness for the boundary value problem \eqref{W1}.
\begin{theorem} \label{thm:W1}
Let $c^2$, $b>0$, $\alpha \geq 0$, $\delta \in (0,1)$, $k \in \mathbb{R}$, $q>d-1$, $q \geq 1$, $g \in L^{\infty }(0,T;W^{-\frac{q}{q+1},\frac{q+1}{q}}(\Gamma))$, $g_t \in L^{\frac{q+1}{q}}(0,T;H^{-\frac{q}{q+1},\frac{q+1}{q}}(\Gamma))$. 
For any $T>0$ there is a $\kappa_T>0$ such that for all 
$u_0, u_1 \in W^{1,q+1}(\Om)$, with
\begin{align} \label{W1_cond1}
&C_{\Gamma}(g)+|u_0|^2_{L^{1}(\Om)} 
 +|\nabla u_0|^2_{L^{q+1}(\Om)} + |u_{1}|^2_{H^1(\Om)}+|\nabla u_{0}|^2_{L^2(\Om)}+|u_{1}|^{q+1}_{W^{1,q+1}(\Om)}\nn \\
& + |u_{1}|^2_{L^2(\hG)} \leq \kappa_T^2
\end{align}
there exists a unique weak solution $u \in \cW$ of \eqref{W1}, where
\begin{equation}\label{defcW1}
\begin{split}
\cW =\{v\in X 
:& \ \|v_{tt}\|_{L^2(0,T;L^2(\Omega))}\leq \overline{m}\\
& \wedge \| v_t\|_{L^{\infty}(0,T;H^1(\Omega))}\leq \overline{m}\\
& \wedge \| \nabla v_t\|_{L^{q+1}(0,T;L^{q+1}(\Omega))}\leq \overline{M} \\
& \wedge (v,v_t)|_{t=0}=(u_0, u_1)
\}, 
\end{split}
\end{equation}
with
\begin{align} \label{W1_cond2}
2|k|C_{W^{1,\qq+1},L^{\infty}}^\Omega \Bigl[&\max \{1+C_P,C^{\Om}_1 \}\kappa_T+(1+C_P) T^{\frac{q}{q+1}}\overline{M} 
+C^{\Om}_2 T\overline{m})\Bigr] <1,
\end{align}
and $\overline{m}$ and $\overline{M}$ sufficiently small, where $C_{\Gamma}(g)$ is defined as in \eqref{Cg}.
\end{theorem}
\begin{proof}
\noindent We will carry out the proof by using a fixed point argument. We define an operator $\cT :\cW \to X$, $v\mapsto \cT v=u$, where $u$ solves \eqref{W1linweak} with
\begin{align}
a=1-2kv, \ f=-2kv_t.
\end{align}
We will show that assumptions of Proposition \ref{prop:W1lin} are satisfied. Since $v \in \cW$, and $q>d-1$ so we can make use of the embedding $W^{1,q+1}(\Omega) \hookrightarrow L^{\infty}(\Omega)$,  we have by \eqref{W1Poincare}
\begin{align*}
|2kv(x,t)|\leq 2|k|C_{W^{1,\qq+1},L^{\infty}}^\Omega \Bigl[&\max \{1+C_P,C^{\Om}_1,\}\kappa_T+(1+C_P) T^{\frac{q}{q+1}}\overline{M} \\
&  +C^{\Om}_2 T\overline{m}\Bigr],
\end{align*}
and $a_t=-2kv_t \in L^{\infty}(0,T;L^2(\Om)$. \\
It follows that $0 <\underline{a}=1-a_0 < a < \overline{a}=1+a_0$, where
\begin{align*}
a_0=2|k|C_{W^{1,\qq+1},L^{\infty}}^\Omega \Bigl[&\max \{1+C_P,C^{\Om}_1\}\kappa_T+(1+C_P) T^{\frac{q}{q+1}}\overline{M}
+C^{\Om}_2 T\overline{m}\Bigr]. 
\end{align*}
Furthermore,
\begin{align*}
&\|f-\frac{1}{2}a_t\|_{L^{\infty}(0,T;L^2(\Om))}=\|kv_t\|_{L^{\infty}(0,T;L^2(\Om))} \leq |k| \overline{m}, \\
& \|f\|_{L^{\infty}(0,T;H^1(\Om))}=2|k|\|v_t\|_{L^{\infty}(0,T;H^1(\Om))} \leq 2|k|\overline{m}.
\end{align*}
Hence the higher order energy estimate \eqref{est2} holds and by choosing $\overline{m}$, $\overline{M}>0$ such that
\begin{align*}
& 2|k|C^{\Om}_{W^{1,q+1},L^{\infty}}((1+C_P) T^{\frac{q}{q+1}}\overline{M}+C^{\Om}_2 T\overline{m}) <1, \nonumber \\
& \overline{m} < \frac{1}{|k|}\min \Bigl \{\frac{b(1-\delta)}{2(C_{H^1,L^4}^\Omega)^2}, \frac{\underline{a}}{4T(C_{H^1,L^4}^\Omega)^2} \Bigr \},
\end{align*}
and making the bound $\kappa_T$ on initial and boundary data small enough
\begin{align*}
  \kappa_T & < \frac{1}{\max\{1+C_P,C^{\Om}_1\}}\Bigl(\frac{1}{2|k|C^{\Om}_{W^{1,q+1},L^{\infty}}}-(1+C_P) T^{\frac{q}{q+1}}\overline{M}
-C^{\Om}_2 T\overline{m}\Bigr),\\
 \kappa_T^2 &\leq \frac{1}{\overline{C}} \text{min} \Bigl\{\Bigl(\frac{\underline{a}}{4}-(C^{\Om}_{H^1,L^4})^2\hat{b}T-\epsilon_0(\mu+1)-\mu \frac{1}{2\tau}(C^{\Om}_{H^1,L^4})^4\tilde{b}^2T\Bigr)\overline{m}^2, \\
&\quad \quad \quad  \quad \quad \mu\frac{\underline{a}-\tau}{2}\overline{m}^2, \mu \Bigl(\frac{b(1-\delta)}{4}-\sigma\Bigr)\overline{m}^2,  
  \Bigl(\frac{b\delta}{2}-\epsilon_1(\mu+1)\Bigr)\overline{M}^{q+1} \Bigr\},
\end{align*}
we achieve that $u \in \cW$, with constants $\epsilon_0$, $\epsilon_1$, $\tau, \eta, \sigma, \mu$ chosen as in \eqref{epsilon_choice} and \eqref{W1constants} and $\overline{C}$ as in \eqref{est2}. \\
\noindent In order to prove contractivity, consider $v^i \in \cW$, $u^i=\cT v^i \in \cW$, $i=1,2$ and denote $\hat{u}=u^1-u^2, \hat{v}=v^1-v^2$. Subtracting the equation \eqref{W1lin} for $u^1$ and $u^2$ yields:
\begin{eqnarray} \label{contractivity}
\begin{cases}
(1-2kv^1)\hat{u}_{tt} -c^2 \Delta \hat{u}-b(1-\delta) \Delta \hat{u}_t -b\delta \, \text{div}\Bigl(|\nabla u^1_t|^{\qq-1}\nabla u^1_t-|\nabla u^2_t|^{\qq-1}\nabla u^2_t\Bigr), \vspace{2mm}\\
 =2k(\hat{v}u^2_{tt}+v_t^1\hat{u_t}+\hat{v}_tu_t^2) \  \ \text{in} \ \Om,
\vspace{2mm} \\
c^2 \frac{\partial \hat{u}}{\partial n}+b(1-\delta)\frac{\partial \hat{u}_t}{\partial n}+b\delta(|\nabla u^1_t|^{q-1}\frac{\partial u^1_t}{\partial n}-|\nabla u^2_t|^{q-1}\frac{\partial u^2_t}{\partial n})=0  \ \ \text{on} \ \Gamma,
\vspace{2mm} \\
\alpha \hat{u}_t+c^2 \frac{\partial \hat{u}}{\partial n}+b(1-\delta)\frac{\partial \hat{u}_t}{\partial n}+b\delta(|\nabla u^1_t|^{q-1}\frac{\partial u^1_t}{\partial n}-|\nabla u^2_t|^{q-1}\frac{\partial u^2_t}{\partial n})=0  \ \ \text{on} \ \hat{\Gamma},
\vspace{2mm}\\
(\hat{u},\hat{u}_t)|_{t=0}=(0, 0).
\end{cases}
\end{eqnarray}
\noindent After testing \eqref{contractivity} with $\hat{u}_t$ and making use of the inequality \eqref{wsig_nonneg}, %
we obtain 
\begin{align*}
& \frac{1}{2}\Bigl[\int_{\Omega} (1-2kv^1)(\hat{u}_t)^2 \, dx+c^2 |\nabla \hat{u}|^2_{L^2(\Omega)}\Bigr]_0^t+b(1-\delta)\int_0^t |\nabla \hat{u_t}|^2_{L^2(\Omega)} \, ds \nn \\
&+\alpha \int_0^t \int_{\hG} |\hat{u}_t|^2 \, dx \, ds\\
\leq& \ 2|k| \int_0 ^t \int_{\Omega} (\frac{1}{2}v_t^1(\hat{u}_t)^2+\hat{v}u_{tt}^2\hat{u}_t+\hat{v}_tu_t^2\hat{u}_t) \, dx \, ds,
\end{align*}
and therefore we have
\begin{align*}
& \frac{1}{2}\Bigl[\int_{\Omega} (1-2kv^1)(\hat{u}_t)^2 \, dx+c^2 |\nabla \hat{u}|^2_{L^2(\Omega)}\Bigr]_0^t+b(1-\delta)\int_0^t |\nabla \hat{u_t}|^2_{L^2(\Omega)} \, ds \nn \\
&+\alpha \int_0^t \int_{\hG} |\hat{u}_t|^2 \, dx \, ds\\
\leq& \ |k|(C^{\Omega}_{H^1,L^4})^2 \Bigl(\|v_t^1\|_{L^{\infty}(0,T;L^2(\Omega))} \int_0^t | \hat{u}_t|^2_{H^1(\Omega)} \, ds \\
&+ \|u^2_{tt}\|_{L^2(0,T;L^2(\Omega))}[\| \hat{v}\|^2_{L^{\infty}(0,T;H^1(\Omega))}+\int_0^t |\hat{u}_t|^2_{H^1(\Omega)} \, ds] \\
&+\|u_t^2\|_{L^{\infty}(0,T;L^2(\Omega))}[\| \hat{v}_t\|^2_{L^2(0,T;H^1(\Omega))}+\int_0^t | \hat{u}_t|^2_{H^1(\Omega)} \, ds]\Bigr).
\end{align*}
Utilizing the fact that $v^1, v^2, u^1, u^2 \in \cW$ and the inequalities $\|\nabla \hat{v}\|^2_{L^{\infty}(0,T,L^2(\Omega))} \leq T \|\nabla \hat{v}_t\|^2_{L^2(0,T,L^2(\Omega))}$, $\|\hat{v}\|^2_{L^{\infty}(0,T,L^2(\Omega))} \leq T \|\hat{v}_t\|^2_{L^2(0,T,L^2(\Omega))}$ and $\|\hat{v}_t\|^2_{L^2(0,T;L^2(\Om))} \leq T \|\hat{v}_t\|^2_{L^{\infty}(0,T;L^2(\Om))}$ leads to 
\begin{align*}
& \frac{1-a_0}{4}\|\hat{u}_t\|^2_{L^{\infty}(0,T;L^2(\Om))}+\frac{c^2}{4}\|\nabla \hat{u}\|^2_{L^{\infty}(0,T;L^2(\Om))}+\frac{b(1-\delta)}{2}\|\nabla \hat{u}_t\|^2_{L^2(0,T;L^2(\Omega))}\\
\leq& \ |k|(C^{\Omega}_{H^1,L^4})^2 \overline{m}\Bigl(3(T\|\hat{u}_t\|^2_{L^{\infty}(0,T;L^2(\Om))}+\|\nabla \hat{u}_t\|^2_{L^2(0,T;L^2(\Om))}) 
+ T^2\| \hat{v}_t\|^2_{L^{\infty}(0,T;L^2(\Omega))} \\
&+T\|\nabla \hat{v}_t\|^2_{L^2(0,T;L^2(\Omega))}
+T\| \hat{v}_t\|^2_{L^{\infty}(0,T;L^2(\Omega))}+\| \nabla \hat{v}_t\|^2_{L^2(0,T;L^2(\Omega))}\Bigr).
\end{align*}
It follows that
\begin{align*}
& \Bigl( \frac{1-a_0}{4}-3T|k|(C^{\Omega}_{H^1,L^4})^2 \overline{m}\Bigr)\|\hat{u}_t\|^2_{L^{\infty}(0,T;L^2(\Om))}+\frac{c^2}{4}\|\nabla \hat{u}\|^2_{L^{\infty}(0,T;L^2(\Om))} \\
&+\Bigl(\frac{b(1-\delta)}{2}-3|k|(C^{\Omega}_{H^1,L^4})^2 \overline{m} \Bigr) \|\nabla \hat{u}_t\|^2_{L^2(0,T;L^2(\Omega))}\\
\leq& \ |k|(C^{\Omega}_{H^1,L^4})^2 \overline{m}(T+1)\max\{1,T\}\Bigl(\|\hat{v}_t\|^2_{L^{\infty}(0,T;L^2(\Om))}+ \|\nabla \hat{v}_t\|^2_{L^2(0,T;L^2(\Omega))})
\end{align*}
and altogether we have 
\begin{equation} \label{contractivity_ab} 
\begin{split}
 &\text{min} \{\frac{1-a_0}{4}-3T|k|(C^{\Omega}_{H^1,L^4})^2 \overline{m},\frac{b(1-\delta)}{2}-3|k|(C^{\Omega}_{H^1,L^4})^2 \overline{m}, \frac{c^2}{4}\} \,|||u|||^2  \\
\leq& \ |k|(C^{\Omega}_{H^1,L^4})^2 \overline{m}(T+1)\max\{1,T\} \,|||v|||^2, 
\end{split}
\end{equation}
where $|||u|||^2=\|\hat{u}_t\|^2_{L^{\infty}(0,T;L^2(\Om))}+ \|\nabla \hat{u}_t\|^2_{L^2(0,T;L^2(\Omega))}+\|\nabla \hat{u}\|^2_{L^{\infty}(0,T;L^2(\Om))}
$. We conclude from \eqref{contractivity_ab} that $\cT$ is a contraction with respect to the norm $|||\cdot|||$, provided that $\overline{m}$ is sufficiently small. This, together with the self-mapping property and  $\cW$ being closed, provides existence and uniqueness of a solution.
\end{proof}
Relying on Proposition \ref{prop:W1lin_beta} we can obtain local well-posedness for the problem \eqref{W1_beta} with $\beta>0$. Since we need to avoid degeneracy of the term $1-2ku$ and therefore make use of the estimate \eqref{W1Poincare} to get the condition \eqref{W1_cond2}, we cannot completely avoid restriction on final time in the fully nonlinear equation. Inspecting the proof of Theorem \ref{thm:W1} immediately yields:
\begin{theorem} \label{thm:W1_beta}
Let $\beta>0$ and the assumptions of Theorem \ref{thm:W1} hold.
For any $T>0$ there is a $\kappa_T>0$ such that for all 
$u_0, u_1 \in W^{1,q+1}(\Om)$, with \eqref{W1_cond1}, there exists a unique weak solution $u \in \cW$ of \eqref{W1}, where $\cW$ is defined as in \eqref{defcW1}, with \eqref{W1_cond2} and $\overline{m}$ and $\overline{M}$ sufficiently small.
\end{theorem}
For obtaining well-posedness for the problem \eqref{W1_gamma} with $\gamma>0$, we cannot rely on estimates in Proposition \ref{prop:W1lin_gamma} to prove self-mapping of the fixed-point operator $\cT$, instead we make use of \eqref{W1lin_est2_gamma2}; therefore restrictions on final time persist in the nonlinear equation. Analogously to Theorem \ref{thm:W1} we obtain:
\begin{theorem} \label{thm:W1_gamma}
Let $\gamma>0$ and the assumptions of Theorem \ref{thm:W1} hold.
For any $T>0$ there is a $\kappa_T>0$ such that for all 
$u_0, u_1 \in W^{1,q+1}(\Om)$, with 
\begin{align} \label{W1_cond1_gamma}
&\displaystyle \sum_{s=0}^{1}\|\frac{d^s}{d t^s} g\|^{\frac{q+1}{q}}_{L^{\frac{q+1}{q}}(0,T;W^{-\frac{q}{q+1},\frac{q+1}{q}}(\Gamma))}  +\|g\|^{\frac{q+1}{q}}_{L^{\infty}(0,T;W^{-\frac{q}{q+1},\frac{q+1}{q}}(\Gamma))} 
 +|u_0|^2_{L^{1}(\Om)}  \nn \\
&+|\nabla u_0|^2_{L^{q+1}(\Om)}+ |u_{1}|^2_{H^1(\Om)}+|\nabla u_{0}|^2_{L^2(\Om)}+|u_{1}|^{q+1}_{W^{1,q+1}(\Om)} + |u_{1}|^2_{L^2(\hG)} \leq \kappa_T^2,
\end{align} there exists a unique weak solution $u \in \cW$ of \eqref{W1}, where $\cW$ is defined as in \eqref{defcW1}, with \eqref{W1_cond2}, and $\overline{m}$ and $\overline{M}$ sufficiently small.
\end{theorem}
\section{Acoustic-acoustic coupling} \label{section_W1coupling}
\noindent We will now consider the problem of an acoustic-acoustic coupling which can be modeled by the equation with coefficients varying in space \eqref{W1_coupled}. We will make the following assumptions on the coefficients in \eqref{W1_coupled}:
\begin{equation}\label{condcoeffacoust}
\begin{cases}
\lambda, \varrho, b, \delta, k \in L^{\infty}(\Om),  \\
\exists \ul{\lambda}, \ol{\lambda}, \ul{\varrho}, \ol{\varrho}: \ 
0<\ul{\lambda}\leq \lambda(x) \leq \ol{\lambda}\,, \ 0<\ul{\varrho} \leq \varrho(x) \leq \ol{\varrho}, \
\mbox{ in }\Omega,\\
\exists \ul{b}, \ol{b}, \ul{\delta}, \ol{\delta}: \ 
0<\ul{b}\leq b(x) \leq \ol{b} \,, \ 0<\ul{\delta}\leq\delta(x) \leq \ol{\delta}<1 \
\mbox{ in }\Omega .
\end{cases}
\end{equation}
\noindent We can again first inspect the problem with nonlinearity present only in the damping term:
\begin{equation}\label{W1lin_coupled}
\begin{cases}
a u_{tt}-\text{div}(\frac{1}{\varrho(x)} \nabla u)-\text{div}\Bigl(b(x)( 
((1-\delta(x)) +\delta(x)|\nabla u_t|^{q-1})\nabla u_t\Bigr)\nn \\
+fu_t=0 \,  \text{ in } \Omega \times (0,T],
\vspace{2mm} \\
\frac{1}{\varrho(x)}\frac{\partial u}{\partial n}+b(x)((1-\delta(x))+\delta(x)|\nabla u_t|^{q-1})\frac{\partial u_t}{\partial n}=g  \ \ \text{on} \ \Gamma \times (0,T],
\vspace{2mm} \\
\alpha (x) u_t+\frac{1}{\varrho(x)} \frac{\partial u}{\partial n}+b(x)((1-\delta(x))+\delta(x)|\nabla u_t|^{q-1})\frac{\partial u_t}{\partial n}=0  \ \ \text{on} \ \hat{\Gamma} \times (0,T],
\vspace{2mm} \\
(u,u_t)=(u_0, u_1) \ \ \text{on} \ \overline{\Om}\times \{t=0\}
.\\
\end{cases}
\end{equation}
\noindent Analogously to Propositon \ref{prop:W1lin} and Theorem \ref{thm:W1} we obtain
\begin{corollary} Let the assumptions \eqref{condcoeffacoust} and the assumptions (i) in Propositon \ref{prop:W1lin} be satisfied,
with 
$$\|f-\frac12 a_t\|_{L^{\infty}(0,T;L^2(\Omega))}\leq\hat{b}
< \min \Bigl \{\frac{\ul{b}(1-\ol{\delta})}{2(C_{H^1,L^4}^\Omega)^2}, \frac{\underline{a}}{4T(C_{H^1,L^4}^\Omega)^2} \Bigr \}\,. $$
\noindent Then \eqref{W1lin_coupled} has a weak solution $u\in \tilde{X}$, with $\tilde{X}$ defined as in \eqref{W1_Xtilde},
which satisfies the energy estimate
\begin{align*}
&\Bigl[\frac{\underline{a}}{4}-\hat{b}(\Chl)^2T-\epsilon_0 \Bigr]\|u_t\|^2_{\LiT}
+ \frac{\underline{c}^2}{4} \|\nabla u\|_{\LiT}^2 \nn \\
&+ \Bigl[\frac{\underline{b}(1-\overline{\delta})}{2}-\hat{b}(\Chl)^2 \Bigr]\|\nabla u_t\|^2_{\LT} \nn \\
&+\Bigl[\frac{\underline{b}\underline{\delta}}{2}-\epsilon_1 \Bigr]\|\nabla u_t\|^{q+1}_{\LqT}+\frac{\underline{\alpha}}{2}\| u_{n,t}^{(k)}\|^2_{L^2(0,T;L^2(\hG))} \nn \\
\leq& \ \frac{\overline{a}}{2}|u_1|^2_{L^2(\Om)}+\frac{\overline{c}^2}{2}|\nabla u_0|^2_{L^2(\Om)}+\frac{1}{4\epsilon_0}(C^{tr}_1 C^{\Om}_2)^2\|g\|^2_{L^{1}(0,T;W^{-\frac{q}{q+1},\frac{q+1}{q}}(\Gamma))}\nn \\ &+C(\epsilon_1,q+1)(C^{tr}_1(1+C_P))^{\frac{q+1}{q}}\|g\|^{\frac{q+1}{q}}_{L^{\frac{q+1}{q}}(0,T;W^{-\frac{q}{q+1},\frac{q+1}{q}}(\Gamma))} ,
\end{align*}
for some sufficiently small constants $\epsilon_0$, $\epsilon_1 >0$. \\
\noindent If, additionally, assumptions (ii) of Proposition \ref{prop:W1lin} hold, then $u\in X$, where $X$ is defined as in \eqref{W1_X}, and satisfies the energy estimate \eqref{est2}, where $\alpha$ is replaced with $\underline{\alpha}$, $b(1-\delta)$ with $\underline{b}(1-\overline{\delta})$, $b\delta$ with $\underline{b}\, \underline{\delta}$ and $\frac{c^2}{2}(1-\mu \frac{c^2}{\sigma})$ with $\frac{\underline{c}^2}{4}-\mu \frac{\overline{c}^4}{4\sigma}$,
for some small enough constants $\tau, \eta, \sigma, \mu>0$
and some large enough $\overline{C}>0$, independent of $T$.
\end{corollary}
\begin{corollary}
Let $g \in L^{\infty }(0,T;W^{-\frac{q}{q+1},\frac{q+1}{q}}(\Gamma))$, $g_t \in L^{\frac{q+1}{q}}(0,T;W^{-\frac{q}{q+1},\frac{q+1}{q}}(\Gamma))$, $q>d-1$, $q \geq 1$ and assumptions \eqref{condcoeffacoust} be satisfied. 
For any $T>0$ there is a $\kappa_T>0$ such that for all $u_0, u_1 \in W^{1,q+1}(\Om)$,
with \eqref{W1_cond1}, there exists a unique weak solution $u \in \cW$ of \eqref{W1_coupled}, where $\cW$ is defined as in \eqref{defcW1}, with \eqref{W1_cond2} with $|k|$ replaced by $|k|_{L^{\infty}(\Om)}$, and $\overline{m}$ and $\overline{M}$ sufficiently small. 
\end{corollary}
\section{Westervelt's equation in the formulation \eqref{W2_beta}} \label{section_W2}
\noindent We begin with the problem \eqref{W2_beta} in the case $\beta=0$:
\begin{align}\label{W2}
\begin{cases}
(1-2ku)u_{tt}-c^2 \text{div} (\nabla u +\varepsilon |\nabla u|^{q-1}\nabla u)
-b \Delta u_t  \\
=2k(u_t)^2 \, \text{ in } \Omega \times (0,T],
\vspace{2mm} \\
c^2 \frac{\partial u}{\partial n}+c^2 \varepsilon|\nabla u|^{q-1}\frac{\partial u}{\partial n}+b\frac{\partial u_t}{\partial n}=g  \ \ \text{on} \ \Gamma \times (0,T],
\vspace{2mm} \\
\alpha u_t+c^2 \frac{\partial u}{\partial n}+c^2 \varepsilon|\nabla u|^{q-1}\frac{\partial u}{\partial n}+b\frac{\partial u_t}{\partial n}=0  \ \ \text{on} \ \hat{\Gamma} \times (0,T],
\vspace{2mm} \\
(u,u_t)=(u_0, u_1) \ \ \text{on} \ \overline{\Om}\times \{t=0\}.
\end{cases}
\end{align}
\noindent We will first study the problem with the nonlinearity appearing only through damping:
\begin{equation}\label{W2lin}
\begin{cases}
au_{tt}-c^2 \text{div} (\nabla u +\varepsilon |\nabla u|^{q-1}\nabla u)
-b \Delta u_t
+fu_t=0
 \, \text{ in } \Omega \times (0,T],
\vspace{2mm} \\
c^2 \frac{\partial u}{\partial n}+c^2 \varepsilon|\nabla u|^{q-1}\frac{\partial u}{\partial n}+b\frac{\partial u_t}{\partial n}=g  \ \ \text{on} \ \Gamma \times (0,T],
\vspace{2mm} \\
\alpha u_t+c^2 \frac{\partial u}{\partial n}+c^2 \varepsilon|\nabla u|^{q-1}\frac{\partial u}{\partial n}+b\frac{\partial u_t}{\partial n}=0  \ \ \text{on} \ \hat{\Gamma} \times (0,T],
\vspace{2mm}\\
(u,u_t)=(u_0, u_1) \ \ \text{on} \ \overline{\Om}\times \{t=0\}
.\\
\end{cases}
\end{equation}
\begin{proposition} \label{prop:W2lin}
Let $T>0$, $c^2$, $b$, $\varepsilon>0$, $\alpha \geq 0$, $\delta \in (0,1)$, $q \geq 1$ and assume that
\begin{itemize} 
\item 
$a\in L^{\infty}(0,T;L^{\infty}(\Omega))$, $ a_t\in L^{\infty}(0,T;L^2(\Omega))$,
$0< \underline{a} < a(x,t) < \overline{a}$,
\item
$f\in L^{\infty}(0,T;L^2(\Omega))$, 
\item
 $g \in L^2(0,T;H^{-1/2}(\Gamma))$,
\item
$ u_0 \in W^{1,q+1}(\Omega)$, $u_1\in L^2(\Omega)$,
\end{itemize}
with 
$$\|f-\frac{1}{2}a_t\|_{L^{\infty}(0,T;L^2(\Omega))} \leq \hat{b} < \min \{\frac{\underline{a}}{4T(\Chl)^2}, \frac{b}{2(\Chl)^2}\}. $$
\noindent Then \eqref{W2lin} has a weak solution 
\begin{align} \label{W2lin_X}
u \in \tilde{X}:=~&  C^1(0,T;L^2(\Om)) \cap C(0,T;W^{1,q+1}(\Om)) \cap H^1(0,T;H^1(\Om)) ,
\end{align}
which satisfies the energy estimate
\begin{equation} \label{W2_energyest}
\begin{split}
&\Bigl[\frac{\underline{a}}{4}-(C_2^{tr})^2 \tau -T\hat{b}(\Chl)^2 \Bigr]\|u_{t}\|^2_{\LiT} +\frac{c^2}{4}\|\nabla u\|^2_{\LiT}\\
& +\Bigl[\frac{b}{2}-(C_2^{\text{tr}})^2 \tau  -\hat{b}(\Chl)^2\Bigr]\|\nabla u_t\|^2_{\LT}
 \\
&+\frac{c^2 \varepsilon}{2(q+1)}\|\nabla u\|^{q+1}_{L^{\infty}(0,T;L^{q+1}(\Om))}+\frac{\alpha}{2}\|u_t\|^2_{L^2(0,T;L^2(\hG))}  \\
 \leq& \  \frac{1}{4\tau }(\|g\|^2_{L^1(0,T;H^{-1/2}(\Gamma))}+\|g\|^2_{L^2(0,T;H^{-1/2}(\Gamma))})+\frac{\underline{a}}{2}|u_{1}|^2_{L^2(\Om)} \\
&+\frac{c^2}{2}|\nabla u_0|^2_{L^2(\Om)}+\frac{c^2 \varepsilon}{q+1}|\nabla u_0|^{q+1}_{L^{q+1}(\Om)}, 
\end{split}
\end{equation}
for some constant 
\begin{equation*}
0< \tau < \min \Bigl \{\frac{b-2\hat{b}(\Chl)^2}{2(C_2^{\text{tr}})^2},\frac{\underline{a}-4(\Chl)^2\hat{b}T}{4(C_2^{\text{tr}})^2} \Bigr \}\, .
\end{equation*}
\end{proposition}
\begin{proof}
The proof follows along the line of the standard Galerkin approximation method. Here we will focus on deriving the energy estimate. \\
The weak form of the problem is given as follows:
\begin{equation}\label{W2linweak}
\begin{split}
&\int_\Omega \Bigl\{a u_{tt} w + c^2 (\nabla u +\varepsilon |\nabla u|^{q-1}\nabla u)\cdot \nabla w +b \nabla u_t \cdot \nabla w
 \Bigr\} \, dx +\alpha \int_{\hat{\Gamma}} u_t w \, dx  \\
=& \ -\int_0^t \int_\Omega fu_t w\, dx+\int_{\Gamma}g w \, dx,\quad \forall w\in W^{1,q+1}(\Omega).
\end{split}
\end{equation}
Testing \eqref{W2linweak} with $u_t$ and integrating with respect to space and time yields
\begin{equation} \label{W2_ineq1}
\begin{split}
&\Bigl[\frac{1}{2}\int_{\Om} a(u_t)^2 \, dx +\frac{c^2}{2} |\nabla u|^2_{L^2(\Om)}+\frac{c^2 \varepsilon}{q+1} |\nabla u|^{q+1}_{L^{q+1}(\Om)}\Bigr]_0^t  \\
& + b \int_0^t |\nabla u_t(s)|^2_{L^2(\Om)} \, ds+\alpha \int_0^t |u_t(s)|^2_{L^2(\hG)} \, ds  \\
 =& \- \int_0^t \int_{\Om} (f-\frac{1}{2}a_t)(u_t)^2 \, dx \, ds + \int_0^t \int_{\Gamma} gu_t \, dx \, ds  \\ 
\leq& \ \|f-\frac{1}{2}a_t\|_{L^{\infty}(0,T;L^2(\Om))} \int_0^t  |u_t(s)|^2_{L^4(\Om)}  \, ds +\int_0^t \int_{\Gamma} gu_t \, dx \, ds.
\end{split}
\end{equation}
By taking $\displaystyle \esssup_{[0,T]}$ in \eqref{W2_ineq1} and making use of the embedding $H^1(\Om) \hookrightarrow L^4(\Om)$ and estimating the boundary integral in the following way 
\begin{align*}
\int_0^t \int_{\Gamma} gu_t \, dx \, ds \leq& \ \tau (C_2^{tr})^2 \|\nabla u_t\|^2_{L^2(0,T;L^2(\Om))}+\frac{1}{4\tau}\|g\|^2_{L^2(0,T;H^{-1/2}(\Gamma))}  \\
& +\tau (C_2^{tr})^2 \| u_t\|^2_{L^{\infty}(0,T;L^2(\Om))}+\frac{1}{4\tau}\|g\|^2_{L^1(0,T;H^{-1/2}(\Gamma))},
\end{align*} we obtain
\begin{align*} 
&\frac{\underline{a}}{4}\|u_{t}\|^2_{\LiT}
+\frac{b}{2}\|\nabla u_t\|^2_{\LT}  +\frac{c^2 \varepsilon}{2(q+1)}\|\nabla u\|^{q+1}_{L^{\infty}(0,T;L^{q+1}(\Om))} \nn \\
& +\frac{c^2}{4}\|\nabla u\|^2_{\LiT}+\frac{\alpha}{2}\|u_t\|^2_{L^2(0,T;L^2(\hG))}  \nn \\
\leq& \ \hat{b}(\Chl)^2 \|u_t\|^2_{L^2(0,T;H^1(\Om))}+\tau (C_2^{tr})^2(\|u_t\|^2_{L^{\infty}(0,T;L^2(\Om))}+\|\nabla u_t\|^2_{L^2(0,T;L^2(\Om))}) \nn \\
&+\frac{1}{4 \tau}(\|g\|^2_{L^2(0,T;H^{-1/2}(\Gamma))}+\|g\|^2_{L^1(0,T;H^{-1/2}(\Gamma))}) +\frac{\underline{a}}{2}|u_{1}|^2_{L^2(\Om)} 
+\frac{c^2}{2}|\nabla u_0|^2_{L^2(\Om)}\nn \\
&+\frac{c^2 \varepsilon}{q+1}|\nabla u_0|^{q+1}_{L^{q+1}(\Om)}, 
\end{align*}
which leads to \eqref{W2_energyest}.
\end{proof}
\noindent If we consider an equation with an added linear lower order damping term  
\begin{equation}\label{W2lin_beta}
\begin{cases}
au_{tt}-c^2 \text{div} (\nabla u +\varepsilon |\nabla u|^{q-1}\nabla u)
-b \Delta u_t +\beta u_t
+fu_t=0  \, \text{ in } \Omega \times (0,T],
\vspace{2mm} \\
c^2 \frac{\partial u}{\partial n}+c^2 \varepsilon|\nabla u|^{q-1}\frac{\partial u}{\partial n}+b\frac{\partial u_t}{\partial n}=g  \ \ \text{on} \, \Gamma \times (0,T],
\vspace{2mm} \\
\alpha u_t+c^2 \frac{\partial u}{\partial n}+c^2 \varepsilon|\nabla u|^{q-1}\frac{\partial u}{\partial n}+b\frac{\partial u_t}{\partial n}=0  \ \ \text{on} \, \hat{\Gamma} \times (0,T],
\\
(u,u_t)=(u_0, u_1) \ \ \text{on} \ \overline{\Om}\times \{t=0\},
\\
\end{cases}
\end{equation}
we will be able to obtain an energy estimate valid for arbitrary time:
\begin{proposition} \label{prop:W2lin_beta}
Let $\beta>0$ and the assumptions in Proposition \ref{prop:W2lin} hold with 
$$\|f-\frac{1}{2}a_t\|_{L^{\infty}(0,T;L^2(\Omega))} \leq \hat{b} < \frac{1}{2(\Chl)^2} \min \{b, \beta\}. $$
\noindent Then \eqref{W2lin_beta} has a weak solution in $\tilde{X}$, defined as in \eqref{W2lin_X}, which satisfies the energy estimate
\begin{equation} \label{W2_energyest_1}
\begin{split}
&\frac{\underline{a}}{4}\|u_{t}\|^2_{\LiT}+\Bigl[\frac{b}{2}-(C_2^{\text{tr}})^2 \tau -\hat{b}(\Chl)^2\Bigr] \|\nabla u_t\|^2_{L^2(0,T;L^2(\Om))}
  \\
& +\Bigl[\frac{\beta}{2}-(C_2^{\text{tr}})^2 \tau -\hat{b}(\Chl)^2\Bigr] \|u_t\|^2_{L^2(0,T;L^2(\Om))}+\frac{c^2}{4}\|\nabla u\|^2_{\LiT} \\
& +\frac{c^2 \varepsilon}{2(q+1)}\|\nabla u\|^{q+1}_{L^{\infty}(0,T;L^{q+1}(\Om))} +\frac{\alpha}{2}\|u_t\|^2_{L^2(0,T;L^2(\hG))}  \\
\leq& \  \frac{1}{4\tau }\|g\|^2_{L^2(0,T;H^{-1/2}(\Gamma))}+\frac{\underline{a}}{2}|u_{1}|^2_{L^2(\Om)}+\frac{c^2}{2}|\nabla u_0|^2_{L^2(\Om)}+\frac{c^2 \varepsilon}{q+1}|\nabla u_0|^{q+1}_{L^{q+1}(\Om)}, 
\end{split}
\end{equation}
for some constant 
\begin{equation*}
0< \tau < \min \Bigl \{\frac{b-2\hat{b}(\Chl)^2}{2(C_2^{\text{tr}})^2},\frac{\beta-2(\Chl)^2\hat{b}}{2(C_2^{\text{tr}})^2} \Bigr \}\, .
\end{equation*}
\end{proposition}
\noindent We now proceed to the question of local existence of weak solutions for the problem \eqref{W2}. 
\begin{theorem} \label{thm:W2}
Let $c^2$, $b>0$, $\alpha \geq 0$, $\delta \in (0,1)$, $k \in \mathbb{R}$, $q>d-1$, $q \geq 1$, $g \in L^2(0,T;H^{-1/2}(\Gamma))$. 
For any $T>0$ there is a $\kappa_T>0$ such that for all $u_0 \in W^{1,q+1}(\Om)$, $u_1 \in L^2(\Om)$ with
\begin{align*}
&\|g\|^2_{L^2(0,T;H^{-1/2}(\Gamma))} +\|g\|^2_{L^1(0,T;H^{-1/2}(\Gamma))}
+|u_{1}|^2_{L^2(\Om)}+|\nabla u_{0}|^2_{L^2(\Om)} 
+|\nabla u_{0}|^{q+1}_{L^{q+1}(\Om)} \nn \\
&+ |u_0|^2_{L^1(\Om)} \leq \kappa_T^2
\end{align*}
there exists a weak solution $u \in \cW$ of \eqref{W2} where
\begin{equation}\label{defcW2}
\begin{split}
\cW =\{v\in \tilde{X} 
:& \ \|v_{t}\|_{L^{\infty}(0,T;L^2(\Omega))}\leq \overline{m}\\
& \wedge \| \nabla v_t\|_{L^2(0,T;L^2(\Omega))}\leq \overline{m}\\
& \wedge \| \nabla v\|_{L^{\infty}(0,T;L^{q+1}(\Omega))}\leq \overline{M} 
\}, 
\end{split}
\end{equation}
with
\begin{align} \label{W2cond}
&2|k|C_{W^{1,\qq+1},L^{\infty}}^\Omega [C_1^{\Om}\kappa_T+(1+C_P)\overline{M}+C_2^{\Om} T \overline{m}] <1,
\end{align}
and $\overline{m}$ and $\overline{M}$ are sufficiently small.
\end{theorem}
\begin{proof}
We define an operator $\cT: \cW \rightarrow \tilde{X}$, $v\mapsto \cT v=u$, where $u$ solves \eqref{W2lin} with
\begin{align}
a=1-2kv, \ f=-2kv_t.
\end{align}
Proposition \ref{prop:W2lin}  will allow us to prove that $\cT$ is a self-mapping. The assumptions of the proposition are satisfied, since for $v \in \cW$ because of \eqref{W2Poincare} we have
$$0< \underline{a}=1-a_0 <a(x,t) < \overline{a}=1+a_0,$$ 
where
\begin{align*}
& a_0= 2|k|C_{W^{1,\qq+1},L^{\infty}}^\Omega [C_1^{\Om}\kappa_T+(1+C_P)\overline{M}+C_2^{\Om} T \overline{m}], 
\end{align*}
and by \eqref{defcW2} $\|f-\frac{1}{2}a_t\|_{L^{\infty}(0,T;L^2(\Om))}=\|kv_t\|_{\LiT} \leq |k|\overline{m}$.
The energy estimate \eqref{W2_energyest} holds and we can conclude that for any $\overline{m}$, $\overline{M}>0$ such that 
\begin{align*}
& 2|k|\Cwl ((1+C_P)\overline{M}+C_2^{\Om}T\overline{m})<1, \\
& \overline{m} \leq \frac{1}{|k|}\min \{\frac{\underline{a}}{4(\Chl)^2}, \frac{b}{2(\Chl)^2}\}, 
\end{align*}
and under the assumption on smallness of initial and boundary data
\begin{equation*}
\begin{split}
\kappa_T<& \ \frac{1}{C_1^{\Om}}\Bigl(\frac{1}{2|k|C_{W^{1,\qq+1},L^{\infty}}^\Omega}-(1+C_P)\overline{M}-C_2^{\Om}T\overline{m}\Bigr),\\
\kappa_T^2 \leq& \ \min\{\frac{1}{\overline{C}}(\frac{\underline{a}}{4}-T\hat{b}(\Chl)^2-(C_2^{tr})^2 \tau)\overline{m}^2, \\
& \quad \quad \quad \frac{1}{\overline{C}}(\frac{b}{2}-\hat{b}(\Chl)^2-(C_2^{tr})^2 \tau )\overline{m}^2, \, \frac{1}{\overline{C}}\frac{c^2 \varepsilon}{2(q+1)}\overline{M}^{q+1}\}, 
\end{split}
\end{equation*} where $\overline{C}=\max\{\frac{1}{4\tau},\frac{\overline{a}}{2}, \frac{c^2}{2}, \frac{c^2 \varepsilon}{q+1}\}$,
operator $\cT$ maps into $\cW$. \\

\noindent Since $\cW$ is closed and bounded in the dual of a separable Banach space, $\cW$ is weakly-star compact. Existence of solutions then results from a compactness argument (see Theorem 6.1, \cite{brunn}): the sequence of fixed point iterates $u^n$ defined by 
$u^n = \mathcal{T} u^{n-1}$, 
$$ (1-2ku^{n-1})u^n_{tt}-c^2 \text{div}(\nabla u^n +\epsilon |\nabla u^n|^{q-1}\nabla u^n)-b \Delta u^n_t+\beta u^n_t=2ku_t^{n-1}u_t^n,$$ with $u^0$ chosen to be compatible with initial and boundary conditions, has a
weakly-star convergent subsequence whose w$\ast$-limit $\bar{u}$ lies in $\mathcal{W}$. This limit is a weak solution of the problem since
\begin{align*}
&\int_0^T \int_\Omega \big\{ ( 1-2k \bar{u}) \bar{u}_{tt}\phi  + [ c^2 ( \nabla \bar{u} + \varepsilon  | \nabla \bar{u} |^{q-1} \nabla \bar{u}) 
+ b \nabla \bar{u}_t ]\cdot \nabla \phi - 2k(\bar{u}_t )^2 \phi \\
&+\beta u_t \phi \big\} dx \, ds 
+\alpha \int_0^T \int_{\hG} \bar{u}_t \phi \, dx \, ds -\int_0^T \int_{\Gamma} g \phi \, dx \, ds\\
=& \ \int_0^T \int_\Omega \big\{ -\bar{u}_t (( 1-2k \bar{u}) \phi )_t + [ c^2 ( \nabla \bar{u} + \varepsilon  | \nabla \bar{u} |^{q-1} \nabla \bar{u}) 
+ b \nabla \bar{u}_t ]\cdot \nabla \phi - 2k(\bar{u}_t )^2 \phi \\
&+\beta u_t \phi \big\} dx \, ds  
+\alpha \int_0^T \int_{\hG} \bar{u}_t \phi \, dx \, ds -\int_0^T \int_{\Gamma} g \phi \, dx \, ds\\
=& \ \int_0^T \int_\Omega \Big\{ - (\bar{u}- u^n)_t ((1-2k\bar{u})\phi)_t + 2k u_t^n ((\bar{u}- u^{n-1} )\phi)_t \\
&- 2k(\bar{u}_t-u_t^n)\bar{u}_t\phi - 2k(\bar{u}_t-u_t^{n-1})u^n_t\phi + \beta(\bar{u}_t-u^n_t)\phi+ [c^2 \nabla(\bar{u}-u^n) \\
&+b\nabla (\bar{u}-u^n)_t]\cdot \nabla \phi
+c^2 \varepsilon \int_0^1 |\nabla(u^n+\sigma\hat{u})|^{q-3} [ |\nabla(u^n+\sigma \hat{u})|^2 \nabla \hat{u}\\
& +(q-1)(\nabla(u^n+\sigma \hat{u}) \cdot \nabla\hat{u})\nabla(u^n+\sigma\hat{u}) ] d\sigma \cdot \nabla\phi \Big\} \,dx\, ds \\
&+\alpha \int_0^T \int_{\hG} (\bar{u}- u^n) \phi \, dx \, ds \, 
\to 0 \text{ as } k \rightarrow \infty,
\end{align*}
for any $\phi \in C_0^\infty((0,T)\times \Omega)$, where $\hat{u}=\bar{u}-u^n$. 
\end{proof}
\indent Relying on Proposition \ref{prop:W2lin_beta}, we can also achieve short-time existence of solutions for the problem \eqref{W2_beta}, with $\beta>0$. Due to the estimate \eqref{W2Poincare} and therefore bound \eqref{W2cond}, the dependency on final time $T$ cannot be completely avoided.
\begin{theorem} \label{thm:W2_beta}
Let the assumptions of Theorem \ref{thm:W2} hold and $\beta>0$. For any $T>0$ there is a $\kappa_T>0$ such that for all $u_0 \in W^{1,q+1}(\Om)$, $u_1 \in L^2(\Om)$ with
\begin{align*}
&\|g\|^2_{L^2(0,T;H^{-1/2}(\Gamma))} +|u_{1}|^2_{L^2(\Om)}+|\nabla u_{0}|^2_{L^2(\Om)} 
+|\nabla u_{0}|^{q+1}_{L^{q+1}(\Om)}+ |u_0|^2_{L^1(\Om)} \leq \kappa_T^2,
\end{align*}
there exists a weak solution $u \in \cW$ of \eqref{W2_beta}, where $\cW$ is defined as in \eqref{defcW2}, with \eqref{W2cond}, and $\overline{m}$ and $\overline{M}$ are sufficiently small.
\end{theorem}
\indent Note that here, as in the case of homogeneous Dirichlet boundary conditions, the uniqueness remains an open problem due to the presence of $q-$Laplace damping term which hinders the derivation of higher order energy estimates. For details, the reader is refered to Remark 8, \cite{brunn}. 
\section{Westervelt's equation in the formulation \eqref{W3_gamma}} \label{section_W3}
\noindent We begin with investigations of the problem \eqref{W3_gamma} in the case $\gamma=0$:
\begin{align}\label{W3}
\begin{cases}
u_{tt}-\frac{c^2}{1-2\tilde{k}u_t}\Delta u
-b\,\text{div}\Bigl(((1-\delta) +\delta|\nabla u_t|^{\qq-1})\nabla u_t\Bigr)  =0 \, \text{ in } \Omega \times (0,T],
\vspace{2mm} \\
\frac{c^2}{1-2\tilde{k}u_t} \frac{\partial u}{\partial n}+b((1-\delta)+\delta|\nabla u_t|^{q-1})\frac{\partial u_t}{\partial n}=g   \ \text{on} \ \Gamma \times (0,T],
\vspace{2mm} \\
\alpha u_t+\frac{c^2}{1-2\tilde{k}u_t} \frac{\partial u}{\partial n}+b((1-\delta)+\delta|\nabla u_t|^{q-1})\frac{\partial u_t}{\partial n}=0  \ \text{on} \ \hat{\Gamma}\times (0,T],
\vspace{2mm} \\
(u,u_t)=(u_0, u_1) \ \ \text{on} \ \overline{\Om}\times \{t=0\}.\\
\end{cases}
\end{align}
We will once again first consider an equation with the nonlinearity only appearing through the damping term:
\begin{equation}\label{W3lin}
\begin{cases}
u_{tt}-a\Delta u
-b\,\text{div}\Bigl(((1-\delta) +\delta|\nabla u_t|^{\qq-1})\nabla u_t\Bigr)
=0 \, \text{ in } \Omega \times (0,T],
\vspace{2mm} \\
a \frac{\partial u}{\partial n}+b((1-\delta)+\delta|\nabla u_t|^{q-1})\frac{\partial u_t}{\partial n}=g  \ \ \text{on} \, \Gamma \times (0,T],
\vspace{2mm} \\
\alpha u_t+a \frac{\partial u}{\partial n}+b((1-\delta)+\delta|\nabla u_t|^{q-1})\frac{\partial u_t}{\partial n}=0  \ \ \text{on} \, \hat{\Gamma} \times (0,T],
\vspace{2mm}
\\
(u,u_t)=(u_0, u_1) \ \ \text{on} \ \overline{\Om}\times \{t=0\}.
\end{cases}
\end{equation}
\begin{proposition} \label{prop:W3lina} 
Let $T>0$, $b>0$, $\alpha \geq 0$, $\delta \in (0,1)$ and assume that
\begin{itemize} 
\item 
$a\in L^2(0,T;L^{\infty}(\Omega))$, $\nabla a\in L^2(0,T;L^2(\Omega))$,
\item
 $g \in L^{\frac{q+1}{q}}(0,T;W^{-\frac{q}{q+1},\frac{q+1}{q}}(\Gamma))$,
\item
$ u_0\in H^1(\Omega)$, $u_1\in L^2(\Omega)$,
\item $q > d-1, q \geq 1$,
\end{itemize}
\begin{align} \label{assumption1}
\begin{cases}
\hat{b}:=\frac{b(1-\delta)}{2}-\frac{T}{2}\|\nabla a\|_{L^2(0,T;L^2(\Omega))}-(\sqrt{T}+\frac{1}{2})\|a\|_{L^2(0,T;L^{\infty}(\Omega))}>0, \\
\tilde{b}:=\frac{1}{4}-2T(\qCwl C^{\Om}_2)^2\|\nabla a\|_{L^2(0,T;L^2(\Omega))}  >0, \ \text{for} \ q>1, 
\end{cases}
\end{align}
\begin{align} \label{cond_b}
\begin{cases}
\hat{b}:=\frac{b}{2}-(\frac{T}{2}+(C^{\Om}_{H^1,L^{\infty}})^2)\|\nabla a\|_{L^2(0,T;L^2(\Om))}\\\quad \quad -(\sqrt{T}+\frac{1}{2})\|a\|_{L^2(0,T;L^{\infty}(\Om))} >0, \\
\tilde{b}:=\frac{1}{4}-T(C^{\Om}_{H^1,L^{\infty}})^2\|\nabla a\|_{L^2(0,T;L^2(\Omega))}>0,\ \text{for} \ q=1.
\end{cases}
\end{align}
\noindent Then \eqref{W3lin} has a weak solution 
\begin{align} \label{W3_tildeX}
u\in \tilde{X}:=~&C^1(0,T;L^2(\Om)) \cap W^{1,q+1}(0,T;W^{1,q+1}(\Om))\},
\end{align}
which, for $q>1$, satisfies the energy estimate
\begin{equation} \label{W3lin_lower}
\begin{split}
&\hat{b}\|\nabla u_t\|^2_{L^2(0,T;L^2(\Om))}+\Bigl(\tilde{b}-\epsilon_0\Bigr)\|u_t\|_{L^{\infty}(0,T;L^2(\Omega)}^2   \\
&+(\frac{b \delta}{2}-\epsilon_1)\|\nabla u_t\|_{L^{q+1}(0,T;L^{q+1}(\Omega))}^{\qq+1} 
+\frac{\alpha}{2}\|u_t\|^2_{L^2(0,T;L^2(\hG))}   \\
\leq& \ \Bigl(\|a\|_{L^2(0,T;L^{\infty}(\Omega))} + \|\nabla a\|_{L^2(0,T;L^2(\Omega))}\Bigr)\frac12 |\nabla u_{0}|_{L^2(\Om)}^2 +\frac{1}{2}|u_1|^2_{L^2(\Om)} \\
&+ C(\tfrac{\epsilon_1}{2}, \tfrac{q+1}{2})T\Bigl((\qCwl (1+C_P))^2 \|\nabla a\|_{L^2(0,T;L^2(\Om))}\Bigr)^{\frac{\qq+1}{\qq-1}} \\
& +C(\tfrac{\epsilon_1}{2},q+1)(\qCwl(1+C_P))^{\frac{q+1}{q}}\|g\|^{\frac{q+1}{q}}_{L^{\frac{q+1}{q}}(0,T;W^{-\frac{q}{q+1},\frac{q+1}{q}}(\Gamma))}  \\
&+\frac{1}{4\epsilon_0} (C_1^{tr}C_2^{\Om})^2\|g\|^2_{L^1(0,T;W^{-\frac{q}{q+1},\frac{q+1}{q}}(\Gamma))},
\end{split}
\end{equation}
and for $q=1$ satisfies 
\begin{equation*} 
\begin{split}
&\Bigl(\hat{b}-\epsilon_0\Bigr)\|\nabla u_t\|^2_{L^2(0,T;L^2(\Om))} 
+\Bigl(\tilde{b}-\epsilon_0  \Bigr)\|u_t\|_{L^{\infty}(0,T;L^2(\Omega)}^2 
+\frac{\alpha}{2}\|u_t\|^2_{L^2(0,T;L^2(\hG))}   \\
\leq& \ \Bigl(\|a\|_{L^2(0,T;L^{\infty}(\Omega))} + \|\nabla a\|_{L^2(0,T;L^2(\Omega))}\Bigr)\frac12 |\nabla u_{0}|_{L^2(\Om)}^2 +\frac{1}{2}|u_1|^2_{L^2(\Om)} \\
&+\frac{1}{4 \epsilon_0}(C_2^{tr})^2(\|g\|^2_{L^1(0,T;H^{-1/2}(\Gamma))}+\|g\|^2_{L^2(0,T;H^{-1/2}(\Gamma))}),
\end{split}
\end{equation*}
for some sufficiently small constants $\epsilon_0, \epsilon_1 >0$.
\end{proposition}
\begin{proof} 
We will focus on acquiring crucial energy estimates. Testing the problem with $u_t$ and integrating with respect to space and time leads to
\begin{align} \label{W3_est0}
&\frac{1}{2}\left[|u_t(s)|_{L^2(\Om)}^2\right]_0^t
+\int_0^t \Bigl(b(1-\delta)|\nabla u_t(s)|_{L^2(\Om)}^2  \nn \\ 
&+b\delta|\nabla u_t(s)|_{L^{q+1}(\Omega)}^{\qq+1} 
 +\alpha |u_t(s)|^2_{L^2(\hG)} \Bigr)\, ds
\nn \\
=& \ \int_0^t\int_\Omega\Bigl( -a \nabla u_t \cdot \nabla u -u_t\nabla a \cdot \nabla u 
\Bigr) \, dx \, ds+\int_0^t \int_{\Gamma} gu_t \, dx \, ds \nn \\
\leq& \ \int_0^t \Bigl(|a(s)|_{L^{\infty}(\Omega)} |\nabla u_t(s)|_{L^2(\Omega)}
+|\nabla a(s)|_{L^2(\Om)}|u_t(s)|_{L^{\infty}(\Om)}\Bigr) \nn \\
& \cdot \Bigl[|\nabla u_{0}|_{L^2(\Om)}+\sqrt{s\int_0^s|\nabla u_t(\sigma)|_{L^2(\Om)}^{2}\,d\sigma} \Bigr] \, ds+\int_0^t \int_{\Gamma} gu_t \, dx \, ds \nn \\
\leq& \ \Bigl(\|a\|_{L^2(0,t;L^{\infty}(\Om))} \|\nabla u_t\|_{L^2(0,t;L^2(\Om))} \nn \\
&+\|\nabla a\|_{L^2(0,t;L^2(\Om))}\sqrt{\int_0^t | u_t(s)|^2_{L^{\infty}(\Om)}\, ds} \Bigr) \nn \\
& \cdot\Bigl[|\nabla u_{0}|_{L^2(\Om)}+\sqrt{t} \|\nabla u_t\|_{L^2(0,t;L^2(\Om))}\Bigr] +\int_0^t \int_{\Gamma} gu_t \, dx \, ds   \\
\leq& \ \|a\|_{L^2(0,T;L^{\infty}(\Om))} \Bigl( \frac{1}{2} \|\nabla u_t\|_{\LT}^2
+\frac{1}{2} |\nabla u_{0}|_{L^2(\Om)}^2 \nn \\
&+\sqrt{T}\|\nabla u_t\|_{\LT}^2 \Bigr) 
+\|\nabla a\|_{L^2(0,T;L^2(\Omega))} \int_0^T | u_t(s)|^2_{L^{\infty}(\Om)}\, ds \nn \\
&+\|\nabla a\|_{L^2(0,T;L^2(\Omega))}
\Bigl(\frac{1}{2} |\nabla u_{0}|_{L^2(\Om)}^2 
+\frac{1}{2} T\|\nabla u_t\|_{\LT}^2 \Bigr) \nn \\
&+\frac{\epsilon_1}{2} \|\nabla u_{t}\|^{q+1}_{L^{q+1}(0,T;L^{q+1}(\Om))} +\epsilon_0 \|u_{t}\|^2_{L^{\infty}(0,T;L^2(\Om))}
\nn \\
&+C(\tfrac{\epsilon_1}{2},q+1)(C_1^{tr}(1+C_P))^{\frac{q+1}{q}}\|g\|^{\frac{q+1}{q}}_{L^{\frac{q+1}{q}}(0,T;W^{-\frac{q}{q+1},\frac{q+1}{q}}(\Gamma))}  \nn\\
&+\frac{1}{4\epsilon_0}(C_1^{tr} C^{\Om}_2)^2\|g\|^2_{L^1(0,T;W^{-\frac{q}{q+1},\frac{q+1}{q}}(\Gamma))}, \nn
\end{align}
where we have applied \eqref{W1_inequality1} to estimate the boundary integral on the right side.  \\
We can make use of the  embedding $W^{1,q+1}(\Om) \hookrightarrow L^{\infty}(\Om)$ together with the inequality \eqref{Poincare_est2} to obtain
\begin{align*}
& \int_0^T | u_t(s)|^2_{L^{\infty}(\Om)}\, ds \leq (\qCwl)^2 \int_0^T | u_t(s)|^2_{W^{1,q+1}(\Om)}  \, ds \\
\leq& \ (\qCwl)^2 \int_0^t \Bigl((1+C_P)|\nabla u_t(s)|_{L^{q+1}(\Om)}+C_2 ^{\Om} |u_t(s)|_{L^2(\Om)} \Bigr)^2  \, ds \\
\leq& \ 2 (\qCwl)^2 (1+C_P)^2 \int_0^t |\nabla u_t(s)|^2_{L^{q+1}(\Om)} \, ds \nn \\
&+2(\qCwl C^{\Om}_2)^2  \|u_t\|^2_{L^2(0,T;L^2(\Om))},
\end{align*}
and then from \eqref{W3_est0}, for $q>1$, we further get
\begin{align} \label{W3lin_est11}
&\frac{1}{2} \left[|u_{t}|_{L^2(\Om)}^2\right]_0^t
+\int_0^t \Bigl(b(1-\delta)|\nabla u_t(s)|_{L^2(\Om)}^2 
+b\delta|\nabla u_t(s)|_{L^{q+1}(\Om)}^{\qq+1} \, ds\Bigr) \nn \\
&+ \alpha \int_0^t |u_t (s)|^2_{L^2(\hG)} \, ds \nn \\
\leq& \ \|a\|_{L^2(0,T;L^{\infty}(\Om))} \Bigl( (\sqrt{T}+\frac12)\|\nabla u_t\|_{L^2(0,T;L^2(\Om))}^2
+\frac{1}{2} |\nabla u_{0}|_{L^2(\Om)}^2\Bigr) \nn \\
& + \epsilon_1 \|\nabla u_t\|^{q+1}_{L^{q+1}(0,T;L^{q+1}(\Om))} +\epsilon_0 \|u_{t}\|^2_{L^{\infty}(0,T;L^2(\Om))} \nn \\
&+C(\tfrac{\epsilon_1}{2}, \tfrac{q+1}{2})T((\qCwl(1+C_P))^2\|\nabla a\|_{\LT})^{\frac{q+1}{q-1}}  \\
& +2\|\nabla a\|_{\LT}(\qCwl C^{\Om}_2)^2  \|u_t\|^2_{L^2(0,T;L^2(\Om))}  \nn \\
&+\|\nabla a\|_{L^2(0,T;L^2(\Omega))}\Bigl(\frac12 T \|\nabla u_t\|_{L^2(0,T;L^2(\Omega))}^2  
+\frac12 |\nabla u_{0}|_{L^2(\Om)}^2\Bigr)\nn  \\
&+C(\tfrac{\epsilon_1}{2},q+1)(C_1^{tr}(1+C_P))^{\frac{q+1}{q}}\|g\|^{\frac{q+1}{q}}_{L^{\frac{q+1}{q}}(0,T;W^{-\frac{q}{q+1},\frac{q+1}{q}}(\Gamma))} \nn \\
& +\frac{1}{4\epsilon_0}(C_1^{tr} C^{\Om}_2)^2\|g\|^2_{L^1(0,T;W^{-\frac{q}{q+1},\frac{q+1}{q}}(\Gamma))} \,,\nn
\end{align}
for some $\epsilon_0, \epsilon_1 >0$.
\noindent By taking $ \displaystyle \esssup_{[0,T]}$ in \eqref{W3lin_est11} and making $\epsilon_0$ and $\epsilon_1$ small enough
we gain \eqref{W3lin_lower}.
\end{proof}
\begin{proposition} \label{prop:W3linb} 
\noindent  Let $T>0$, $b>0$, $\alpha \geq 0$, $\delta \in (0,1)$ and assume that
\begin{itemize}
\item $a(t,x) \geq \underline{a}>0$,
\item $a \in L^{\infty}(0,T;L^{\infty}(\Om))$, $a_t \in L^2(0,T;L^2(\Om))$, 
$\nabla a \in L^2(0,T;L^4(\Om))$,
\item
 $g \in L^{\infty}(0,T;W^{-\frac{q}{q+1},\frac{q+1}{q}}(\Gamma))$,  $g_t \in L^{\frac{q+1}{q}}(0,T;W^{-\frac{q}{q+1},\frac{q+1}{q}}(\Gamma))$,
\item 
$ u_0 \in W^{1,4}(\Om), u_1 \in W^{1,q+1}(\Om)$,
\item  $q \geq 3$,
\end{itemize}
with 
\begin{equation} \label{W3_a2}
\begin{split}
\tilde{a}:=& \ \frac{\underline{a}}{4}-\frac{1}{2}\|\nabla a\|_{L^2(0,T;L^4(\Om))} >0, \\
 \tilde{b}:=& \ \frac{1}{4}-\|\nabla a\|_{L^2(0,T;L^4(\Omega))}T(C^{\Om}_{L^{q+1},L^4} C_2^{\Om})^2>0,
 \end{split}
\end{equation} 
\noindent then \eqref{W3lin} has a weak solution 
\begin{align} \label{W3_X}
u\in X:=~& C^1(0,T;W^{1,q+1}(\Om)) \cap H^2(0,T;L^2(\Om))\},
\end{align}
\noindent which satisfies the energy estimate
\begin{align}\label{W3lin_higher} 
&\mu\Bigl[\frac{1}{2}-\tau\|\nabla a\|_{L^2(0,T;L^4(\Om))}\Bigr]\|u_{tt}\|_{L^2(0,T;L^2(\Om))}^2  \nn\\
& +\mu \frac{b(1-\delta)}{8} \|\nabla u_t\|_{L^{\infty}(0,T;L^2(\Omega))}^2+ \Bigl[\tilde{b}- \epsilon_0(\mu+1)\Bigr]\|u_t\|^2_{L^{\infty}(0,T;L^2(\Om)}  \nn\\
&+\Bigl[\frac{b(1-\delta)}{2} -\mu \|a\|_{L^{\infty}(0,T;L^{\infty}(\Om))}\Bigr]\|\nabla u_t\|_{L^2(0,T;L^2(\Omega))}^2  \nn
 \\
& +\Bigl[\mu \frac{b\delta}{2(q+1)}-\eta(2\mu+1) \Bigr]\|\nabla u_t\|_{L^{\infty}(0,T;L^{q+1}(\Om))}^{\qq+1}+\mu \frac{\alpha}{4} \|u_t\|^2_{L^{\infty}(0,T;L^2(\hG))}  \nn\\
&+ \Bigl[\tilde{a}-\mu\frac{2}{b(1-\delta)}\|a\|^2_{L^{\infty}(0,T;L^{\infty}(\Om))}\Bigr] \|\nabla u\|^2_{L^{\infty}(0,T;L^2(\Om))} \nn  \\
& +\Bigl[\frac{b\delta}{2}-\epsilon_1(\mu+1)\Bigr]\|\nabla u_t\|_{L^{q+1}(0,T;L^{q+1}(\Om))}^{q+1}+\frac{\alpha}{2}\|u_t\|^2_{L^2(0,T;L^2(\hG))} \\
\leq& \  
\overline{C} \Bigl((T\|\nabla a\|_{L^2(0,T;L^4(\Omega))})^{\frac{\qq+1}{\qq-1}}+\|a\|_{L^{\infty}(0,T;L^{\infty}(\Omega))}|\nabla u_{0}|_{L^2(\Omega)}^2 \nn  \\
&   +(\|a_t\|_{L^{4/3}(0,T;L^2(\Omega))}
+\|\nabla a\|_{L^2(0,T;L^4(\Omega))})  |\nabla u_{0}|_{L^4(\Omega)}^2 \nn \\
&+\|a\|_{L^{\infty}(0,T;L^{\infty}(\Omega))} 
\Bigl(|\nabla u_{0}|_{L^2(\Omega)}^2+|\nabla u_{1}|_{L^2(\Omega)} |\nabla u_{0}|_{L^2(\Omega)}\Bigr) \nn \\
&+((\frac{1}{2}+T^{3/4})\sqrt{T}\|a_t\|_{L^{4/3}(0,T;L^2(\Omega))})^{\frac{\qq+1}{\qq-1}}+(T^{2}\|\nabla a\|_{L^2(0,T;L^4(\Omega))})^{\frac{\qq+1}{\qq-1}}\nn \\
& + (T^{5/2}\|a_t\|_ {L^2(0,T;L^2(\Om))})^{\frac{\qq+1}{\qq-1}}+\|a_t\|_{L^2(0,T;L^2(\Om))}\sqrt{T}|\nabla u_0|^2_{L^4(\Om)}
 \nn \\
&+|u_1|^2_{H^1(\Om)} 
 +|u_1|^{q+1}_{W^{1,q+1}(\Om)} 
 +|u_1|^2_{L^2(\hG)}+ C_{\Gamma}(g)\Bigr),\nn
\end{align}
for some sufficiently small constants $\epsilon_0, \epsilon_1, \eta, \mu, \tau>0$, some large enough $\overline{C}>0$ and $C_{\Gamma}(g)$ defined as in \eqref{Cg}.
\end{proposition}
\begin{proof}
In order to obtain higher order estimate, we  will multiply \eqref{W3lin} first by $u_t$, proceeding differently than in Proposition \ref{prop:W3lina}, and then by $u_{tt}$ and combine the two obtained estimates. Multiplication by $u_t$ and integration with respect to space and time produces
\begin{align} \label{W3lin_est3}
&\frac{1}{2}\left[|u_t(s)|_{L^2(\Omega)}^2+|\sqrt{a}\nabla u|_{L^2(\Omega)}^2\right]_0^t +\int_0^t \Bigl(b(1-\delta)|\nabla u_t(s)|_{L^2(\Omega)}^2 \nn \\
&+b\delta|\nabla u_t(s)|_{L^{q+1}(\Omega)}^{\qq+1}\Bigr)\, ds  
  + \alpha \int_0^t  |u_t(s)|^2_{L^2(\hG)}  \, ds \nn \\
=& \ \int_0^t\int_\Omega\Bigl( \frac12 a_t |\nabla u|^2 -u_t\nabla a \cdot \nabla u 
\Bigr) \, dx \, ds+\int_0^t \int_{\Gamma} gu_t \, dx \, ds  \nn \\
\leq& \ \int_0^t\int_\Omega  \frac12 a_t |\nabla u|^2  \, dx \, ds
+\|\nabla a\|_{L^2(0,T;L^4(\Omega))}\Bigl(\frac{T}{2} \|u_t\|_{L^{\infty} (0,T;L^4(\Omega))}^2  \\
&+\frac12 \|\nabla u\|_{L^{\infty}(0,T;L^2(\Omega))}^2\Bigr) 
+\epsilon_1 \|\nabla u_{t}\|^{q+1}_{L^{q+1}(0,T;L^{q+1}(\Om))} \nn \\
&+C(\epsilon_1,q+1)(C_1^{tr}(1+C_P))^{\frac{q+1}{q}}\|g\|^{\frac{q+1}{q}}_{L^{\frac{q+1}{q}}(0,T;W^{-\frac{q}{q+1},\frac{q+1}{q}}(\Gamma))} \nn \\
& +\epsilon_0 \|u_{t}\|^2_{L^{\infty}(0,T;L^2(\Om)}+\frac{1}{4\epsilon_0}(C_1^{tr} C^{\Om}_2)^2\|g\|^2_{L^1(0,T;W^{-\frac{q}{q+1},\frac{q+1}{q}}(\Gamma))},\nn
\end{align}
for some $\epsilon_0$, $\epsilon_1>0$. We will make use of the embedding $L^{q+1}(\Om) \hookrightarrow L^4(\Om)$ and Young's inequality \eqref{Young} to estimate
\begin{equation}   \label{W3lin_a}
\begin{split}
& \frac{1}{2}\int_0^t\int_\Omega a_t |\nabla u|^2 \, dx \, ds \\
 \leq& \ \frac{1}{2} \|\nabla u\|^2 _{L^{\infty}(0,T;L^4(\Om))} \int_0^t |a_t|_{L^2(\Om)} \, ds    \\
\leq& \  \frac{1}{2} \|a_t\|_{L^2(0,T;L^2(\Om))} \sqrt{T} \|\nabla u\|^2 _{L^{\infty}(0,T;L^4(\Om))}   \\
\leq& \ \|a_t\|_{L^2(0,T;L^2(\Om))} \sqrt{T} \Bigl[  T^2 \|\nabla u_t\|^2_{L^{\infty}(0,T;L^4(\Om))}+|\nabla u_0|^2_{L^4(\Om)} \Bigr]    \\
\leq& \ \|a_t\|_{L^2(0,T;L^2(\Om))} \sqrt{T} \Bigl[  T^2 (C^{\Om}_{L^{q+1},L^4})^2 \|\nabla u_t\|^2_{L^{\infty}(0,T;L^{q+1}(\Om))}+|\nabla u_0|^2_{L^4(\Om)} \Bigr]   \\
\leq& \ \frac{\eta}{2} \|\nabla u_t\|^{q+1}_{L^{\infty}(0,T;L^{q+1}(\Om))} + \|a_t\|_{L^2(0,T;L^2(\Om))} \sqrt{T} |\nabla u_0|^2_{L^4(\Om)} \\
&  +C(\tfrac{\eta}{2}, \tfrac{q+1}{2})(\|a_t\|_{L^2(0,T;L^2(\Om))}T^{5/2} (C^{\Om}_{L^{q+1},L^4})^2 )^{\frac{q+1}{q-1}} ,
\end{split}
\end{equation}
for some $\eta>0$ and $q>1$. We can also obtain
\begin{align*}
& \|\nabla a\|_{L^2(0,T;L^4(\Omega))}\frac{T}{2}\|u_t\|^2_{L^{\infty}(0,T;L^4(\Om))} \nn \\
\leq& 
 \  \|\nabla a\|_{L^2(0,T;L^4(\Omega))}\frac{T}{2} (C^{\Om}_{L^{q+1},L^4})^2\|u_t\|^{2}_{L^{\infty}(0,T;L^{q+1}(\Om))} \nn \\
\leq& \ \|\nabla a\|_{L^2(0,T;L^4(\Omega))}T (C^{\Om}_{L^{q+1},L^4})^2\Bigl[C_P^2\|\nabla u_t\|^{2}_{L^{\infty}(0,T;L^{q+1}(\Om))}\nn \\
&+(C_2^{\Om})^2\|u_t\|^{2}_{L^{\infty}(0,T;L^2(\Om))}\Bigr] \\
 \leq& \ \frac{\eta}{2} \|\nabla u_t\|^{q+1}_{L^{\infty}(0,T;L^{q+1}(\Om))}+C(\tfrac{\eta}{2}, \tfrac{q+1}{2})((C_PC^{\Om}_{L^{q+1},L^4})^2\|\nabla a\|_{L^2(0,T;L^4(\Omega))}T)^{\frac{q+1}{q-1}}\\
&+\|\nabla a\|_{L^2(0,T;L^4(\Omega))}T(C^{\Om}_{L^{q+1},L^4}C_2^{\Om})^2\|u_t\|^{2}_{L^{\infty}(0,T;L^2(\Om))},
\end{align*}
which together with  \eqref{W3lin_a} results in the following estimate:
\begin{align} \label{W3lin_est4_0}
&(\tilde{b}-\epsilon_0)\|u_t\|_{L^{\infty}(0,T;L^2(\Omega))}^2  +\tilde{a}\|\nabla u\|^2_{L^{\infty}(0,T;L^2(\Om))} +\frac{\alpha}{2} \|u_t\|^2_{L^2(0,T;L^2(\hG))} \nn \\
& +\frac{b(1-\delta)}{2}\|\nabla u_t\|_{L^2(0,T;L^2(\Omega))}^2+(\frac{b\delta}{2}-\epsilon_1)\|\nabla u_t\|_{L^{q+1}(0,T;L^{q+1}(\Omega))}^{\qq+1} \nn
\\
\leq& \ \eta \|\nabla u_t\|^{q+1}_{L^{\infty}(0,T;L^{q+1}(\Om))}+ \|a_t\|_{L^2(0,T;L^2(\Om))} \sqrt{T} |\nabla u_0|^2_{L^4(\Om)} \nn \\
&+C(\tfrac{\eta}{2}, \tfrac{q+1}{2})((C_PC^{\Om}_{L^{q+1},L^4})^2\|\nabla a\|_{L^2(0,T;L^4(\Omega))}T)^{\frac{q+1}{q-1}} \\
& +C(\tfrac{\eta}{2}, \tfrac{q+1}{2})(\|a_t\|_{L^2(0,T;L^2(\Om))} T^{5/2} (C^{\Om}_{L^{q+1},L^4})^2 )^{\frac{q+1}{q-1}} \nn \\
& +\frac12 \|a\|_{L^{\infty}(0,T;L^{\infty}(\Omega))}|\nabla u_{0}|_{L^2(\Omega)}^2 
+\frac{1}{4 \epsilon_0} \|g\|^2_{L^1(0,T;W^{-\frac{q}{q+1},\frac{q+1}{q}}(\Gamma))}\nn \\
&+\frac12|u_{1}|_{L^2(\Omega)}^2+ C(\epsilon_1,q+1)(C_1^{tr}(1+C_P)\|g\|_{L^{\frac{q+1}{q}}(0,T;W^{-\frac{q}{q+1},\frac{q+1}{q}}(\Gamma))})^{\frac{q+1}{q}}. \nn
\end{align}
Testing with $u_{tt}$ yields
\begin{align*}
&\int_0^t |u_{tt}(s)|_{L^2(\Omega)}^2\, ds
+\left[\frac{b(1-\delta)}{2}|\nabla u_t|_{L^2(\Om)}^2 
+\frac{b\delta}{\qq+1}|\nabla u_t|_{L^{q+1}(\Omega)}^{\qq+1}+\frac{\alpha}{2}|u_t|^2_{L^2(\hG)}
\right]_0^t
\\
=& \ \int_0^t\int_\Omega\left( -a \nabla u_{tt} \cdot \nabla u -u_{tt} \nabla a \cdot \nabla u
\right) \, dx \, ds+\int_0^t \int_{\Gamma} g u_{tt} \, dx \, ds\\
=& \ \int_0^t\int_\Omega\left( a_t \nabla u_t \cdot \nabla u +a|\nabla u_t|^2
-u_{tt}\nabla a \cdot \nabla u \right)dx\, ds-\left[\int_\Omega  a \nabla u_t \cdot \nabla u \,dx\right]_0^t \\
&+\int_0^t \int_{\Gamma} g u_{tt} \, dx \, ds\\
\leq& \
\int_0^t\left( |a_t(s)|_{L^2(\Omega)} |\nabla u_t(s)|_{L^4(\Omega)}
+|\nabla a(s)|_{L^4(\Omega)}|u_{tt}(s)|_{L^2(\Omega)} \right) \\
& \cdot \Bigl[|\nabla u_{0}|_{L^4(\Omega)}+\sqrt[4]{s^3\int_0^s|\nabla u_{t}(\sigma)|_{L^4(\Omega)}^{4}\,d\sigma}\Bigr] \, ds
\\
&+\|a\|_{L^{\infty}(0,T;L^{\infty}(\Omega))} \Bigl(|\nabla u_t(t)|_{L^2(\Omega)}
|\nabla u(t)|_{L^2(\Omega)}+|\nabla u_{1}|_{L^2(\Omega)} |\nabla u_{0}|_{L^2(\Omega)}\nn \\
& +\|\nabla u_{t}\|_{L^2(0,T;L^2(\Omega))}^2\Bigr) 
+\int_0^t \int_{\Gamma} g u_{tt} \, dx \, ds\\
\leq& \
\|a_t\|_{L^{4/3}(0,T;L^2(\Omega))} \left((\frac12+T^{\frac{3}{4}})\|\nabla u_t\|_{L^4(0,T;L^4(\Omega))}^2+\frac12|\nabla u_{0}|_{L^4(\Omega)}^2\right) \\
&+\|\nabla a\|_ {L^2(0,T;L^4(\Om))} ( \tau \|u_{tt}\|_{L^2(0,T;L^2(\Omega))}^2
+\frac{1}{2 \tau} |\nabla u_{0}|_{L^4(\Omega)}^2 \nn \\
&+\frac{1}{2 \tau} T^{\frac{3}{2}} \|\nabla u_t\|_{L^4(0,T;L^4(\Omega))}^2 )
+\|a\|_{L^{\infty}(0,T;L^{\infty}(\Omega))} \Bigl(|\nabla u_{1}|_{L^2(\Omega)} |\nabla u_{0}|_{L^2(\Omega)} \nn \\
&+\|\nabla u_{t}\|_{L^2(0,T;L^2(\Omega))}^2\Bigr)
+\frac{2}{b(1-\delta)}\|a\|_{L^{\infty}(0,T;L^{\infty}(\Omega))}^2|\nabla u(t)|_{L^2(\Omega)}^2 \nn \\
&+\frac{b(1-\delta)}{8}|\nabla u_{t}(t)|_{L^2(\Omega)}^2+\int_0^t \int_{\Gamma} g u_{tt} \, dx \, ds,
\end{align*}
for some $\tau>0$. We can make use of Young's inequality and the inequality \eqref{W1_ineq2} for the boundary integral together with the inequality $\|\nabla u_t\|_{L^4(0,T;L^4(\Om))} \leq T^{\frac{1}{4}}\|\nabla u_t\|_{L^{\infty}(0,T;L^4(\Om))} \leq T^{\frac{1}{4}}C^{\Om}_{L^{q+1},L^4}\|\nabla u_t\|_{L^{\infty}(0,T;L^{q+1}(\Om))}$ to obtain
\begin{align*}
&\int_0^t |u_{tt}(s)|_{L^2(\Omega)}^2\, ds
+\left[\frac{b(1-\delta)}{2}|\nabla u_t|_{L^2(\Om)}^2 
+\frac{b\delta}{\qq+1}|\nabla u_t|_{L^{q+1}(\Omega)}^{\qq+1}+\frac{\alpha}{2}|u_t|^2_{L^2(\hG)}
\right]_0^t
\\
\leq& \
\|a_t\|_{L^{4/3}(0,T;L^2(\Omega))}\frac12|\nabla u_{0}|_{L^4(\Omega)}^2+\frac{\eta}{2}\|\nabla u_t\|_{L^{\infty}(0,T;L^{q+1}(\Omega))}^{\qq+1} \nn \\
& +C(\tfrac{\eta}{2},\tfrac{q+1}{2})((\frac12+T^{\frac{3}{4}})\sqrt{T}(C^{\Om}_{L^{q+1},L^4})^2\|a_t\|_{L^{4/3}(0,T;L^2(\Omega))})^{\frac{\qq+1}{\qq-1}} \\
& +\|\nabla a\|_{L^2(0,T;L^4(\Omega))} \Bigl(\tau\|u_{tt}\|_{L^2(0,T;L^2(\Omega))}^2+\frac{1}{2\epsilon} |\nabla u_{0}|_{L^4(\Omega)}^2\Bigr) \\
&+\frac{\eta}{2}\|\nabla u_{t}\|_{L^{\infty}(0,T;L^{q+1}(\Omega))}^{\qq+1}+C(\tfrac{\eta}{2},\tfrac{q+1}{2})(\frac{1}{2 \tau }T^{2}(C^{\Om}_{L^{q+1},L^4})^2\|\nabla a\|_{L^2(0,T;L^4(\Omega))})^{\frac{\qq+1}{\qq-1}}\\
&+\|a\|_{L^{\infty}(0,T;L^{\infty}(\Omega))} \Bigl(|\nabla u_{1}|_{L^2(\Omega)} |\nabla u_{0}|_{L^2(\Omega)}
+\|\nabla u_{t}\|_{L^2(0,T;L^2(\Omega))}^2\Bigr) \\
&+\frac{2}{b(1-\delta)}\|a\|_{L^{\infty}(0,T;L^{\infty}(\Omega))}^2|\nabla u(t)|_{L^2(\Omega)}^2 +\frac{b(1-\delta)}{8}|\nabla u_{t}(t)|_{L^2(\Omega)}^2 \nn \\
&+\eta \|\nabla u_{t}\|^{q+1}_{L^{\infty}(0,T;L^{q+1}(\Om))}+C(\eta, q+1)(C_1^{tr}(1+C_P))^{\frac{q+1}{q}}\|g\|^{\frac{q+1}{q}}_{L^{\infty}(0,T;W^{-\frac{q}{q+1},\frac{q+1}{q}}(\Gamma))} \nn \\
&+ \epsilon_0 \|u_{t}\|^{2}_{L^{\infty}(0,T;L^2(\Om))}+\frac{1}{2\epsilon_0}(C_1^{tr}C_2^{\Om})^2\|g\|^2_{L^{\infty}(0,T;W^{-\frac{q}{q+1},\frac{q+1}{q}}(\Gamma))}\nn \\
& +|u_{1}|^{q+1}_{W^{1,q+1}(\Omega)} +C(1,q+1)(C_1^{tr}|g(0)|_{W^{-\frac{q}{q+1},\frac{q+1}{q}}(\Gamma))})^{\frac{q+1}{q}}\nn \\
&  +C(\epsilon_1,q+1)(C_1^{tr}(1+C_P))^{\frac{q+1}{q}}\|g_t\|^{\frac{q+1}{q}}_{L^{\frac{q+1}{q}}(0,T;W^{-\frac{q}{q+1},\frac{q+1}{q}}(\Gamma))} \nn \\
& +\epsilon_1 \|\nabla u_{t}\|^{q+1}_{L^{q+1}(0,T;L^{q+1}(\Om))}+\frac{1}{2\epsilon_0}(C_1^{tr} C^{\Om}_2)^2\|g_t\|^2_{L^1(0,T;W^{-\frac{q}{q+1},\frac{q+1}{q}}(\Gamma))},
\end{align*}
which, by taking essential supremum with respect to $t$ and then adding $\mu$ times obtained inequality to \eqref{W3lin_est4_0}  results in the higher order estimate \eqref{W3lin_higher}.
\end{proof}
Let us now consider the problem with the added lower order nonlinear damping term
\begin{equation}\label{W3lin_gamma}
\begin{cases}
u_{tt}-a\Delta u
-b\,\text{div}\Bigl(((1-\delta) +\delta|\nabla u_t|^{\qq-1})\nabla u_t\Bigr)
\vspace{2mm} \\
+\gamma|u_t|^{\qq-1}u_t
=0 \, \text{ in } \Omega \times (0,T],
\vspace{2mm} \\
a \frac{\partial u}{\partial n}+b((1-\delta)+\delta|\nabla u_t|^{q-1})\frac{\partial u_t}{\partial n}=g  \ \ \text{on} \, \Gamma \times (0,T],
\vspace{2mm} \\
\alpha u_t+a \frac{\partial u}{\partial n}+b((1-\delta)+\delta|\nabla u_t|^{q-1})\frac{\partial u_t}{\partial n}=0  \ \ \text{on} \, \hat{\Gamma} \times (0,T],
\vspace{2mm}
\\
(u,u_t)=(u_0, u_1) \ \ \text{on} \ \overline{\Om}\times \{t=0\},
\end{cases}
\end{equation}
where $\gamma>0$. This is a linearized version of \eqref{W3_gamma}, where nonlinearity appears only through the damping terms. We can utilize the embedding $W^{1,q+1}(\Om) \hookrightarrow L^{\infty}(\Om)$, Young's inequality in the form \eqref{Young} and estimate the boundary integral by employing \eqref{W1_boundary_gamma1}, to obtain
\begin{align} \label{W3_est1}
&\frac{1}{2} \left[|u_{t}(s)|_{L^2(\Om)}^2\right]_0^t
+\int_0^t \Bigl(b(1-\delta)|\nabla u_t(s)|_{L^2(\Om)}^2 
+b\delta|\nabla u_t(s)|_{L^{q+1}(\Om)}^{\qq+1} \nn \\ 
& +\gamma |u_t(s)|^{q+1}_{L^{q+1}(\Om)}  + \alpha |u_t(s)|^2_{L^2(\hG)} \Bigr)\, ds  \nn \\
\leq& \ \|a\|_{L^2(0,T;L^{\infty}(\Om))} \Bigl( \sqrt{T}\|\nabla u_t\|_{\LT}^2 
+\frac12 \|\nabla u_t\|_{\LT}^2 \\
&+\frac12 |\nabla u_{0}|_{L^2(\Om)}^2\Bigr) 
+\epsilon_0 \int_0^T | u_t(s)|^{q+1}_{W^{1,q+1}(\Om)}\, ds \nn \\ &+C(\tfrac{\epsilon_0}{2}, \tfrac{q+1}{2})T((C^{\Om}_{W^{1,q+1},L^{\infty}})^2\|\nabla a\|_{L^2(0,T;L^2(\Omega))})^{\frac{q+1}{q-1}} \nn  \\
&+\|\nabla a\|_{L^2(0,T;L^2(\Omega))}
\Bigl(\frac{1}{2} T\|\nabla u_t\|_{\LT}^2 +\frac{1}{2} |\nabla u_{0}|_{L^2(\Om)}^2\Bigr)    \nn \\
& +C(\tfrac{\epsilon_0}{2}, \tfrac{q+1}{2})(C_1^{tr} \|g\|_{L^{\frac{q+1}{q}}(0,T;W^{-\frac{q}{q+1},\frac{q+1}{q}}(\Gamma))})^{\frac{q+1}{q}}, \nn
\end{align}
for some $\epsilon_0 >0$ and $q>1$, $q>d-1$. By taking $ \displaystyle \esssup_{[0,T]}$ in \eqref{W3_est1} we get 
\begin{equation} \label{W3lin_lower_gamma}
\begin{split}
&\hat{b}\|\nabla u_t\|^2_{\LT} +\frac{1}{4}\|u_{t}\|^2_{\LiT} +\frac{\alpha}{2}\| u_t\|^2_{L^2(0,T;L^2(\hG))} 
 \\
&+(\frac{b\delta}{2}-\epsilon_0)\|\nabla u_t\|^{q+1}_{\LqT}  +(\frac{\gamma}{2}-\epsilon_0)\| u_t\|^{q+1}_{\LqT} \\
\leq& \ \Bigl(\|a\|_{L^2(0,T;L^{\infty}(\Omega))} +\|\nabla a\|_{L^2(0,T;L^2(\Omega))}\Bigr)\frac12 |\nabla u_{0}|_{L^2(\Om)}^2 +\frac{1}{2}|u_1|^2_{L^2(\Om)} \\
&+ C(\tfrac{\epsilon_0}{2}, \tfrac{q+1}{2})T((C^{\Om}_{W^{1,q+1},L^{\infty}})^2\|\nabla a\|_{\LiT})^{\frac{q+1}{q-1}}  \\
& +C(\tfrac{\epsilon_0}{2}, \tfrac{q+1}{2})(C_1^{tr} \|g\|_{L^{\frac{q+1}{q}}(0,T;W^{-\frac{q}{q+1},\frac{q+1}{q}}(\Gamma))})^{\frac{q+1}{q}} , 
\end{split}
\end{equation}
for some $ 0< \epsilon_0 < \frac{1}{2} \min \{b \delta, \gamma\}$ and $\hat{b}>0$ defined as in \eqref{assumption1}.\\
Note that the addition of the lower order damping term allows us to remove the second assumption in \eqref{assumption1} on smallness of $a$. \\
In the case of $q=1$ (and $d=1$), $u$ satisfies
\begin{equation} \label{W3lin_lower_gamma_q=1}
\begin{split}
&\Bigl(\hat{b}-\epsilon_0 \Bigr)\|\nabla u_t\|^2_{L^2(0,T;L^2(\Om))}+\frac{1}{4}\|u_t\|_{L^{\infty}(0,T;L^2(\Omega)}^2 \\
& +\Bigl(\tilde{b}-\epsilon_0\Bigr)\|u_t\|^2_{L^2(0,T;L^2(\Om))} +\frac{\alpha}{2}\|u_t\|^2_{L^2(0,T;L^2(\hG))}  \\
\leq& \ \Bigl(\|a\|_{L^2(0,T;L^{\infty}(\Omega))} + \|\nabla a\|_{L^2(0,T;L^2(\Omega))}\Bigr)\frac12 |\nabla u_{0}|_{L^2(\Om)}^2 +\frac{1}{2}|u_1|^2_{L^2(\Om)} \\
&+\frac{1}{4 \epsilon_0}(C_2^{tr})^2\|g\|^2_{L^2(0,T;H^{-1/2}(\Gamma))},
\end{split}
\end{equation}
where $\tilde{b}:= \ \frac{\gamma}{2}-(C^{\Om}_{H^1,L^{\infty}})^2\|\nabla a\|_{L^2(0,T;L^2(\Omega))}>0$, and $\hat{b}$ is defined as in \eqref{cond_b}.\\
To obtain higher order estimate, we test the problem again by $u_t$ and integrate with respect to space and time to obtain
\begin{align} \label{W3lin_est3}
&\frac12\left[|u_t|_{L^2(\Omega)}^2+|\sqrt{a}\nabla u|_{L^2(\Omega)}^2\right]_0^t +\int_0^t \Bigl(b(1-\delta)|\nabla u_t|_{L^2(\Omega)}^2 
+b\delta|\nabla u_t|_{L^{q+1}(\Omega)}^{\qq+1} \nn \\
&+\gamma|u_t|_{L^{q+1}(\Omega)}^{q+1}\Bigr)\, ds 
 + \alpha \int_0^t \int_{\hG} |u_t|^2_{L^2(\hG)} \, dx \, ds \nn \\
=& \ \int_0^t\int_\Omega\Bigl( \frac12 a_t |\nabla u|^2 -u_t\nabla a \cdot \nabla u 
\Bigr) \, dx \, ds+\int_0^t \int_{\Gamma} gu_t \, dx \, ds  \\
&\leq \int_0^t\int_\Omega  \frac12 a_t |\nabla u|^2  \, dx \, ds
+\|\nabla a\|_{L^2(0,T;L^4(\Omega))}\Bigl(\frac{T}{2} \|u_t\|_{L^{\infty} (0,T;L^4(\Omega))}^2  \nn \\
&+\frac12 \|\nabla u\|_{L^{\infty}(0,T;L^2(\Omega))}^2\Bigr) 
+\epsilon_1 \| u_t\|^{q+1}_{L^{q+1}(0,T;W^{1,q+1}(\Om))}\, ds  \nn \\
&\quad+C(\epsilon_1, \tfrac{q+1}{2})(C_1^{tr} \|g\|_{L^{\frac{q+1}{q}}(0,T;W^{-\frac{q}{q+1},\frac{q+1}{q}}(\Gamma))})^{\frac{q+1}{q}}. \nn
\end{align}
Taking $\displaystyle \esssup_{[0,T]}$ in \eqref{W3lin_est3} and making use of \eqref{W3lin_a} and
\begin{align*}
& \|\nabla a\|_{L^2(0,T;L^4(\Omega))}\frac{T}{2}\|u_t\|^2_{L^{\infty}(0,T;L^4(\Om))} \nn \\
\leq& \ \eta \|u_t\|^{q+1}_{L^{\infty}(0,T;L^{q+1}(\Om))} +C(\eta, \tfrac{q+1}{2})((C^{\Om}_{L^{q+1},L^4})^2\|\nabla a\|_{L^2(0,T;L^4(\Omega))}T)^{\frac{q+1}{q-1}},
\end{align*}
leads to the estimate
\begin{align} \label{W3lin_est4}
&\frac{1}{4}\|u_t\|_{L^{\infty}(0,T;L^2(\Omega))}^2+\Bigl(\frac{\ul{a}}{4}-\frac{1}{2}\|\nabla a\|_{L^2(0,T;L^4(\Om))})\Bigr)\|\nabla u\|^2_{L^{\infty}(0,T;L^2(\Om))} \nn  \\
& +\frac{b(1-\delta)}{2}\|\nabla u_t\|_{L^2(0,T;L^2(\Omega))}^2 
+(\frac{b\delta}{2}-\epsilon_1)\|\nabla u_t\|_{L^{q+1}(0,T;L^{q+1}(\Omega))}^{\qq+1} \nn \\
&+(\frac{\gamma}{2}-\epsilon_1) \|u_t\|_{L^{q+1}(0,T;L^{q+1}(\Omega))}^{\qq+1} +\frac{\alpha}{2} \|u_t\|^2_{L^2(0,T;L^2(\hG))}
\nn \\
\leq& \ \eta \|u_t\|_{L^{\infty}(0,T;W^{1,q+1}(\Omega))}^{\qq+1}  +C(\eta, \tfrac{q+1}{2})(\|a_t\|_{L^2(0,T;L^2(\Om))} T^{5/2} (C^{\Om}_{L^{q+1},L^4})^2 )^{\frac{q+1}{q-1}}  \\
& + \|a_t\|_{L^2(0,T;L^2(\Om))} \sqrt{T} |\nabla u_0|^2_{L^4(\Om)}+\frac12 \|a\|_{L^{\infty}(0,T;L^{\infty}(\Omega))}|\nabla u_{0}|_{L^2(\Omega)}^2 \nn  \\
& +C(\eta,\tfrac{\qq+1}{2}) (T(C^{\Om}_{L^{q+1},L^4})^2\|\nabla a\|_{L^2(0,T;L^4(\Omega))})^{\frac{\qq+1}{\qq-1}}
+\frac12|u_{1}|_{L^2(\Omega)}^2 \nn \\
& +C(\tfrac{\epsilon_1}{2}, \tfrac{q+1}{2})(C_1^{tr} \|g\|_{L^{\frac{q+1}{q}}(0,T;W^{-\frac{q}{q+1},\frac{q+1}{q}}(\Gamma))})^{\frac{q+1}{q}}, \nn
\end{align}
for some $\eta >0$. \\
Testing with $u_{tt}$ and proceeding as in the case of $\gamma=0$, with the use of \eqref{W1_boundary_gamma2} for the estimation of the boundary integral, results in the higher order energy estimate
\begin{align} \label{W3lin_higher_gamma}
&\mu\Bigl(\frac{1}{2}-\tau \|\nabla a\|_{L^2(0,T;L^4(\Om))}\Bigr)\|u_{tt}\|_{L^2(0,T;L^2(\Om))}^2+ \frac{1}{4} \|u_t\|^2_{L^{\infty}(0,T;L^2(\Om)} \nn \\ &+\Bigl(\frac{b\delta}{2}-\epsilon_1(\mu+1)\Bigr)\|\nabla u_t\|_{L^{q+1}(0,T;L^{q+1}(\Om))}^{q+1} \nn \\
&+\Bigl(\frac{b(1-\delta)}{2} -\mu \|a\|_{L^{\infty}(0,T;L^{\infty}(\Om))}\Bigr)\|\nabla u_t\|_{L^2(0,T;L^2(\Omega))}^2   \nn \\
&  +\Bigl(\mu \frac{b\delta}{2(q+1)}-\mu(\eta+1)\Bigr)\|\nabla u_t\|_{L^{\infty}(0,T;L^{q+1}(\Om))}^{\qq+1}+\frac{\alpha}{2}\|u_t\|^2_{L^2(0,T;L^2(\hG))} \nn \\
&+ \Bigl(\tilde{a}-\mu\frac{2}{b(1-\delta)}\|a\|^2_{L^{\infty}(0,T;L^{\infty}(\Om))}\Bigr)\|\nabla u\|^2_{L^{\infty}(0,T;L^2(\Om))} \nn \\
& +(\frac{\gamma}{2}-\mu(\epsilon_1+1))\|u_t\|^{q+1}_{L^{q+1}(0,T;L^{q+1}(\Om)}  +\mu \frac{b(1-\delta)}{8} \|\nabla u_t\|_{L^{\infty}(0,T;L^2(\Omega))}^2  \\
&+\Bigl(\mu \frac{\gamma}{2(q+1)}-\eta(2\mu+1)\Bigr)\|u_t\|^{q+1}_{L^{\infty}(0,T;L^{q+1}(\Om))} +\mu \frac{\alpha}{2} \|u_t\|^2_{L^{\infty}(0,T;L^2(\hG))} \nn \\
\leq& \  
\overline{C} \Bigl( (T\|\nabla a\|_{L^2(0,T;L^4(\Omega))})^{\frac{\qq+1}{\qq-1}}+\|a\|_{L^{\infty}(0,T;L^{\infty}(\Omega))}|\nabla u_{0}|_{L^2(\Omega)}^2  \nn \\
&   +(\|a_t\|_{L^{4/3}(0,T;L^2(\Omega))}
+\|\nabla a\|_{L^2(0,T;L^4(\Omega))})  |\nabla u_{0}|_{L^4(\Omega)}^2 \nn \\
& +\|a\|_{L^{\infty}(0,T;L^{\infty}(\Omega))} 
\Bigl(|\nabla u_{0}|_{L^2(\Omega)}^2+|\nabla u_{1}|_{L^2(\Omega)} |\nabla u_{0}|_{L^2(\Omega)}\Bigr)\nn  \\
 &+((\frac{1}{2}+T^{3/4})\sqrt{T}\|a_t\|_{L^{4/3}(0,T;L^2(\Omega))})^{\frac{\qq+1}{\qq-1}}+|u_1|^2_{H^1(\Om)} +| u_1|^{q+1}_{W^{1,q+1}(\Om)}\nn \\
&+(T^{2}\|\nabla a\|_{L^2(0,T;L^4(\Omega))})^{\frac{\qq+1}{\qq-1}}
+ (T^{5/2}\|a_t\|_ {L^2(0,T;L^2(\Om))})^{\frac{\qq+1}{\qq-1}}  +|u_1|^2_{L^2(\hG)}\nn \\
&+\displaystyle \sum_{s=0}^1 \|\frac{d^s}{d t^s} g\|^{\frac{q+1}{q}}_{L^{\frac{q+1}{q}}(0,T;W^{-\frac{q}{q+1},\frac{q+1}{q}}(\Gamma))}+\| g\|^{\frac{q+1}{q}}_{L^{\infty}(0,T;W^{-\frac{q}{q+1},\frac{q+1}{q}}(\Gamma))}\Bigr)\nn
\end{align}
with $\tilde{a}$ defined in \eqref{W3_a2}, for some sufficiently small constants $\epsilon_1, \eta, \mu, \tau>0$
and some large enough $\overline{C}>0$. Note that here the second assumption in \eqref{W3_a2} on smallness of $a$ was not needed.
\begin{proposition} \label{prop:W3lin_gamma} 
Let $T>0$, $b>0$, $\alpha \geq 0$, $\delta \in (0,1)$, $\gamma>0$ and let the assumptions in Proposition \ref{prop:W3lina} hold with
\begin{align*}
\begin{cases}
\frac{b(1-\delta)}{2}-(\sqrt{T}+\frac{1}{2})\|a\|_{L^2(0,T;L^{\infty}(\Omega))}-\frac{T}{2}\|\nabla a\|_{L^2(0,T;L^2(\Omega))}>0, \ \text{for} \ q>1,\\
\hat{b}= \frac{b}{2}-(\frac{T}{2}+(C^{\Om}_{H^1,L^{\infty}})^2)\|\nabla a\|_{L^2(0,T;L^2(\Om))} -(\sqrt{T}+\frac{1}{2})\|a\|_{L^2(0,T;L^{\infty}(\Om))} >0,  \nn \\
\tilde{b}= \ \frac{\gamma}{2}-(C^{\Om}_{H^1,L^{\infty}})^2\|\nabla a\|_{L^2(0,T;L^2(\Omega))}>0, \ \text{for} \ q=1.
\end{cases}
\end{align*}
\noindent Then \eqref{W3lin_gamma} has a weak solution $u\in \tilde{X}$, with $\tilde{X}$ defined as in \eqref{W3_tildeX}, which satisfies the energy estimate \eqref{W3lin_lower_gamma} for $q>1$ and estimate \eqref{W3lin_lower_gamma_q=1} for $q=1$. \\
\noindent  If the assumptions in Proposition \ref{prop:W3linb} are satisfied with 
$$\tilde{a}=\frac{\underline{a}}{4}-\frac{1}{2}\|\nabla a\|_{L^2(0,T;L^4(\Om))} > 0,$$ then $u\in X$, with $X$ as in \eqref{W3_X}, \noindent and satisfies the energy estimate \eqref{W3lin_higher_gamma}.
\end{proposition}
\noindent We will now proceed to investigate existence of solutions for \eqref{W3}.
\begin{theorem}\label{thm:W3}
Let $c^2$, $b>0$, $\alpha \geq 0$, $\delta \in (0,1)$, $\tilde{k}\in\R$, $\qq\geq3$, $g \in L^{\infty}(0,T;W^{-\frac{q}{q+1},\frac{q+1}{q}}(\Gamma))$, $g_t \in L^{\frac{q+1}{q}}(0,T;W^{-\frac{q}{q+1},\frac{q+1}{q}}(\Gamma))$. There exist $\kappa>0$, $T>0$  such that for all  $u_0 \in W^{1,4}(\Om)$, $u_1 \in W^{1,q+1}(\Om)$, with
\begin{align*}
&C_{\Gamma}(g) +|\nabla u_0|^2_{L^2(\Om)} +|u_1|^2_{H^1(\Om)}+|u_1|^{q+1}_{W^{1,q+1}(\Om)} +|u_1|^2_{L^2(\hG)}
\leq \kappa^2
\end{align*}
there exists a weak solution $u \in \cW$ of \eqref{W3} where
\begin{equation}\label{defcW_W3}
\begin{split}
\cW =\{v\in X 
:& \ \|v_{tt}\|_{L^2(0,T;L^2(\Omega))}\leq \overline{m} \\
& \wedge \|v_t\|_{L^{\infty}(0,T;H^1(\Omega))}\leq \overline{m} \\
& \wedge \|\nabla v_t\|_{L^{\infty}(0,T;L^{q+1}(\Omega))}\leq \overline{M}\}\,,
\end{split}
\end{equation}
with $\overline{m}$ and $\overline{M}$ sufficiently small, and $C_{\Gamma}(g)$ is defined as in \eqref{Cg}.
\end{theorem}
\begin{proof}
We define an operator $\cT :\cW \to X$, $v\mapsto \cT v=\psii$ where $u$ solves \eqref{W3lin} with 
\begin{align}
a=\frac{c^2}{1-2\tilde{k}v_t}.
\end{align}
From \eqref{W3Poincare}, we obtain for $v \in \cW$ 
\begin{align*}
\|2\tilde{k} v_t\|_{L^{\infty}(0,T;L^{\infty}(\Om))}\leq 2 |\tilde{k}| C_{W^{1,q+1},L^{\infty}}^{\Om}\Bigl[(1+C_{P})\overline{M}+C_2^{\Om}\overline{m} \Bigr],
\end{align*}
and, assuming $2 |\tilde{k}| C_{W^{1,q+1},L^{\infty}}^{\Om}\Bigl[(1+C_{P})\overline{M}+C_2^{\Om}\overline{m} \Bigr]<1$, we can verify hypothesis of Proposition \ref{prop:W3linb}:
\begin{align*}
 a(t,x)\geq& \ \frac{c^2}{1+2|\tilde{k}|\|v_t\|_{L^{\infty}(0,T;L^{\infty}(\Omega))}} \\
  \geq& \ \frac{c^2}{1+2|\tilde{k}|C^{\Om}_{W^{1,\qq+1},L^{\infty}}((1+C_{P})\overline{M}+C_2^{\Om}\overline{m})} := \underline{a},
\\ 
\|a\|_{L^{\infty}(0,T;L^{\infty}(\Omega))} \leq& \ \frac{c^2}{1-2|\tilde{k}|\|v_t\|_{L^{\infty}(0,T;L^{\infty}(\Omega))}} \\
\leq & \ \frac{c^2}{1-2|\tilde{k}|C_{W^{1,\qq+1},L^{\infty}}^\Omega \Bigl[(1+C_{P})\overline{M}+C_2^{\Om}\overline{m} \Bigr]}, 
\\  
\|\nabla a\|_{L^2(0,T;L^4(\Omega))}=& \ \|\frac{2\tilde{k}c^2}{(1-2\tilde{k}v_t)^2}\nabla v_t\|_{L^2(0,T;L^4(\Omega))}\\
 \leq& \ \frac{2|\tilde{k}|c^2}{(1-2|\tilde{k}|C_{W^{1,\qq+1},L^{\infty}}^\Omega((1+C_{P})\overline{M}+C_2^{\Om}\overline{m}))^2}
 C_{L^{q+1},L^4}^\Omega \sqrt{T}\overline{M}, 
\\
\|a_t\|_{L^{4/3}(0,T;L^2(\Omega))}=& \ \|\frac{2\tilde{k}c^2}{(1-2\tilde{k}v_t)^2} v_{tt}\|_{L^{4/3}(0,T;L^2(\Omega))}\\ 
\leq& \ \frac{2|\tilde{k}|c^2}{(1-2|\tilde{k}|C_{W^{1,\qq+1},L^{\infty}}^\Omega ((1+C_{P})\overline{M}+C_2^{\Om}\overline{m}))^2}
\sqrt[4]{T} \overline{m}, \\
\|a_t\|_{L^2(0,T;L^2(\Omega))}  \leq& \ \frac{2|\tilde{k}|c^2}{(1-2|\tilde{k}|C_{W^{1,\qq+1},L^{\infty}}^\Omega ((1+C_{P})\overline{M}+C_2^{\Om}\overline{m}))^2} \overline{m}.
\end{align*}
It follows that assumptions are satisfied provided $\overline{m}$, $\overline{M}$, $\kappa$ and $T$ are sufficiently small such that
\begin{align*}
& 2|\tilde{k}|C_{W^{1,q+1},L^{\infty}}^{\Om} \Bigl[(1+C_{P})\overline{M}+C_2^{\Om}\overline{m} \Bigr]<1,  \\
& \frac{2|\tilde{k}|c^2}{(1-2|\tilde{k}|C_{W^{1,\qq+1},L^{\infty}}^\Omega((1+C_{P})\overline{M}+C_2^{\Om}\overline{m}))^2}
 C_{L^{q+1},L^4}^\Omega \sqrt{T}\overline{M} \\
\leq& \  \frac{c^2}{2(1+2|\tilde{k}|C^{\Om}_{W^{1,\qq+1},L^{\infty}}((1+C_{P})\overline{M}+C_2^{\Om}\overline{m}))}, \  \text{and} \\
&\frac{2|\tilde{k}|c^2 T^{3/2}\overline{M}(C^{\Om}_{L^{q+1},L^4})^3(C_2^{\Om})^2}{(1-2|\tilde{k}|C_{W^{1,\qq+1},L^{\infty}}^\Omega((1+C_{P})\overline{M}+C_2^{\Om}\overline{m}))^2}< \frac{1}{4}. 
\end{align*}
Therefore the energy estimate \eqref{W3lin_higher} is satisfied and we have
\begin{align*}
&\mu\Bigl[\frac{1}{2}-\tau\|\nabla a\|_{L^2(0,T;L^4(\Om))}\Bigr]\|u_{tt}\|_{L^2(0,T;L^2(\Om))}^2+\mu \frac{b(1-\delta)}{8} \|\nabla u_t\|_{L^{\infty}(0,T;L^2(\Omega))}^2 \\ &+\Bigl[\frac{b\delta}{2}-\epsilon_1(\mu+1)\Bigr]\|\nabla u_t\|_{L^{q+1}(0,T;L^{q+1}(\Om))}^{q+1} + \Bigl[\tilde{b}- \epsilon_0(\mu+1)\Bigr]\|u_t\|^2_{L^{\infty}(0,T;L^2(\Om)} \nn \\
&+\Bigl[\frac{b(1-\delta)}{2} -\mu \|a\|_{L^{\infty}(0,T;L^{\infty}(\Om))}\Bigr]\|\nabla u_t\|_{L^2(0,T;L^2(\Omega))}^2  
 \nn \\
 &+ \Bigl[\tilde{a}-\mu\frac{2}{b(1-\delta)}\|a\|^2_{L^{\infty}(0,T;L^{\infty}(\Om))}\Bigr] \|\nabla u\|^2_{L^{\infty}(0,T;L^2(\Om))}+\frac{\alpha}{2}\|u_t\|^2_{L^2(0,T;L^2(\hG))} \nn \\
&  +\Bigl[\mu \frac{b\delta}{2(q+1)}-\eta(2\mu+1) \Bigr]\|\nabla u_t\|_{L^{\infty}(0,T;L^{q+1}(\Om))}^{\qq+1}+\mu \frac{\alpha}{4} \|u_t\|^2_{L^{\infty}(0,T;L^2(\hG))} \nn \\
\leq& \  \tilde{C} \Bigl(
(T\sqrt{T}\overline{M})^{\frac{\qq+1}{\qq-1}}+ |\nabla u_{0}|^2_{L^2(\Omega)} +|u_{1}|_{H^1(\Omega)}^2 
+ (T^{5/2}\overline{M})^{\frac{\qq+1}{\qq-1}} \nn \\
&+(\sqrt[4]{T} \overline{m}+\sqrt{T}\overline{m}+\sqrt{T}\overline{M}) |\nabla u_{0}|_{L^4(\Omega)}^2 +| u_1|^{q+1}_{W^{1,q+1}(\Om)} +|u_1|^2_{L^2(\hG)}\nn \\
&+((\frac{1}{2}+T^{3/4})T^{3/4} \overline{m})^{\frac{\qq+1}{\qq-1}} +(T^{5/2}\overline{m})^{\frac{\qq+1}{\qq-1}}+C_{\Gamma}(g)
 \Bigr),
\end{align*}
 for some large enough $\tilde{C}$, and hence if $T$ and the bound $\kappa$ are sufficiently small, and we choose
$\overline{m}$ and $\overline{M}$ appropriately, $\cT$ is a self-mapping. Since $\cW$ is closed, we obtain existence of solutions through compactness argument.
\end{proof}
\noindent Relying on Proposition \ref{prop:W3lin_gamma}, we can obtain local existence of solutions for the problem \eqref{W3_gamma} with $\gamma>0$.
\begin{theorem}\label{thm:W3_gamma}
Let the assumptions of Theorem \ref{thm:W3} hold and $\gamma >0$. There exist $\kappa>0$, $T>0$  such that for all $u_0 \in W^{1,4}(\Om)$, $u_1 \in W^{1,q+1}(\Om)$, with
\begin{align*}
&\displaystyle \sum_{s=0}^1 \|\frac{d^s}{d t^s} g\|^{\frac{q+1}{q}}_{L^{\frac{q+1}{q}}(0,T;W^{-\frac{q}{q+1},\frac{q+1}{q}}(\Gamma))} +\|g\|^{\frac{q+1}{q}}_{L^{\infty}(0,T;W^{-\frac{q}{q+1},\frac{q+1}{q}}(\Gamma)))} \\
& +|\nabla u_0|^2_{L^2(\Om)}+|u_1|^2_{H^1(\Om)}  +|u_1|^{q+1}_{W^{1,q+1}(\Om)} +|u_1|^2_{L^2(\hG)}
\leq \kappa^2
\end{align*}
there exists a weak solution $u \in \cW$ of \eqref{W3_gamma}, where $\cW$ is defined as in \eqref{defcW_W3}, and $\overline{m}$ and $\overline{M}$ are sufficiently small.
\end{theorem}
\indent Due to the presence of $q-$Laplace damping term, the derivation of energy estimates is possible only for multipliers of lower order (see Remark 4, \cite{brunn}) and the question of uniqueness remains open. 

\bigskip
\noindent
{\bf Acknowledgments.} The author thanks Barbara Kaltenbacher for many fruitful discussions and comments. The financial support by the FWF (Austrian Science Fund) under grant P24970 is gratefully acknowledged.

\medskip
\end{document}